\documentclass[a4paper,11pt]{amsart}

\usepackage[colorlinks=true, linkcolor=red, linktoc=page]{hyperref}
\usepackage{amssymb}
\usepackage{enumitem}
\usepackage{url}
\usepackage{xcolor}
\usepackage{graphicx}
\usepackage{geometry}
\usepackage{subcaption}

\usepackage{psfrag}

\usepackage[T1]{fontenc}
\usepackage[english]{babel}
\usepackage{times}
\usepackage{ucs}
\usepackage[utf8x]{inputenc}

\usepackage{dsfont}
\usepackage[mathscr]{euscript}
\usepackage{stmaryrd}

\makeatletter

\def\section{\@startsection{section}{1}%
\z@{.7\linespacing\@plus\linespacing}{.5\linespacing}%
{\normalfont\bfseries\centering}}

\def\@settitle{\begin{center}%
  \baselineskip14\p@\relax
    \bfseries
    \LARGE\@title
  \end{center}%
}

\def\@setauthors{%
  \begingroup
  \trivlist
  \centering\footnotesize \@topsep30\p@\relax
  \advance\@topsep by -\baselineskip
  \item\relax
  \andify\authors
  \def\\{\protect\linebreak}%
 {\Large\authors}%
  \endtrivlist
  \endgroup
}

\def\maketitle{\par
  \@topnum\z@ 
  \@setcopyright
  \thispagestyle{firstpage}
  \ifx\@empty\shortauthors \let\shortauthors\shorttitle
  \else \andify\shortauthors
  \fi
  \@maketitle@hook
  \begingroup
  \@maketitle
  \toks@\@xp{\shortauthors}\@temptokena\@xp{\shorttitle}%
  \toks4{\def\\{ \ignorespaces}}
  \edef\@tempa{%
    \@nx\markboth{\the\toks4
      \@nx{\the\toks@}}{\the\@temptokena}}%
  \@tempa
  \endgroup
  \c@footnote\z@
  \def\do##1{\let##1\relax}%
  \do\maketitle \do\@maketitle \do\title \do\@xtitle \do\@title
  \do\author \do\@xauthor \do\address \do\@xaddress
  \do\email \do\@xemail \do\curraddr \do\@xcurraddr
  \do\commby \do\@commby
  \do\dedicatory \do\@dedicatory \do\thanks \do\thankses
  \do\keywords \do\@keywords \do\subjclass \do\@subjclass
}
\makeatother

\newtheorem{defi}{Definition}[section]
\newtheorem{lem}[defi]{Lemma}

\newtheorem{prop}[defi]{Proposition}
\newtheorem{theo}[defi]{Theorem}
\newtheorem{cor}[defi]{Corollary}
\newtheorem{rk}[defi]{Remark}

\newcommand{\Exp}{\mathbb{E}}

\renewcommand{\Pr}{\mathbb{P}}
\newcommand{\R}{\mathbb{R}}

\newcommand{\dd}{\mathrm{d}}

\newcommand{\ee}{\mathrm{e}}
\newcommand{\ind}[1]{\mathds{1}_{\{#1\}}}

\newcommand{\dto}{\downarrow}

\renewcommand{\hat}{\widehat}
\renewcommand{\tilde}{\widetilde}
\renewcommand{\bar}{\overline}

\renewcommand{\emptyset}{\varnothing}

\newcommand{\bfen}{-}
\newcommand{\bfex}{+}
\newcommand{\ret}{\mathbf{re}}
\newcommand{\exi}{\mathbf{ex}}
\newcommand{\bfren}{\ret}

\newcommand{\metasp}{\mathscr{O}}
 
\title{Estimation of statistics of transitions and Hill relation for Langevin dynamics}
\keywords{Langevin dynamics, transition path, invariant measure, Hill relation}
\subjclass{60J20, 60J70}

\author{Tony Lelièvre}
\address{{\bf Tony Lelièvre}\newline
{\rm \indent CERMICS, Ecole des Ponts, Matherials Joint Project-Team, Inria, Marne-la-Vallée, France}}
\email{\href{mailto:tony.lelievre@enpc.fr}{tony.lelievre@enpc.fr}}

\author{Mouad Ramil}
\address{{\bf Mouad Ramil}\newline
{\rm \indent Research Institute of Mathematics, Seoul National University, Seoul, Republic of Korea}}
\email{\href{mailto:ramilm@snu.ac.kr}{ramilm@snu.ac.kr}}

\author{Julien Reygner}
\address{{\bf Julien Reygner}\newline
{\rm \indent CERMICS, Ecole des Ponts, Marne-la-Vallée, France}}
\email{\href{mailto:julien.reygner@enpc.fr}{julien.reygner@enpc.fr}}

\begin{document}
   
\begin{abstract}
  In molecular dynamics, statistics of transitions, such as the mean transition time, are macroscopic observables which provide important dynamical information on the underlying microscopic stochastic process. A direct estimation using simulations of microscopic trajectories over long time scales is typically computationally intractable in metastable situations. To overcome this issue, several numerical methods rely on a potential-theoretic identity, sometimes attributed to Hill in the computational statistical physics litterature, which expresses statistics of transitions in terms of the invariant measure of the sequence of configurations by which the underlying process enters metastable sets. The use of this identity then allows to replace the long time simulation problem with a rare event sampling problem, for which efficient algorithms are available.
  
  In this article, we rigorously analyse such a method for molecular systems modelled by the Langevin dynamics. Our main contributions are twofold. First, we prove the Hill relation in the fairly general context of positive Harris recurrent chains, and show that this formula applies to the Langevin dynamics. Second, we provide an explicit expression of the invariant measure involved in the Hill relation, and describe an elementary exact simulation procedure. Overall, this yields a simple and complete numerical method to estimate statistics of transitions. 
\end{abstract}

\maketitle

\baselineskip=15pt

\section{Introduction}

\subsection{Metastability and mean transition time for the Langevin dynamics}

In computational statistical physics, the \emph{Langevin dynamics} is a stochastic process commonly used to simulate thermostated systems. It describes the evolution of the position-velocity pair $(q_t,p_t)$ in the phase space $\R^d \times \R^d$ according to the stochastic differential equation
\begin{equation}\label{eq:Langevin}
  \left\{\begin{aligned}
    \dd q_t &= p_t \dd t,\\
    \dd p_t &= F(q_t)\dd t - \gamma p_t \dd t + \sqrt{2\gamma\beta^{-1}}\dd W_t,
  \end{aligned}\right.
\end{equation}
where the vector field $F : \R^d \to \R^d$ is the \emph{force},
$\gamma>0$ is the \emph{friction} parameter, $\beta>0$ the
\emph{inverse temperature}, and the process $(W_t)_{t \geq 0}$ is a
$d$-dimensional Brownian motion~\cite{Tuc10,LeiMat15}. A particular
case of interest is the \emph{conservative case}, namely when $F =
-\nabla V$ for a function $V : \R^d \to \R$ which is called the
\emph{potential energy} of the system. In this case, upon
integrability assumption on $V$, the process $(q_t,p_t)_{t \geq 0}$ is
known to be ergodic with respect to the \emph{Boltzmann--Gibbs
  measure} with density 
\begin{equation}\label{eq:BG}
  \rho(q,p) = \frac{1}{Z_\beta} \ee^{-\beta H(q,p)}, \qquad H(q,p) := V(q) + \frac{|p|^2}{2}, \qquad Z_\beta := \int_{\R^d \times \R^d} \ee^{-\beta H(q,p)} \dd q \dd p.
\end{equation}
In a non-conservative case (or non-equilbrium case), namely when $F$ is
a non-conservative force, the invariant measure of the Langevin
dynamics is generally not explicit, see e.g.~\cite{MonRam} for details.

In most cases of interest, the process $(q_t,p_t)_{t \geq 0}$ is \emph{metastable}, which means that it spends most of its time in certain subsets of the phase space, called \emph{metastable sets}, and performs abrupt and seemingly unpredictable transitions between these sets. Metastable sets often represent macroscopic conformations of the system, on which \emph{statistics of transitions} provide important quantitative information. However, because of the scale separation between time steps employed to simulate the process $(q_t,p_t)_{t \geq 0}$ and typical times at which metastable transitions occur, the direct simulation of these rare events usually turns out to be impossible~\cite{AllValTen09,LelSto16}. 

Various methods have been introduced in the literature to numerically compute statistics of transitions. In order to provide an example of such a quantity, let us fix $A$ and $B$ two subsets of $\R^d$ with disjoint closures. We denote by $(\tau^\bfen_n)_{n \geq 0}$ the successive time indices at which the process $(q_t)_{t \geq 0}$ enters the set $A \cup B$. Under regularity assumptions on the boundary of $A$ and $B$ which will be made explicit below, the sequence $(Y^\bfen_n)_{n \geq 0}$, defined by
\begin{equation*}
  \forall n \geq 0, \qquad Y^\bfen_n := (q_{\tau^\bfen_n},p_{\tau^\bfen_n}),
\end{equation*}
is a time homogeneous Markov chain, which takes its values in the set
${\mathcal{A}^-} \cup {\mathcal{B}^-}$, where ${\mathcal{A}^-}$
(resp. ${\mathcal{B}^-}$) denotes the set of configurations $(q,p)$
such that $q \in \partial A$ (resp. $q \in \partial B$) and $p$ points
toward the interior of $A$ (resp. of $B$), see Figure~\ref{fig:section11}.

We now define the sequences $(\eta^\ret_{{\mathcal{A}^-},k})_{k \geq 0}$ and $(\eta^\ret_{{\mathcal{B}^-},k})_{k \geq 0}$ by
\begin{equation}\label{eq:etare}
  \begin{array}{ll}
    \eta^\ret_{{\mathcal{A}^-},0} := \min\{n \geq 0: Y^\bfen_n \in {\mathcal{A}^-}\}, \quad &\eta^\ret_{{\mathcal{B}^-},0} := \min\{n \geq \eta^\ret_{{\mathcal{A}^-},0}: Y^\bfen_n \in {\mathcal{B}^-}\},\\
    \eta^\ret_{{\mathcal{A}^-},k+1} := \min\{n \geq \eta^\ret_{{\mathcal{B}^-},k}: Y^\bfen_n \in {\mathcal{A}^-}\}, \quad & \eta^\ret_{{\mathcal{B}^-},k+1} := \min\{n \geq \eta^\ret_{{\mathcal{A}^-},k+1}: Y^\bfen_n \in {\mathcal{B}^-}\}, \quad k \geq 0.
  \end{array}
\end{equation}
They respectively refer to the successive return times of the chain $(Y^\bfen_n)_{n \geq 0}$ in ${\mathcal{A}^-}$ after a visit in ${\mathcal{B}^-}$, and conversely, with the convention that these times are counted from the first visit of the chain in ${\mathcal{A}^-}$. 

At the continuous-time level, we introduce the notation
\begin{equation*}
  \tau^\ret_{{\mathcal{A}^-},k} := \tau^\bfen_{\eta^\ret_{{\mathcal{A}^-},k}}, \qquad \tau^\ret_{{\mathcal{B}^-},k} := \tau^\bfen_{\eta^\ret_{{\mathcal{B}^-},k}},
\end{equation*}
and call the trajectory of $(q_t,p_t)$ on
$[\tau^\ret_{{\mathcal{A}^-},k}, \tau^\ret_{{\mathcal{B}^-},k}]$
the \emph{$k$-th transition path} between $A$ and $B$, see Figure~\ref{fig:section11}. Its (time) length is denoted by 
\begin{equation*}
  \Delta \tau^\bfren_{AB,k} := \tau^\ret_{{\mathcal{B}^-},k} - \tau^\ret_{{\mathcal{A}^-},k},
\end{equation*}
and the \emph{mean transition time} between $A$ and $B$ is then defined by
\begin{equation}\label{eq:TAB}
  T_{AB} := \lim_{\ell \to +\infty} \frac{1}{\ell}\sum_{k=0}^{\ell-1} \Delta \tau^\bfren_{AB,k}.
\end{equation}
It is a prototypical example of an average quantity over transition
paths of practical interest~\cite{EVan06,LuNol15}.

From the strong Markov property and the ergodic theorem for Markov
chains, the mean transition time may be rewritten as
\begin{equation}\label{eq:TAB-rew}
  T_{AB} = \Exp_{\nu^\ret_{\mathcal{A}^-}}\left[\Delta \tau^\bfren_{AB,0}\right],
\end{equation}
where $\nu^\ret_{\mathcal{A}^-}$ is the invariant measure of the
Markov chain $(Y^\bfen_{\eta^\ret_{{\mathcal{A}^-},k}})_{k \geq 0}$,
which is usually called the \emph{reactive entrance distribution} in $\mathcal{A}^-$~\cite{LuNol15}. Here and throughout this article, the notation $\Exp_\mu$ and $\Pr_\mu$ refers to the fact that the stochastic process with respect to which the expectation is taken has initial distribution~$\mu$. When $\mu=\delta_x$, we shall simply write $\Exp_x$ and $\Pr_x$.

\begin{psfrags}
  \psfrag{A}{$A$}
  \psfrag{B}{$B$}
  \psfrag{Yn-}{$q_{\tau^\bfen_n}$}
  \psfrag{Yn+1-}{$q_{\tau^\bfen_{n+1}}$}
\begin{figure}
\begin{center}
    \includegraphics[width=0.7\textwidth]{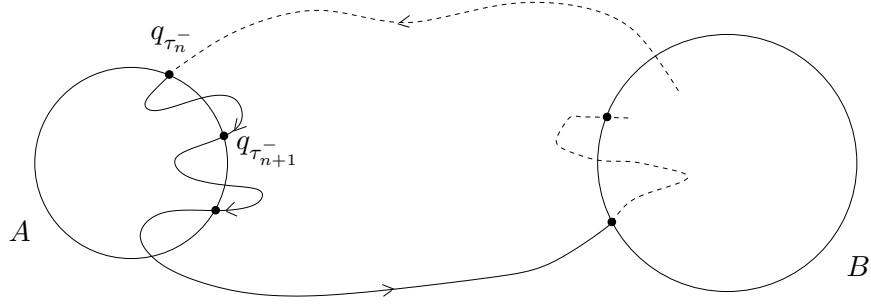}
\caption{Schematic representation of the successive entry points of
  $(q_t)_{t \ge 0}$ in $A \cup B$. The part of the trajectory in solid
line is a transition path from $A$ to $B$.}
\label{fig:section11}
\end{center}
\end{figure}
\end{psfrags}

\subsection{The Hill relation and the main contributions of the article}\label{ss:Hill-intro}

In the case where $A$ and $B$ are metastable sets for $(q_t)_{t \geq 0}$, the formula~\eqref{eq:TAB-rew} presents two main computational difficulties: 
\begin{itemize}
  \item[(i)] the reactive entrance distribution $\nu^\ret_{\mathcal{A}^-}$ is neither explicit nor easy to sample from, 
  \item[(ii)] the direct simulation of the random variable $\Delta \tau^\bfren_{AB,0}$ may be prohibitively long.
\end{itemize}
To overcome these issues, several numerical methods are based on the alternative expression
\begin{equation}\label{eq:Hill-intro}
  T_{AB} = \frac{\Exp_{\pi_{\mathcal{A}^-}}\left[\tau^\bfen_1\right]}{\Pr_{\pi_{\mathcal{A}^-}}(Y^\bfen_1 \in {\mathcal{B}^-})},
\end{equation}
where $\pi_{\mathcal{A}^-} := \pi^\bfen(\cdot|{\mathcal{A}^-})$ is the restriction (in the sense of conditional probability measures) to ${\mathcal{A}^-}$ of the invariant measure $\pi^\bfen$ of the Markov chain $(Y^\bfen_n)_{n \geq 0}$. This formula, which is sometimes attributed to Hill~\cite{Hil05}, is commonly encountered in the statistical physics literature~\cite{AllValTen09,BhaZuc11,ZucCho17}, and used for a large class of stochastic processes, beyond the specific instance of the Langevin dynamics~\eqref{eq:Langevin}\footnote{See for instance the blog post \url{http://statisticalbiophysicsblog.org/?p=8} for more context.}.

The present article is a continuation of~\cite{BauGuyLel}, in which the mathematical analysis of this formula and its use in rare event estimation was initiated for the case where the chain $(Y^\bfen_n)_{n \geq 0}$ takes its values in a compact state space. This allows in particular to obtain a formula similar to~\eqref{eq:Hill-intro} for the \emph{overdamped Langevin dynamics}
\begin{equation}\label{eq:odLangevin}
  \dd \bar{q}_t = F(\bar{q}_t)\dd t + \sqrt{2\beta^{-1}}\dd W_t,
\end{equation}
which describes the $\gamma \to +\infty$ limit of the time rescaled position process $(q_{\gamma t})_{t \geq 0}$ defined by~\eqref{eq:Langevin}. While easier to study mathematically, thanks to the uniform ellipticity of its infinitesimal generator and the fact that metastable sets are generally assumed to be bounded in $\R^d$, the overdamped Langevin dynamics~\eqref{eq:odLangevin} is arguably less physically relevant, and less commonly employed in actual molecular dynamics simulations, than the Langevin dynamics~\eqref{eq:Langevin}.

Our first main contribution is a proof of the identity~\eqref{eq:Hill-intro}, and a clarification of the assumptions under which it holds, in the general setting of positive Harris recurrent chains. In particular, it does not require metastable sets to be bounded and therefore applies to the Langevin dynamics~\eqref{eq:Langevin}.

The major benefit of the Hill relation~\eqref{eq:Hill-intro} is that it no longer involves the reactive entrance distribution~$\nu^\ret_{\mathcal{A}^-}$, but rather the measure $\pi_{\mathcal{A}^-}$. On the one hand, in metastable situations the latter is often argued to be easy to approximate, through direct simulation of the chain. Indeed, starting from ${\mathcal{A}^-}$, one may typically expect $(Y^\bfen_n)_{n \geq 0}$ to perform many steps in ${\mathcal{A}^-}$, and thus to reach a local equilibrium in ${\mathcal{A}^-}$ on a short time scale compared with the first transition toward ${\mathcal{B}^-}$ --- this phenomenon is referred to as \emph{thermalisation} in the study of metastability~\cite{OliVar05}. In~\cite{BauGuyLel}, this local equilibrium is identified as the \emph{quasistationary distribution} in ${\mathcal{A}^-}$ and shown to be a good approximation of the invariant measure $\pi_{\mathcal{A}^-}$. On the other hand, we shall prove in the present article that for the Langevin dynamics~\eqref{eq:Langevin}, the measure $\pi_{\mathcal{A}^-}$ is actually \emph{explicit} and can be directly sampled from, without any approximation argument. This is the second main contribution of this article and explains in which sense the use of the Hill relation~\eqref{eq:Hill-intro} allows to mitigate the first difficulty mentioned at the beginning of this subsection.

In order to address the second difficulty, we rewrite the right-hand side of~\eqref{eq:Hill-intro} under the form
\begin{equation}\label{eq:hill-ams}
  \frac{\Exp_{\pi_{\mathcal{A}^-}}\left[\tau^\bfen_1\right]}{\Pr_{\pi_{\mathcal{A}^-}}(Y^\bfen_1 \in {\mathcal{B}^-})} = \Exp_{\pi_{\mathcal{A}^-}}\left[\tau^\bfen_1 | Y^\bfen_1 \in {\mathcal{A}^-} \right]\left(\frac{1}{\Pr_{\pi_{\mathcal{A}^-}}(Y^\bfen_1 \in {\mathcal{B}^-})}-1\right) + \Exp_{\pi_{\mathcal{A}^-}}\left[\tau^\bfen_1 | Y^\bfen_1 \in {\mathcal{B}^-} \right].
\end{equation}
The quantity $\Exp_{\pi_{\mathcal{A}^-}}[\tau^\bfen_1 | Y^\bfen_1 \in {\mathcal{A}^-}]$ can be evaluated by brute force simulation of the chain $(Y^\bfen_n)_{n \geq 0}$, whereas both quantities $\Exp_{\pi_{\mathcal{A}^-}}[\tau^\bfen_1 | Y^\bfen_1 \in {\mathcal{B}^-}]$ and $\Pr_{\pi_{\mathcal{A}^-}}(Y^\bfen_1 \in {\mathcal{B}^-})$ are statistics of \emph{reactive trajectories} and may therefore be computed by means of rare event simulation algorithms, such as Weighted Ensemble Simulation~\cite{ZucCho17}, Transition Interface Sampling~\cite{VanMorBol03}, Forward Flux Sampling~\cite{AllValTen09}, or Adaptive Multilevel Splitting~\cite{CerGuy07} to name but a few.

\subsection{Relation with earlier works} As has already been mentioned, this work is a continuation of the article~\cite{BauGuyLel}, in which the Hill relation is proved for the overdamped Langevin dynamics~\eqref{eq:odLangevin}, and the main focus of which is put on the quantification of the error introduced when replacing, in the right-hand side of~\eqref{eq:Hill-intro}, the measure $\pi_{\mathcal{A}^-}$ with the quasistationary distribution of $(Y^\bfen_n)_{n \geq 0}$ in ${\mathcal{A}^-}$. We refer in particular to the introduction of~\cite{BauGuyLel} for more details and references on algorithms for rare event simulation in computational statistical physics.

Besides the extension of the proof of the Hill relation to the Langevin dynamics~\eqref{eq:Langevin}, the main novelty of the present article is the explicit computation of the measure $\pi_{\mathcal{A}^-}$. We believe this result to be of general interest, independently from the application to the estimation of statistics of transitions; still, in the latter context, it provides a direct simulation algorithm and spares the need to resort to the quasistationary distribution. Remarkably, this measure also appears in the formula defining the transmission coefficient in Variational Transition State Theory, see~\cite[Eq.~(68)]{VanTal05}. Let us finally mention that the study of quasistationary distributions for Langevin-like dynamics~\eqref{eq:Langevin} was recently carried out in the series of works~\cite{LelRamRey:qsd,Ram20,Ram} and~\cite{GuiNecWu,GuiNecWu:2}, and that it remains crucial for the mathematical analysis of many other algorithms of molecular dynamics dealing with metastability.

At several places in the article, we shall use technical results regarding trajectorial and analytical properties of the Langevin dynamics~\eqref{eq:Langevin} recently proved in our work~\cite{LelRamRey:kfp}. We shall also borrow notation and terminology from various fields which are connected to our study, such as \emph{Transition Path Theory}~\cite{EVan06,EVan10,LuNol15}, potential theory~\cite{BovEckGayKle04,BovDen15}, or the theory of Harris chains~\cite{Asm03,HerLas03,MeyTwe09,DouMouPriSou18}.

\subsection{Organisation of the article} The precise setting under
which we work as well as a detailed statement of our results are
presented in Section~\ref{s:mainres}. Section~\ref{s:prelim} contains
the proof of preliminary results ensuring the well-posedness of the
sequence $(Y^\bfen_n)_{n \geq 0}$. Section~\ref{s:ergo} is dedicated
to the study of its long time behaviour, and the identification of its
invariant measure $\pi^-$. We also carry out a similar study of the
sequence $(Y^\bfex_n)_{n \geq 0}$ of configurations by which the
process $(q_t,p_t)_{t \geq 0}$ \emph{exits} from metastable sets. This
part is actually written for exits from a general  open and smooth set
$\metasp$, and we later apply the obtained results to the specific case
$\metasp=A \cup B$. Section~\ref{s:Hill} is dedicated to the proof of the Hill relation~\eqref{eq:Hill-intro}. We first show a generalised version of this identity in the abstract setting of positive Harris recurrent chains, and then check that this identity applies to the Langevin dynamics~\eqref{eq:Langevin}. Last, generalities on Harris chains are collected in Appendix~\ref{app:Harris}, Appendix~\ref{app:DL} contains the proof of an auxiliary result used in the paper concerning the probabilistic interpretation of the Dirichlet problem for the Langevin dynamics, and Appendix~\ref{app:Etau} is dedicated to the discussion of an assumption made in our last statement regarding the finiteness of~$T_{AB}$.

\section{Setting and statement of the main results}\label{s:mainres}

\subsection{Basic notation}\label{ss:basnot} For $a,b \in \R$, we write $a \wedge b = \min(a,b)$, $a \vee b = \max(a,b)$, $[a]_+ = 0 \vee a$, $[a]_- = 0 \vee (-a)$. The notation $|u| = \sqrt{u \cdot u}$ refers to the Euclidean norm on $\R^d$. For any $r>0$, $\mathrm{B}(u,r)$ and $\bar{\mathrm{B}}(u,r)$ respectively denote the open and closed balls centered in $u \in \R^d$ and with radius $r$. More generally, the closure of a set $A \subset \R^d$ is denoted by $\bar{A}$.

For a random variable $X$ in some measurable space $\mathcal{S}$ and a
probability measure $\mu$ on $\mathcal{S}$, the notation $X \sim \mu$
means that $\mu$ is the law of $X$. For any $f \in
L^1(\mathcal{S},\mu)$, we use the notation $\mu(f)$ as a shorthand for
$\int_{\mathcal S} f \dd \mu$. Given a probability measure $\mu(\dd x)$ and a Markov kernel $P(x,\dd y)$ on $\mathcal{S}$, we denote by $\mu \otimes P$ the probability measure $\mu(\dd x)P(x,\dd y)$ on $\mathcal{S} \times \mathcal{S}$. Finally, for any measurable function~$g:\mathcal{S} \to \mathcal{T}$, where $\mathcal{S}$ and $\mathcal{T}$ are measurable spaces, we denote by $\mu \circ g^{-1}$ or $\mu(g^{-1}(\dd x))$ the pushforward of $\mu$ by $g$, and we denote by $P(x,g^{-1}(\dd y))$ the pushforward of $P(x,\cdot)$ by $g$.


\subsection{Assumptions on the continuous dynamics}\label{ss:ass-Langevin}

Throughout this work, we assume that the vector field $F : \R^d \to \R^d$ is $C^\infty$. As a consequence, the coefficients of the stochastic differential equation~\eqref{eq:Langevin} are locally Lipschitz continuous, therefore this equation possesses a unique strong solution, which is defined on some filtered probability space $(\Omega,\mathcal{F},(\mathcal{F}_t)_{t \geq 0}, \Pr)$ up to some explosion time $\tau_\infty$. We shall work under the following set of assumptions:
\begin{enumerate}[label=(A\arabic*), ref=A\arabic*]
  \item\label{ass:A1} $\tau_\infty=\infty$, almost surely;
  \item\label{ass:A2} the process $(q_t,p_t)_{t \geq 0}$ has a unique stationary distribution $\mu(\dd q\dd p)$, this measure has a smooth and positive density $\rho(q,p)$ with respect to the Lebesgue measure on $\R^d \times \R^d$, and for any $G \in L^1(\R^d \times \R^d, \mu)$,
  \begin{equation*}
    \lim_{t \to +\infty} \frac{1}{t}\int_0^t G(q_s,p_s)\dd s = \mu(G), \qquad \text{almost surely,}
  \end{equation*}
  for any initial condition;
  \item\label{ass:A3} the density $\rho(q,p)$ satisfies
  \begin{equation*}
      \int_{\R^d \times \R^d} \left\{(|F(q)| + |p|) \rho(q,p) + |\nabla_p \rho(q,p)|\right\} \dd q\dd p < +\infty.
  \end{equation*}
\end{enumerate}

In the conservative case described in the introduction of the article, namely when $F = -\nabla V$ for some $C^\infty$ potential function $V : \R^d \to \R$, these assumptions are satisfied if 
\begin{equation*}
  \lim_{|q| \to +\infty} V(q) = +\infty, \qquad \int_{\R^d} (1+|\nabla V(q)|)\mathrm{e}^{-\beta V(q)}\dd q < +\infty.
\end{equation*} 
In this case, $\mu$ is the Boltzmann--Gibbs measure with density $\rho$ defined in~\eqref{eq:BG}. We refer to~\cite[Examples~5.10 and~7.3]{Rey06}, \cite[Remark~2.21]{LelRamRey:kfp} and~\cite[Théorème~1]{Pag01} for details.


\begin{rk}
  The Langevin dynamics can more generally be written
  \begin{equation*}
    \left\{\begin{aligned}
      \dd q_t &= M^{-1}p_t \dd t,\\
      \dd p_t &= F(q_t)\dd t - \gamma M^{-1}p_t \dd t + \sqrt{2\gamma\beta^{-1}}\dd W_t,
    \end{aligned}\right.
  \end{equation*}
  where $M$ is a diagonal matrix with positive diagonal entries, describing the masses of the particles. In this formulation, the coordinate $p_t \in \R^d$ refers to a vector of \emph{momenta} rather than velocities. This system may however be reduced to~\eqref{eq:Langevin}, for which $M$ is the identity, through a simple change of variables~\cite[Remark~3.37, p.~210]{LelRouSto10}.
\end{rk}

\subsection{Notation and assumptions on metastable sets} In this subsection, we let $\metasp \subset \R^d$ satisfy the following assumptions.
\begin{enumerate}[label=(B), ref=B]
  \item\label{ass:B1} The sets $\metasp$ and $\R^d \setminus
    \bar{\metasp}$ are open, nonempty, and have a $C^2$ boundary $\Sigma$.
\end{enumerate}
A straightforward consequence of this assumption is that $\metasp$ and $\R^d \setminus
    \bar{\metasp}$ have positive Lebesgue measures.
Let us emphasise the fact that, throughout this article, neither $\metasp$ nor $\R^d \setminus \bar{\metasp}$ are assumed to be bounded.

Under these assumptions, for any $q \in \Sigma$ we denote by $\mathsf{n}(q) \in \R^d$ the unit normal vector to $\Sigma$ which is oriented toward the exterior of $\metasp$. This allows us to introduce the partition of $\Sigma \times \R^d$ into three sets
\begin{align*}
  \Gamma^\bfex &:= \{(q,p) \in \Sigma \times \R^d : p \cdot \mathsf{n}(q) > 0\},\\
  \Gamma^\bfen &:= \{(q,p) \in \Sigma \times \R^d : p \cdot \mathsf{n}(q) < 0\},\\
  \Gamma^0 &:= \{(q,p) \in \Sigma \times \R^d : p \cdot \mathsf{n}(q) = 0\}.
\end{align*}

The results of Subsections~\ref{ss:defY} and~\ref{ss:erg-Y} are stated
for a general set $\metasp$ satisfying Assumption~(\ref{ass:B1}). In
Subsection~\ref{ss:Hill-Langevin}, in order to state the Hill
relation~\eqref{eq:Hill-intro}, we shall apply these results to the
specific case discussed in the introduction where $\metasp = A \cup B$ for two open subsets $A, B \subset \R^d$ with disjoint closures.

\subsection{The sequences \texorpdfstring{$(Y^\bfex_n)_{n \geq 0}$}{Y+} and \texorpdfstring{$(Y^\bfen_n)_{n \geq 0}$}{Y-}}\label{ss:defY}

The main objects of interest in this article are the sequences $(Y^\bfex_n)_{n \geq 0}$ and $(Y^\bfen_n)_{n \geq 0}$ of successive exits from, and entrances in, the set $\metasp \times \R^d$. To define these sequences, we need the following preliminary result.

\begin{lem}[Return time to $\Sigma$]\label{lem:return}
  Under Assumptions~(\ref{ass:A1}--\ref{ass:A2}) and~(\ref{ass:B1}), let $\tau := \inf\{t > 0: q_t \in \Sigma\}$. For any $(q,p) \in (\R^d \times \R^d) \setminus \Gamma^0$, we have $0 < \tau < +\infty$, $\mathbb{P}_{(q,p)}$-almost surely. Besides:
  \begin{enumerate}[label=(\roman*),ref=\roman*]
    \item if $q \in \metasp$ or $(q,p) \in \Gamma^\bfen$, then $(q_\tau,p_\tau) \in \Gamma^\bfex$;
    \item if $q \in \R^d \setminus \bar{\metasp}$ or $(q,p) \in \Gamma^\bfex$, then $(q_\tau,p_\tau) \in \Gamma^\bfen$.
  \end{enumerate}
  Besides, if $(q,p) \in \Gamma^0$, then
  \begin{equation*}
      \inf\{t \geq 0: (q_t,p_t) \in \Gamma^\bfen\} = \inf\{t \geq 0: (q_t,p_t) \in \Gamma^\bfex\} = 0, \qquad \text{$\Pr_{(q,p)}$-almost surely.}
  \end{equation*}
\end{lem}
The first part of Lemma~\ref{lem:return} allows us to define, for any starting point $(q,p) \in (\R^d \times \R^d) \setminus \Gamma^0$, the sequences of stopping times $(\tau^\bfex_n)_{n \geq 0}$ and $(\tau^\bfen_n)_{n \geq 0}$ by: 
\begin{align*}
  \tau^\bfex_0 := \inf\{t \geq 0: (q_t,p_t) \in \Gamma^\bfex\}, \qquad \tau^\bfex_{n+1} := \inf\{t > \tau^\bfex_n: (q_t,p_t) \in \Gamma^\bfex\}, \quad n \geq 0,\\
  \tau^\bfen_0 := \inf\{t \geq 0: (q_t,p_t) \in \Gamma^\bfen\}, \qquad \tau^\bfen_{n+1} := \inf\{t > \tau^\bfen_n: (q_t,p_t) \in \Gamma^\bfen\}, \quad n \geq 0.
\end{align*}
Notice that the process $(q_t,p_t)_{t \geq 0}$ satisfies the strong Markov property, which combined with Lemma~\ref{lem:return} shows that both sequences are increasing.

\begin{lem}[Nonaccumulation of $(\tau^\bfex_n)_{n \geq 0}$ and $(\tau^\bfen_n)_{n \geq 0}$]\label{lem:nonacc}
  Under the assumptions of Lemma~\ref{lem:return}, for any $(q,p) \in (\R^d \times \R^d) \setminus \Gamma^0$, we have 
  \begin{equation*}
    \lim_{n \to +\infty} \tau^\bfex_n = \lim_{n \to +\infty} \tau^\bfen_n = +\infty, \qquad \text{$\mathbb{P}_{(q,p)}$-almost surely.}
  \end{equation*}
\end{lem}

Lemmas~\ref{lem:return} and~\ref{lem:nonacc} are proved in Section~\ref{s:prelim}.

\begin{rk}[Intertwining between the sequences $(\tau^\bfex_n)_{n \geq 0}$ and $(\tau^\bfen_n)_{n \geq 0}$]\label{rk:intertwining}
  The sequences $(\tau^\bfex_n)_{n \geq 0}$ and $(\tau^\bfen_n)_{n \geq 0}$ are intertwined in the following sense: 
  \begin{itemize}
    \item if $q \in \metasp$ or $(q,p) \in \Gamma^\bfex$, then $\tau_0^\bfex < \tau_0^\bfen < \tau_1^\bfex < \cdots$,
    \item if $q \in \R^d \setminus \bar{\metasp}$ or $(q,p) \in \Gamma^\bfen$, then $\tau_0^\bfen < \tau_0^\bfex < \tau_1^\bfen < \cdots$.
  \end{itemize}
\end{rk}
By the strong Markov property, the random sequences $(Y^\bfex_n)_{n \geq 0}$ and $(Y^\bfen_n)_{n \geq 0}$ defined by 
\begin{equation*}
  Y^\bfex_n := (q_{\tau^\bfex_n}, p_{\tau^\bfex_n}), \qquad Y^\bfen_n := (q_{\tau^\bfen_n}, p_{\tau^\bfen_n}),
\end{equation*}
are time homogeneous Markov chains respectively taking their values in $\Gamma^\bfex$ and $\Gamma^\bfen$, and correspond to the successive exit and entrance points in $\metasp \times \R^d$. We describe their ergodic behaviour in Subsection~\ref{ss:erg-Y} and then present the Hill relation in Subsection~\ref{ss:Hill-Langevin}. We summarise in Tables~\ref{tab:Y} and~\ref{tab:eta} the main notation on which we rely.

\subsection{Ergodic behaviour of the sequences \texorpdfstring{$(Y^\bfex_n)_{n \geq 0}$}{Y+} and \texorpdfstring{$(Y^\bfen_n)_{n \geq 0}$}{Y-}}\label{ss:erg-Y}

We denote by $\dd\sigma_\Sigma(q)$ the surface measure on $\Sigma
\subset \mathbb R^d$ induced by the Lebesgue measure in $\mathbb R^d$
and the Euclidean scalar product. In addition to Assumptions~(\ref{ass:A1}--\ref{ass:A2}--\ref{ass:A3}) and~(\ref{ass:B1}), we suppose that
\begin{enumerate}[label=(C),ref=C]
  \item\label{ass:C} the function $(q,p) \in \Sigma \times \R^d \mapsto |p \cdot \mathsf{n}(q)| \rho(q,p)$ is in $L^1(\Sigma \times \R^d, \dd\sigma_\Sigma(q)\dd p)$,
\end{enumerate}
and introduce the notation
\begin{equation*}
  Z^\bfex := \int_{\Gamma^\bfex} |p \cdot \mathsf{n}(q)| \rho(q,p) \dd\sigma_\Sigma(q)\dd p, \qquad Z^\bfen := \int_{\Gamma^\bfen} |p \cdot \mathsf{n}(q)| \rho(q,p) \dd\sigma_\Sigma(q)\dd p.
\end{equation*}
Notice that by Assumption~(\ref{ass:A2}), $\rho(q,p)>0$ on $\Sigma
\times \R^d$ and therefore $Z^\bfex > 0$ and $Z^\bfen > 0$. Let us
point out that in the conservative case $F=-\nabla V$, Assumption~~(\ref{ass:C}) equivalently rewrites
\begin{equation*}
    \int_\Sigma \ee^{-\beta V(q)}\dd\sigma_\Sigma(q) < +\infty,
\end{equation*}
which is not necessarily implied by the overall integrability of $\ee^{-\beta V}$ on $\R^d$.

\subsubsection{Main result} The following theorem is the first main result of this article. We refer to Appendix~\ref{app:Harris} for a summary of basic facts regarding Harris recurrent chains.

\begin{theo}[Ergodicity of $(Y^\bfex_n)_{n \geq 0}$ and $(Y^\bfen_n)_{n \geq 0}$]\label{theo:main}
  Let Assumptions~(\ref{ass:A1}--\ref{ass:A2}--\ref{ass:A3}), (\ref{ass:B1}) and~(\ref{ass:C}) hold.
  \begin{enumerate}[label=(\roman*),ref=\roman*]
    \item The Markov chains $(Y^\bfex_n)_{n \geq 0}$ and $(Y^\bfen_n)_{n \geq 0}$ are positive Harris recurrent, with unique invariant probability measures respectively denoted by $\pi^\bfex$ and $\pi^\bfen$.
    \item The measures $\pi^\bfex$ and $\pi^\bfen$ have respective densities
  \begin{equation}\label{eq:rho+-}
    \varrho^\bfex(q,p) := \frac{1}{Z^\bfex} \ind{(q,p) \in \Gamma^\bfex} |p \cdot \mathsf{n}(q)| \rho(q,p), \qquad \varrho^\bfen(q,p) := \frac{1}{Z^\bfen} \ind{(q,p) \in \Gamma^\bfen} |p \cdot \mathsf{n}(q)| \rho(q,p),
  \end{equation}
  with respect to the measure $\dd\sigma_\Sigma(q)\dd p$ on $\Sigma \times \R^d$.
  \end{enumerate}
\end{theo} 

Theorem~\ref{theo:main} is proved in Section~\ref{s:ergo}.

\begin{rk}\label{rk:z}
  We shall see in the proof of Theorem~\ref{theo:main} that $Z^\bfex=Z^\bfen$. In the conservative case described in Subsection~\ref{ss:ass-Langevin}, this fact is actually obvious since it can be checked directly that $\rho(q,p)=\rho(q,-p)$, which then implies $Z^\bfex=Z^\bfen$.
\end{rk}

\begin{rk}
  The proof of Theorem~\ref{theo:main} shows that if Assumption~(\ref{ass:C}) does not hold, then the $\sigma$-finite measures with densities
  \begin{equation*}
    \ind{(q,p) \in \Gamma^\bfex} |p \cdot \mathsf{n}(q)| \rho(q,p), \qquad \ind{(q,p) \in \Gamma^\bfen} |p \cdot \mathsf{n}(q)| \rho(q,p),
  \end{equation*}
  with respect to the measure $\dd\sigma_\Sigma(q)\dd p$ on $\Sigma \times \R^d$ remain the unique (up to a multiplicative constant) invariant $\sigma$-finite measures of the Harris recurrent Markov chains $(Y^\bfex_n)_{n \geq 0}$ and $(Y^\bfen_n)_{n \geq 0}$.
\end{rk}

\subsubsection{Sequence of successive crossings} One may also be interested in the Markov chain $(Y^\Sigma_m)_{m \geq 0}$ defined as the sequence of the successive crossings of the surface $\Sigma$ by the process $(q_t,p_t)_{t \geq 0}$. Following Remark~\ref{rk:intertwining}, this sequence writes either $(Y^\bfex_0, Y^\bfen_0, Y^\bfex_1, \ldots)$ or $(Y^\bfen_0, Y^\bfex_0, Y^\bfen_1, \ldots)$ depending on whether $\tau^\bfex_0 < \tau^\bfen_0$ or $\tau^\bfen_0 < \tau^\bfex_0$. In this perspective, the equivalent of Theorem~\ref{theo:main} for the Markov chain $(Y^\Sigma_m)_{m \geq 0}$ reads as follows.

\begin{theo}[Ergodicity of $(Y^\Sigma_m)_{m \geq 0}$]\label{theo:Y}
  Under the assumptions of Theorem~\ref{theo:main}, the Markov chain $(Y^\Sigma_m)_{m \geq 0}$ is positive Harris recurrent with unique invariant probability measure
  \begin{equation*}
    \pi^\Sigma := \frac{1}{2}\left(\pi^\bfex + \pi^\bfen\right) = \frac{1}{2Z^\Sigma}|p \cdot \mathsf{n}(q)| \rho(q,p)\dd\sigma_\Sigma(q)\dd p,
  \end{equation*}
  where following Remark~\ref{rk:z} we have set $Z^\Sigma:=Z^\bfex=Z^\bfen$.
\end{theo}

Theorem~\ref{theo:Y} is also proved in Section~\ref{s:ergo}. We provide in Table~\ref{tab:Y} a summary of some notation introduced so far.

\begin{table}[htbp]
\begin{center}\footnotesize
\begin{tabular}{c|c|c|c}
& Stopping times & Markov Chains & Stationary measures \\\hline
Entry points in $A \cup B$& $\tau_{n+1}^-=\inf\{t > \tau_{n}^-,\,
                            (q_t,p_t) \in \Gamma^-\}$ & $Y_n^-=(q_{\tau_n^-},p_{\tau_n^-})$ & $\pi^-$\\
Exit points from $A \cup B$ & $\tau_{n+1}^+=\inf\{t > \tau_{n}^+,\,
                            (q_t,p_t) \in \Gamma^+\}$ & $Y_n^+=(q_{\tau_n^+},p_{\tau_n^+})$ & $\pi^+$\\
Crossing points of $\Sigma$& $\tau_{m+1}=\inf\{t > \tau_{m},\,
                            (q_t,p_t) \in \Gamma^+\cup \Gamma^-\}$ 
&  $Y_m^\Sigma=(q_{\tau_m},p_{\tau_m})$     & $ \pi^\Sigma=\frac 1 2 (\pi^-+\pi^+)$\\\hline
\end{tabular}
\end{center}
\caption{Definitions of the Markov chains $(Y_n^-)_{n \ge 0}$,
  $(Y_n^+)_{n \ge 0}$ and $(Y_m^\Sigma)_{m \ge 0}$, with values
  respectively in $\Gamma^-$, $\Gamma^+$ and $\Gamma^- \cup \Gamma^+$.}\label{tab:Y}
\end{table}

\subsubsection{Sampling from \texorpdfstring{$\pi^\bfen$}{} and \texorpdfstring{$\pi^\bfex$}{} in the conservative case}\label{sss:sample} In the conservative case $F=-\nabla V$, we deduce from~\eqref{eq:BG} and Theorem~\ref{theo:main} that the probability measures $\pi^\pm$ have a density proportional to
\begin{equation*}
  [p \cdot \mathsf{n}(q)]_\pm \ee^{-\beta H(q,p)}, \qquad H(q,p) = V(q) + \frac{|p|^2}{2},
\end{equation*}
with respect to the measure $\dd\sigma_\Sigma(q)\dd p$. From the numerical point of view, sampling from these measures can be achieved through the following two-step procedure:
\begin{itemize}
  \item[(i)] draw $q$ according to the probability measure with density proportional to $\ee^{-\beta V(q)}$ with respect to the surface measure $\dd\sigma_\Sigma(q)$ on $\Sigma$;
  \item[(ii)] conditionally on $\mathsf{n}(q)$, draw $p$ according to the density proportional to $[p \cdot \mathsf{n}(q)]_\pm\ee^{-\beta|p|^2/2}$ with respect to the Lebesgue measure on $\R^d$.
\end{itemize}
Several methods are available to sample from densities on
manifolds~\cite{DiaHolSha13,LelRouSto19,LSZ22} and may be employed to
draw $q$ in the first step. The second step only requires to draw
$d+1$ independent standard Gaussian variables $G_0, G_1, \ldots, G_d$, and set 
\begin{equation*}
  p = \frac{1}{\sqrt{\beta}}\left(\pm\sqrt{G_0^2 + G_1^2} \mathsf{n}(q) + G_2 \mathsf{e}_2 + \cdots + G_d \mathsf{e}_d\right),
\end{equation*}
where $(\mathsf{e}_2, \ldots, \mathsf{e}_d)$ is an orthonormal basis
of $\mathsf{n}(q)^\perp$, which is the tangent space of $\Sigma$ at the point $q$. Then it is elementary to check that, conditionally on $q$, the vector $p$ has the claimed distribution. 

\subsubsection{Reversibility up to momentum reversal in the conservative case} In the conservative case, it is known that if $(q_0,p_0)$ is distributed according to the stationary Boltzmann--Gibbs measure~\eqref{eq:BG}, then for any $T>0$, the following equality in law holds
$$(q_t,p_t)_{0 \le t \le T} \stackrel{\mathcal L}{=}  (\mathsf{R}(q_{T-t},p_{T-t}))_{0 \le t \le T},$$
where $\mathsf{R}$ is the momentum reversal operator defined by $\mathsf{R}(q,p)=(q,-p)$. This property of the Langevin dynamics is sometimes called the reversibility up to momentum reversal~\cite[Section 2.2.3]{LelRouSto10}. The next statement shows that it also holds at the level of the Markov chain $(Y^\Sigma_m)_{m \geq 0}$. 

\begin{prop}[Reversibility up to momentum reversal for $(Y^\Sigma_m)_{m \geq 0}$]\label{prop:rev}
  Let the assumptions of Theorem~\ref{theo:Y} hold, with $F=-\nabla V$. If $Y^\Sigma_0 \sim \pi^\Sigma$, then the pairs $(Y^\Sigma_0,Y^\Sigma_1)$ and $(\mathsf{R}(Y^\Sigma_1),\mathsf{R}(Y^\Sigma_0))$ have the same law.
\end{prop}

Proposition~\ref{prop:rev} is proved in Section~\ref{s:ergo}. It rewrites under the `detailed-balance' form
\begin{equation*}
    \pi^\Sigma(\dd y_0) P^\Sigma(y_0, \dd y_1) =\pi^\Sigma(\mathsf{R}^{-1}(\dd y_1)) P^\Sigma(\mathsf{R}^{-1}(y_1),  \mathsf{R}^{-1} (\dd y_0)),
\end{equation*}
where $P^\Sigma$ denotes the transition kernel of $(Y^\Sigma_m)_{m \geq 0}$ and we recall that the notation for pushforwards of measures and kernels is introduced in Subsection~\ref{ss:basnot}. Combining this identity with the fact that the invariant measure $\pi^\Sigma$ is left invariant by the momentum reversal map $\mathsf{R}$, 
one can easily check that   under the assumptions of Proposition~\ref{prop:rev}, if $Y^\Sigma_0 \sim \pi^\Sigma$, then the triples $(Y^\Sigma_0,Y^\Sigma_1,Y^\Sigma_2)$ and $(\mathsf{R}(Y^\Sigma_2),\mathsf{R}(Y^\Sigma_1),\mathsf{R}(Y^\Sigma_0))$ also have the same law. Applying this statement with test functions $\ind{y_0 \in \Gamma^\pm}g^\pm(y_0,y_2)$, for $g^\pm : \Gamma^\pm \times \Gamma^\pm \to \R$, leads to the following result.

\begin{cor}[Intertwinned reversibility for $(Y^\bfen_n)_{n \geq 0}$ and $(Y^\bfex_n)_{n \geq 0}$]\label{cor:rev}
  Under the assumptions of Proposition~\ref{prop:rev},
  \begin{enumerate}[label=(\roman*),ref=\roman*]
      \item the law of $(Y_0^\bfex,Y_1^\bfex)$ under $\Pr_{\pi^\bfex}$ is the same as the law of $(\mathsf{R}(Y_1^\bfen),\mathsf{R}(Y_0^\bfen))$ under $\Pr_{\pi^\bfen}$;
      \item the law of $(Y_0^\bfen,Y_1^\bfen)$ under $\Pr_{\pi^\bfen}$ is the same as the law of $(\mathsf{R}(Y_1^\bfex),\mathsf{R}(Y_0^\bfex))$ under $\Pr_{\pi^\bfex}$. \end{enumerate}
\end{cor}

\subsection{The Hill relation for the Langevin dynamics}\label{ss:Hill-Langevin} In this subsection, we assume that $\metasp = A \cup B$, where $A$ and $B$ are nonempty open $C^2$ subsets of $\R^d$, with $\bar{A} \cap \bar{B} = \emptyset$. Then Assumption~(\ref{ass:B1}) is satisfied, and we introduce the partition of $\Gamma^\bfen$ into the two sets
\begin{equation*}
  {\mathcal{A}^-} := \Gamma^\bfen \cap (\partial A \times \R^d), \qquad {\mathcal{B}^-} := \Gamma^\bfen \cap (\partial B \times \R^d).
\end{equation*}

\subsubsection{Hill relation} Under the assumptions of Theorem~\ref{theo:main}, the Markov chain $(Y^\bfen_n)_{n \geq 0}$ visits infinitely often both sets ${\mathcal{A}^-}$ and ${\mathcal{B}^-}$, which allows to define the sequences $(\eta^\ret_{{\mathcal{A}^-},k})_{k \geq 0}$ and $(\eta^\ret_{{\mathcal{B}^-},k})_{k \geq 0}$ as in~\eqref{eq:etare}, see also Table~\ref{tab:eta} below. We then introduce the notation
\begin{equation*}
  \forall k \geq 0, \qquad Y^\ret_{{\mathcal{A}^-},k} := Y^\bfen_{\eta^\ret_{{\mathcal{A}^-},k}}.
\end{equation*}

\begin{prop}[Definition of the reactive entrance distribution]\label{prop:nuA}
  Under the assumptions of Theorem~\ref{theo:main} with $\metasp = A
  \cup B$, the sequence $(Y^\ret_{{\mathcal{A}^-},k})_{k \geq 0}$ is a
  positive Harris recurrent Markov chain. Its unique invariant
  probability measure $\nu^\ret_{\mathcal{A}^-}$ is the so-called \emph{reactive entrance distribution} in ${\mathcal{A}^-}$.
\end{prop}

Notice that the term \emph{reactive entrance distribution} has been in
particular introduced in the framework of the \emph{Transition Path Theory}, see~\cite{EVan06,EVan10,LuNol15}. We refer to Table~\ref{tab:eta} for a summary of the notation related with the reactive entrance distribution.

\begin{table}[htbp]
\begin{center}\footnotesize
\begin{tabular}{c|c|c|c}
& Stopping times & Markov Chains & Stationary measures \\\hline
Entry in $A$ coming from $B$& $\eta^\ret_{{\mathcal{A}^-},k+1} = \min\{n \geq \eta^\ret_{{\mathcal{B}^-},k}: Y^\bfen_n \in {\mathcal{A}^-}\}$ & $Y^\ret_{{\mathcal{A}^-},k}=Y^-_{\eta^\ret_{{\mathcal{A}^-},k}}$ & $\nu^\ret_{\mathcal{A}^-}$\\
Entry in $B$ coming from $A$ & $\eta^\ret_{{\mathcal{B}^-},k+1} = \min\{n \geq \eta^\ret_{{\mathcal{A}^-},k+1}: Y^\bfen_n \in {\mathcal{B}^-}\}$ &  $Y^\ret_{{\mathcal{B}^-},k}=Y^-_{\eta^\ret_{{\mathcal{B}^-},k}}$ & $\nu^\ret_{\mathcal{B}^-}$\\\hline
\end{tabular}
\end{center}
\caption{Definitions of the Markov chains
  $(Y^\ret_{{\mathcal{A}^-},k})_{k \ge 0}$ and
  $(Y^\ret_{{\mathcal{B}^-},k})_{k \ge 0}$, with values respectively
  in $\mathcal A^-= \Gamma^- \cap (\partial A \times \R^d)$ and $\mathcal
  B^-= \Gamma^- \cap (\partial B \times \R^d)$.}\label{tab:eta}
\end{table}

We are now in position to present the second main result of this work:
the Hill relation for the Langevin dynamics. Let us first state the Hill relation for the time-discrete dynamics
$(Y^-_n)_{n \ge 0}$.
\begin{theo}[Hill relation for the Langevin dynamics]\label{theo:Hill-Langevin}
  In the setting of Proposition~\ref{prop:nuA}, let $\pi_{\mathcal{A}^-}$ refer to the conditional measure $\pi^\bfen(\cdot|{\mathcal{A}^-})$. For any $g \in L^1({\mathcal{A}^-},\pi_{\mathcal{A}^-})$, we have
  \begin{equation*}
    \Exp_{\nu^\ret_{\mathcal{A}^-}}\left[\sum_{n=0}^{\eta^\ret_{{\mathcal{B}^-},0}-1} |g(Y^\bfen_n)|\right] < +\infty,
  \end{equation*}
  and
  \begin{equation*}
    \Exp_{\nu^\ret_{\mathcal{A}^-}}\left[\sum_{n=0}^{\eta^\ret_{{\mathcal{B}^-},0}-1} g(Y^\bfen_n)\right] = \frac{\pi_{\mathcal{A}^-}(g)}{\Pr_{\pi_{\mathcal{A}^-}}(Y^\bfen_1 \in {\mathcal{B}^-})}.
  \end{equation*}
\end{theo}

The formula~\eqref{eq:Hill-intro} for the mean transition time
$T_{AB}$ then comes as a corollary of Proposition~\ref{prop:nuA} and
Theorem~\ref{theo:Hill-Langevin}, by considering as a test function $g(y)=\Exp_y[\tau_1^\bfen]$, under the following supplementary assumption:
\begin{enumerate}[label=(D),ref=D]
  \item\label{ass:D} in the setting of Theorem~\ref{theo:main}, $\Exp_{\pi^\bfen}[\tau_0^\bfex] + \Exp_{\pi^\bfex}[\tau_0^\bfen] < +\infty$.
\end{enumerate}
This is stated in the next corollary, where we actually address more general \emph{statistics of transitions}, of the form
\begin{equation*}
    \lim_{\ell \to +\infty} \frac{1}{\ell}\sum_{k=0}^{\ell-1} \int_{\tau^\ret_{{\mathcal{A}^-},k}}^{\tau^\ret_{{\mathcal{B}^-},k}} G(q_s,p_s)\dd s,
\end{equation*}
for some function $G : \R^d \times \R^d \to \R$, by applying 
Theorem~\ref{theo:Hill-Langevin} to the test function
$g(y)=\Exp_y[\int_0^{\tau_1^\bfen} G(q_s,p_s) \dd s]$.

\begin{cor}[Hill relation for statistics of transitions]\label{cor:MRT}
  In the setting of Proposition~\ref{prop:nuA} and under the supplementary Assumption~(\ref{ass:D}), let $G : \R^d \times \R^d \to \R$ be a bounded and measurable function. We have
  \begin{equation}\label{eq:MRT:exp}
    \Exp_{\nu^\ret_{\mathcal{A}^-}}\left[\int_0^{\tau^\ret_{{\mathcal{B}^-},0}} |G(q_s,p_s)|\dd s\right] < +\infty, \qquad \Exp_{\pi_{\mathcal{A}^-}}\left[\int_0^{\tau^\bfen_1} |G(q_s,p_s)|\dd s\right] < +\infty.
  \end{equation}
  Besides, for any initial condition $(q_0,p_0)$,
  \begin{equation}\label{eq:MRT:limps}
    \lim_{\ell \to +\infty} \frac{1}{\ell}\sum_{k=0}^{\ell-1} \int_{\tau^\ret_{{\mathcal{A}^-},k}}^{\tau^\ret_{{\mathcal{B}^-},k}} G(q_s,p_s)\dd s = \Exp_{\nu^\ret_{\mathcal{A}^-}}\left[\int_0^{\tau^\ret_{{\mathcal{B}^-},0}} G(q_s,p_s)\dd s\right] \qquad \text{almost surely,}
  \end{equation}
  and the right-hand side satisfies the identity
  \begin{equation}\label{eq:MRT:Hill}
    \Exp_{\nu^\ret_{\mathcal{A}^-}}\left[\int_0^{\tau^\ret_{{\mathcal{B}^-},0}} G(q_s,p_s)\dd s\right] = \frac{\displaystyle\Exp_{\pi_{\mathcal{A}^-}}\left[\int_0^{\tau^\bfen_1} G(q_s,p_s)\dd s\right]}{\Pr_{\pi_{\mathcal{A}^-}}(Y^\bfen_1 \in {\mathcal{B}^-})}.
  \end{equation}
\end{cor}

In particular, taking $G \equiv 1$, we deduce that the limit~\eqref{eq:TAB} exists almost surely, and satisfies the identities~\eqref{eq:TAB-rew} and~\eqref{eq:Hill-intro}. 

Proposition~\ref{prop:nuA}, Theorem~\ref{theo:Hill-Langevin} and
Corollary~\ref{cor:MRT} are proved in Section~\ref{s:Hill}, where the
potential theoretic interpretation of the Hill relation is also
discussed. Assumption~(\ref{ass:D}) is discussed in
Appendix~\ref{app:Etau}, where we use results by Kopec~\cite{Kop15} to
show that it holds for example if $A$ and $B$ are bounded, and
$F=-\nabla V$ where $V$ is smooth and grows at least quadratically at infinity.


\subsubsection{Practical use in rare event algorithms} Let the assumptions of Corollary~\ref{cor:MRT} hold. Following the identity~\eqref{eq:hill-ams}, to compute such an observable as $T_{AB}$, one has to estimate $\Exp_{\pi_{\mathcal{A}^-}}[\tau^\bfen_1 | Y^\bfen_1 \in {\mathcal{B}^-}]$ and $\Pr_{\pi_{\mathcal{A}^-}}(Y^\bfen_1 \in {\mathcal{B}^-})$. To proceed, since under $\pi_{\mathcal{A}^-}$, $\tau^\bfen_0=0$, one may sample an initial condition $(q_0,p_0)$ from $\pi_{\mathcal{A}^-}$, using the procedure described in §~\ref{sss:sample}, and then use a rare event sampling algorithm to simulate a trajectory of $(q_t,p_t)$ over $[0,\tau^\bfen_1]$, conditionally on the event $\{(q_{\tau^\bfen_1}, p_{\tau^\bfen_1}) \in {\mathcal{B}^-}\}$. It is worth pointing out here that, by the strong Markov property and Proposition~\ref{prop:exist} below,
\begin{equation*}
    \Pr_{\pi_{\mathcal{A}^-}}\left((q_{\tau^\bfen_1}, p_{\tau^\bfen_1}) \in {\mathcal{B}^-}\right) = \Pr_{\pi_{\mathcal{A}^+}}\left((q_{\tau^\bfen_0}, p_{\tau^\bfen_0}) \in {\mathcal{B}^-}\right),
\end{equation*}
where the notation $\pi_{\mathcal{A}^+}$ refers to the restriction
of $\pi^\bfex$ to $\mathcal{A}^+:=\Gamma^\bfex \cap (\partial A \times
\R^d)$. Therefore, to estimate the left-hand side, it is also possible
to initialise the rare event algorithm directly under
$\pi_{\mathcal{A}^+}$ (following again the procedure from
§~\ref{sss:sample}). When using this trick to estimate
$\Exp_{\pi_{\mathcal{A}^-}}[\tau^\bfen_1 | Y^\bfen_1 \in
{\mathcal{B}^-}]$, one has to take into account the time elapsed between
$0$ and $\tau^\bfex_0$ to get the correct expected time (this small
correction is however often neglected in practice, see for
example~\cite[Equation (6)]{AllValTen09}, where the whole duration of the
reactive path is neglected).

\section{Proofs of Lemmas~\ref{lem:return} and~\ref{lem:nonacc}}\label{s:prelim}

The proofs of Lemmas~\ref{lem:return} and~\ref{lem:nonacc} rely on the following nonattainability result for the set $\Gamma^0$, which is stated in~\cite[Proposition~2.7]{LelRamRey:kfp} in the case where $\metasp$ is bounded and $F$ is globally bounded and Lipschitz continuous on $\R^d$.

\begin{lem}[Nonattainability of $\Gamma^0$]\label{lem:nonatt}
  Under the assumptions of Lemma~\ref{lem:return}, for any $(q,p) \in (\R^d \times \R^d) \setminus \Gamma^0$,
  \begin{equation*}
    \Pr_{(q,p)}\left(\exists t \geq 0: (q_t,p_t) \in \Gamma^0\right) = 0.
  \end{equation*}
\end{lem}
\begin{proof}
  Let $\tau^0:=\inf\{t>0:(q_t,p_t)\in\Gamma^0\}$. It is sufficient to prove that for all $T>0$ and $(q,p) \in (\R^d \times \R^d) \setminus \Gamma^0$,
  \begin{equation*}
    \Pr_{(q,p)}\left(\tau^0\leq T\right) = 0.
  \end{equation*}
  Let $(F_k)_{k \geq 1}$ be a sequence of smooth compactly supported functions on $\R^d$ such that $F_k=F$ on $\mathrm{B}(0,k)$. Let $(\metasp_k)_{k\geq1}$ be a sequence of open, $C^2$ bounded sets of $\R^d$ such that $\metasp_k \cap \mathrm{B}(0,k) = \metasp \cap \mathrm{B}(0,k)$. Let $\Gamma^0_k:=\{(q,p)\in\partial \metasp_k \times \R^d: p \cdot \mathsf{n}_k(q)=0\}$ where $\mathsf{n}_k$ is the outward unitary vector to $\metasp_k$. Notice that $\Gamma^0_k \cap \mathrm{B}(0,k) = \Gamma^0 \cap \mathrm{B}(0,k)$. 
  
  Fix $(q,p) \in (\R^d \times \R^d) \setminus \Gamma^0$, $k > |q|$ and consider the process $(q_{k,t},p_{k,t})_{t\geq0}$ defined as the unique strong solution to~\eqref{eq:Langevin} with $F_k$ instead of $F$, driven by the same Brownian motion as $(q_t,p_t)_{t \geq 0}$ and with the same initial condition $(q,p)$. It is then a standard result on strong solutions that $(q_{k,t},p_{k,t})$ and $(q_t,p_t)$ coincide until the first time they exit $\mathrm{B}(0,k) \times \R^d$. Let 
  \begin{equation*}
    \tau^0_k := \inf\{t \geq 0: (q_{k,t},p_{k,t}) \in \Gamma^0_k\}.
  \end{equation*}
  By~\cite[Proposition~2.7]{LelRamRey:kfp}, one has that for all $T>0$, $\Pr_{(q,p)}(\tau^0_k\leq T)=0$, since $\metasp_k$ is a $C^2$ bounded set of $\R^d$ and $F_k$ is bounded and globally Lipschitz continuous.
  
  Let $T>0$. For $k\geq1$,
  \begin{equation*}
    \Pr_{(q,p)}\left(\tau^0 \leq T \right) = \Pr_{(q,p)}\left(\tau^0 \leq T, q_{\tau^0} \in \mathrm{B}(0,k)\right) +\Pr_{(q,p)}\left(\tau^0\leq T,q_{\tau^0}\notin \mathrm{B}(0,k)\right).
  \end{equation*}
  On the one hand, $\Pr_{(q,p)}(\tau^0\leq T,q_{\tau^0}\in \mathrm{B}(0,k))\leq\Pr_{(q,p)}(\tau^0_k\leq T)=0$, since $\Gamma^0_k \cap\mathrm{B}(0,k) = \Gamma^0 \cap \mathrm{B}(0,k)$. On the other hand, 
  \begin{equation*}
    \Pr_{(q,p)}\left(\tau^0\leq T,q_{\tau^0}\notin \mathrm{B}(0,k)\right)\leq\Pr_{(q,p)}\left(\sup_{t\in[0,T]}| q_t|\geq k\right)\underset{k\to\infty}{\longrightarrow}0,
  \end{equation*}
  which completes the proof.
\end{proof}

We may now present the proofs of Lemmas~\ref{lem:return} and~\ref{lem:nonacc}.

\begin{proof}[Proof of Lemma~\ref{lem:return}]
To prove the first part of Lemma~\ref{lem:return} on initial
conditions $(q,p) \in (\R^d \times \R^d) \setminus \Gamma^0$, we concentrate on the case when $q \in \metasp$ or
$(q,p) \in \Gamma^\bfen$: the case when $q \in \R^d \setminus
\bar{\metasp}$ or $(q,p) \in \Gamma^\bfex$ can be treated similarly
using the exterior sphere property instead of the interior sphere property.

Let us first prove that $\tau>0$. If $q \in \metasp$, this is
obvious. If $(q,p) \in \Gamma^\bfen$, this follows from the interior sphere property: Assumption~(\ref{ass:B1}) ensures that there exist $r>0$ and $q_\mathrm{int} \in \metasp$ such that $\mathrm{B}(q_\mathrm{int},r) \subset \metasp$ and $\bar{\mathrm{B}}(q_\mathrm{int},r) \cap (\R^d \setminus \metasp) = \{q\}$. Then it is clear that $q-q_\mathrm{int}=r\mathsf{n}(q)$, therefore using the fact that $q_t-q \sim tp$ when $t \to 0$, we get
  \begin{equation*}
    |q_t-q_\mathrm{int}|^2 = |q_t-q|^2 + 2 (q_t-q)\cdot(q-q_\mathrm{int}) + |q-q_\mathrm{int}|^2 = r^2 + t\left(2rp \cdot \mathsf{n}(q) + o(1)\right),
  \end{equation*}
  which implies that $q_t \in \mathrm{B}(q_\mathrm{int},r)$ for $t>0$
  small enough since $p \cdot \mathsf{n}(q)<0$, and therefore that $\tau > 0$. 
  
Let us now conclude the proof of the first part of Lemma~\ref{lem:return}. It follows from Assumption~(\ref{ass:B1}) that
  $\R^d \setminus \bar{\metasp}$ has positive Lebesgue measure and thus, by
  Assumptions~(\ref{ass:A1}--\ref{ass:A2}), it is a recurrent set. Therefore,
  almost surely, there exists $t>0$ such that $q_t \in \R^d \setminus
  \bar{\metasp}$, which by continuity of the trajectory $(q_t)_{t \geq 0}$
  implies that $\tau < +\infty$. Moreover, by the differentiability of this
  trajectory, we have $(q_\tau,p_\tau) \in \Gamma^\bfex \cup
  \Gamma^0$ and by Lemma~\ref{lem:nonatt}, we conclude that
  $(q_\tau,p_\tau) \in \Gamma^\bfex$. This concludes the proof of the first part of Lemma~\ref{lem:return}.
  
Let us now prove the second part of Lemma~\ref{lem:return}, concerning initial
conditions $(q,p) \in \Gamma^0$. We will reuse the notation $F_k$, $\metasp_k$, $(q_{k,t},p_{k,t)})_{t \geq 0}$ from the proof of Lemma~\ref{lem:nonatt}, and let $\metasp'_k$ be an open, $C^2$ bounded set of $\R^d$ such that $\metasp'_k \cap \mathrm{B}(0,k) = (\R^d \setminus \bar{\metasp}) \cap \mathrm{B}(0,k)$. For $k > |q|$, by~\cite[Proposition~2.8~(i)]{LelRamRey:kfp} we have, almost surely,
  \begin{equation*}
      \inf\{t \geq 0: q_{k,t} \in \metasp_k\} = \inf\{t \geq 0: q_{k,t} \in \metasp'_k\} = 0,
  \end{equation*}
  which by the continuity of the sample paths of $(q_{k,t})_{t \geq 0}$ and Lemma~\ref{lem:nonatt} implies that, almost surely,
  \begin{equation*}
      \inf\{t \geq 0: (q_{k,t},p_{k,t}) \in \Gamma_k^\bfen\} = \inf\{t \geq 0: (q_{k,t},p_{k,t}) \in \Gamma_k^\bfex\} = 0.
  \end{equation*}
  By the same localisation argument as at the end of the proof of Lemma~\ref{lem:nonatt}, we conclude that the same property holds for $(q_t,p_t)_{t \geq 0}$ in place of $(q_{k,t},p_{k,t})_{t \geq 0}$.
\end{proof}

\begin{proof}[Proof of Lemma~\ref{lem:nonacc}]
 Let us argue by contradiction and assume that either $\sup_{n \geq 0}
 \tau^\bfex_n = T < +\infty$ or $\sup_{n \geq 0} \tau^\bfen_n = T <
 +\infty$. Then Remark~\ref{rk:intertwining} shows that both sequences
 accumulate at the same time $T$. By the continuity of the trajectory of $(q_t,p_t)_{t \geq 0}$, we deduce that
  \begin{equation*}
    \lim_{n \to +\infty} Y^\bfex_n = \lim_{n \to +\infty} Y^\bfen_n = (q_T,p_T).
  \end{equation*}
  Since, by Assumption~(\ref{ass:B1}), the mapping $q \mapsto \mathsf{n}(q)$ is continuous on $\Sigma$, we then obtain
  \begin{equation*}
    \lim_{n \to +\infty} p_{\tau^\bfex_n} \cdot \mathsf{n}(q_{\tau^\bfex_n}) = \lim_{n \to +\infty} p_{\tau^\bfen_n} \cdot \mathsf{n}(q_{\tau^\bfen_n}) = p_T \cdot \mathsf{n}(q_T).
  \end{equation*}
  But since $p_{\tau^\bfex_n} \cdot \mathsf{n}(q_{\tau^\bfex_n}) > 0$
  while $p_{\tau^\bfen_n} \cdot \mathsf{n}(q_{\tau^\bfen_n}) < 0$, one
  gets as a consequence that $p_T \cdot \mathsf{n}(q_T) = 0$. In other
  words, there exists $T < +\infty$ such that $(q_T,p_T) \in
  \Gamma^0$, which by Lemma~\ref{lem:nonatt} has probability $0$ under
  $\Pr_{(q,p)}$, for any $(q,p)\in (\R^d \times \R^d) \setminus \Gamma^0$.
\end{proof}

\section{Proofs of Theorems~\ref{theo:main} and~\ref{theo:Y}, and of Proposition~\ref{prop:rev}}\label{s:ergo}

This section is organised as follows. In Subsection~\ref{ss:exist}, we
first show that the probability measures $\pi^\bfex$ and $\pi^\bfen$
defined in Theorem~\ref{theo:main} are invariant for the Markov chains
$(Y^\bfex_n)_{n \geq 0}$ and $(Y^\bfen_n)_{n \geq 0}$,
respectively. In Subsection~\ref{ss:HRYSigma}, we show that the Markov
chain $(Y^\Sigma_m)_{m \geq 0}$ defined in Theorem~\ref{theo:Y} is
Harris recurrent. These two results essentially yield all the necessary ingredients to complete the proofs of Theorems~\ref{theo:main} and~\ref{theo:Y}, which is carried out in Subsection~\ref{ss:pfmain}. Subsection~\ref{ss:rev} is then dedicated to the proof of Proposition~\ref{prop:rev}.

\subsection{Existence and identification of the invariant measure}\label{ss:exist} Let $\mathcal{L}$ be the infinitesimal generator of the solution to~\eqref{eq:Langevin}, which is defined on smooth functions $\phi : \R^d \times \R^d \to \R$ by
\begin{equation*}
  \mathcal{L}\phi = p \cdot \nabla_q \phi + F(q) \cdot \nabla_p \phi - \gamma p \cdot \nabla_p \phi + \gamma\beta^{-1}\Delta_p \phi.
\end{equation*}

The building block of the identification of $\pi^\bfex$ and $\pi^\bfen$ as invariant measures for $(Y^\bfex_n)_{n \geq 0}$ and $(Y^\bfen_n)_{n \geq 0}$ is the probabilistic interpretation of the Dirichlet problem associated with $\mathcal{L}$, presented in Proposition~\ref{prop:Dirichlet-Langevin} below. While similar statements are standard for elliptic diffusions, we were not able to find a proof of it in the literature which covers the case which we consider, namely with the degenerate operator $\mathcal{L}$ and the unbounded domain $\metasp \times \R^d$ in the phase space. Therefore we provide in Appendix~\ref{app:DL} a complete proof, partially based on our previous results~\cite{LelRamRey:kfp} on the kinetic Fokker--Planck equation.

In the next statement, we extend the definition of the stopping times $\tau^\bfen_0$ and $\tau^\bfex_0$ for initial conditions $(q,p) \in \Gamma^0$ by letting in this case $\tau^\bfen_0=\tau^\bfex_0=0$. According to the second part of Lemma~\ref{lem:return}, it remains true that $\tau^\bfen_0=\inf\{t \geq 0: (q_t,p_t) \in \Gamma^\bfen\}$ and $\tau^\bfex_0=\inf\{t \geq 0: (q_t,p_t) \in \Gamma^\bfex\}$.

\begin{prop}[Dirichlet problem]\label{prop:Dirichlet-Langevin}
  Let the assumptions of Theorem~\ref{theo:main} hold. 
  \begin{enumerate}[label=(\roman*),ref=\roman*] 
    \item\label{it:propDL:1} Let $f^\bfen: \Gamma^\bfen\cup\Gamma^0 \to \R$ be continuous and bounded, and 
    \begin{equation*}
      u^\bfen : (q,p) \in \R^d \times \R^d \mapsto \Exp_{(q,p)}\left[f^\bfen(q_{\tau^\bfen_0},p_{\tau^\bfen_0})\right].
    \end{equation*} 
    The function $u^\bfen$ is continuous on the closed set $(\R^d \setminus \metasp)\times \R^d$, $C^\infty$ on the open set $(\R^d \setminus \bar{\metasp}) \times \R^d$ and it satisfies
    \begin{equation}\label{eq:Dirichlet-en}
      \left\{\begin{aligned}
        \mathcal{L}u^\bfen & = 0 &\qquad \text{in $(\R^d \setminus \bar{\metasp}) \times \R^d$,}\\
        u^\bfen & = f^\bfen &\qquad\text{on $\Gamma^\bfen \cup \Gamma^0$.} 
      \end{aligned}\right.
    \end{equation}
    \item\label{it:propDL:2} Let $f^\bfex: \Gamma^\bfex\cup\Gamma^0 \to \R$ be continuous and bounded, and 
    \begin{equation*}
      u^\bfex : (q,p) \in \R^d \times \R^d \mapsto \Exp_{(q,p)}\left[f^\bfex(q_{\tau^\bfex_0},p_{\tau^\bfex_0})\right].
    \end{equation*} 
    The function $u^\bfex$ is continuous on the closed set $\bar{\metasp} \times \R^d$, $C^\infty$ on the open set $\metasp \times \R^d$ and it satisfies
    \begin{equation}\label{eq:Dirichlet-ex}
      \left\{\begin{aligned}
        \mathcal{L}u^\bfex & = 0 &\qquad \text{in $\metasp \times \R^d$,}\\
        u^\bfex & = f^\bfex &\qquad\text{on $\Gamma^\bfex \cup \Gamma^0$.} 
      \end{aligned}\right.
    \end{equation}
  \end{enumerate} 
\end{prop}

\begin{rk}[Uniqueness for the Dirichlet problem]
  Converse statements to Proposition~\ref{prop:Dirichlet-Langevin} also hold. Namely, if $u^\bfen$ is bounded, continuous on $((\R^d \setminus \bar{\metasp}) \times \R^d) \cup \Gamma^\bfen$, $C^2$ on $(\R^d \setminus \bar{\metasp}) \times \R^d$, and satisfies the Dirichlet problem
    \begin{equation*}
      \left\{\begin{aligned}
        \mathcal{L}u^\bfen & = 0 &\qquad \text{in $(\R^d \setminus \bar{\metasp}) \times \R^d$,}\\
        u^\bfen & = f^\bfen &\qquad\text{on $\Gamma^\bfen$,} 
      \end{aligned}\right.
    \end{equation*}
  then for any $(q,p) \in ((\R^d \setminus \bar{\metasp}) \times \R^d) \cup \Gamma^\bfen$, we have $u^\bfen(q,p)=\Exp_{(q,p)}[f^\bfen(q_{\tau^\bfen_0},p_{\tau^\bfen_0})]$. This statement is obvious for $(q,p) \in \Gamma^\bfen$. If $(q,p) \in (\R^d \setminus \bar{\metasp}) \times \R^d$, then by Itô's formula, we have for any $t \geq 0$,
  \begin{equation*}
      u^\bfen\left(q_{t \wedge \tau^\bfen_0},p_{t \wedge \tau^\bfen_0}\right) = u^\bfen(q,p) + \sqrt{2\gamma\beta^{-1}}\int_0^{t \wedge \tau^\bfen_0} \nabla_p u^\bfen(q_s,p_s) \cdot \dd W_s, 
  \end{equation*}
  and then
  \begin{equation*}
      u^\bfen(q,p) = \Exp_{(q,p)}\left[u^\bfen\left(q_{t \wedge \tau^\bfen_0},p_{t \wedge \tau^\bfen_0}\right)\right],
  \end{equation*}
  see~\cite[Section~3.1]{LelRamRey:kfp} for details. The conclusion then follows from the dominated convergence theorem, letting $t \to +\infty$ and using the fact that by Lemma~\ref{lem:return}, $\tau^\bfen_0 < +\infty$ and $(q_{\tau^\bfen_0},p_{\tau^\bfen_0}) \in \Gamma^\bfen$, almost surely.
  
  Of course, a similar statement holds for the Dirichlet problem~\eqref{eq:Dirichlet-ex}.
\end{rk}

Using Proposition~\ref{prop:Dirichlet-Langevin}, one can show the
following result concerning the invariance of the probability measures  $\pi^\bfex$ and $\pi^\bfen$.
\begin{prop}[Invariance of $\pi^\bfex$ and $\pi^\bfen$]\label{prop:exist}
  Let the assumptions of Theorem~\ref{theo:main} hold, and let~$\pi^\bfex$ and~$\pi^\bfen$ be the probability measures defined there.  
  \begin{enumerate}[label=(\roman*)] 
    \item If $(q_0,p_0) \sim \pi^\bfex$, then $(q_{\tau^\bfen_0},p_{\tau^\bfen_0}) \sim \pi^\bfen$.
    \item If $(q_0,p_0) \sim \pi^\bfen$, then $(q_{\tau^\bfex_0},p_{\tau^\bfex_0}) \sim \pi^\bfex$.
  \end{enumerate}
\end{prop}
Combined with Lemma~\ref{lem:return} and the strong Markov property,
Proposition~\ref{prop:exist} entails that the probability measures $\pi^\bfex$ and $\pi^\bfen$ are invariant for $(Y^\bfex_n)_{n \geq 0}$ and $(Y^\bfen_n)_{n \geq 0}$, respectively.

\begin{proof}
  We only prove the first point, the proof of the second point follows
  from symmetric arguments. We thus assume that $(q_0,p_0) \sim
  \pi^\bfex$, fix $f^\bfen: \Gamma^\bfen\cup\Gamma^0 \to \R$
  continuous and bounded, and define the function $u^\bfen$ on $\R^d
  \times \R^d$ as in Proposition~\ref{prop:Dirichlet-Langevin}.
  Recall that $\pi^\bfex$ and $\pi^\bfen$ respectively have densities
  $\varrho^\bfex$ and~$\varrho^\bfen$ (defined in~\eqref{eq:rho+-}) with respect to the measure $\dd\sigma_\Sigma(q)\dd p$.
 
\medskip \noindent
  \emph{Sketch of the argument.}   The idea of the proof relies on the
  following integration by parts formula
  \begin{equation}\label{eq:pf-exist:ipp}
    \begin{split}\int_{(\R^d \setminus \bar{\metasp}) \times \R^d} \mathcal{L}u^\bfen(q,p)\rho(q,p)\dd q \dd p &= \int_{\Sigma \times \R^d} p \cdot (-\mathsf{n}(q)) u^\bfen(q,p)\rho(q,p) \dd\sigma_\Sigma(q)\dd p\\
    &\quad + \int_{(\R^d \setminus \bar{\metasp}) \times \R^d} u^\bfen(q,p)\mathcal{L}^*\rho(q,p)\dd q \dd p,
    \end{split}
  \end{equation}
  where the differential operator $\mathcal{L}^*$ is the formal $L^2(\R^d \times \R^d,\dd q \dd p)$
  adjoint of $\mathcal L$ defined by
  \begin{equation}\label{eq:Lstar}
    \mathcal{L}^*\psi = -p \cdot \nabla_q \psi - \nabla_p \cdot (F(q)\psi) + \gamma \nabla_p \cdot (p \psi) + \gamma\beta^{-1}\Delta_p \psi,
  \end{equation}
  for smooth functions $\psi : \R^d \times \R^d \to \R$. Now, Proposition~\ref{prop:Dirichlet-Langevin} shows that $\mathcal{L}u^\bfen=0$ on $(\R^d
  \setminus \bar{\metasp}) \times \R^d$, while by Assumption~(\ref{ass:A2}), $\mathcal{L}^*\rho=0$ everywhere. Therefore, we deduce that
  \begin{equation}\label{eq:pf-exist:ipp2}
    \begin{split}
    0 &= \int_{\Sigma \times \R^d} p \cdot \mathsf{n}(q) u^\bfen(q,p)\rho(q,p) \dd\sigma_\Sigma(q)\dd p\\
    &= \int_{\Gamma^\bfen} p \cdot \mathsf{n}(q) u^\bfen(q,p)\rho(q,p) \dd\sigma_\Sigma(q)\dd p + \int_{\Gamma^\bfex} p \cdot \mathsf{n}(q) u^\bfen(q,p)\rho(q,p) \dd\sigma_\Sigma(q)\dd p.
    \end{split}
  \end{equation}
We will  show below that~\eqref{eq:pf-exist:ipp2} implies the result. However, since the domain of integration $(\R^d \setminus \bar{\metasp}) \times \R^d$ is not bounded, and the derivatives of $u^\bfen$ are generally known to blow up when $q$ approaches $\Sigma$~\cite{HwaJanVel14}, the integration by parts formula~\eqref{eq:pf-exist:ipp} requires some care. Thus, a detailed proof of~\eqref{eq:pf-exist:ipp2} is provided below.
  
\medskip \noindent
{\em Proof of Proposition~\ref{prop:exist} under the assumption that \eqref{eq:pf-exist:ipp2} holds.}
Let us first conclude the proof taking~\eqref{eq:pf-exist:ipp2} for granted.
By the definition of $u^\bfen$, $\pi^\bfex$ and $\Gamma^\bfex$, we have
  \begin{equation}\label{eq:pf-exist:1}
    \begin{aligned}
      \Exp_{\pi^\bfex}\left[f^\bfen\left(q_{\tau^\bfen_0},p_{\tau^\bfen_0}\right)\right] &= \int_{\Sigma \times \R^d} \Exp_{(q,p)}\left[f^\bfen\left(q_{\tau^\bfen_0},p_{\tau^\bfen_0}\right)\right]\pi^\bfex(\dd q\dd p)\\
      &= \frac{1}{Z^\bfex}\int_{\Gamma^\bfex} p \cdot \mathsf{n}(q)u^\bfen(q,p)\rho(q,p)\dd\sigma_\Sigma(q)\dd p.
    \end{aligned}
  \end{equation}
Using \eqref{eq:pf-exist:ipp2}, we then have
  \begin{equation*}
 Z^\bfex
 \Exp_{\pi^\bfex}\left[f^\bfen\left(q_{\tau^\bfen_0},p_{\tau^\bfen_0}\right)\right]=
 -     \int_{\Gamma^\bfen} p \cdot \mathsf{n}(q) u^\bfen(q,p)\rho(q,p) \dd\sigma_\Sigma(q)\dd p,
  \end{equation*}
  while by Proposition~\ref{prop:Dirichlet-Langevin}, 
  \begin{align*}
    \int_{\Gamma^\bfen} p \cdot \mathsf{n}(q) u^\bfen(q,p)\rho(q,p) \dd\sigma_\Sigma(q)\dd p &= \int_{\Gamma^\bfen} p \cdot \mathsf{n}(q) f^\bfen(q,p)\rho(q,p) \dd\sigma_\Sigma(q)\dd p\\
    &= - Z^\bfen \int_{\Sigma \times \R^d} f^\bfen(q,p) \pi^\bfen(\dd q\dd p). 
  \end{align*}
  Letting $f^\bfen \equiv 1$ shows that $Z^\bfen=Z^\bfex$, and we finally conclude that
  \begin{equation*}
    \Exp_{\pi^\bfex}\left[f^\bfen\left(q_{\tau^\bfen_0},p_{\tau^\bfen_0}\right)\right] = \int_{\Sigma \times \R^d} f^\bfen(q,p) \pi^\bfen(\dd q\dd p),
  \end{equation*}
  which implies that, under $\Pr_{\pi^\bfex}$,
  $(q_{\tau^\bfen_0},p_{\tau^\bfen_0}) \sim \pi^\bfen$.

\medskip  \noindent
 \emph{Proof of~\eqref{eq:pf-exist:ipp2}.} Let $M \geq 0$ and $\mathscr{U}_M$ be a bounded, $C^2$ and open subset of $\R^d$ such that $\mathscr{U}_M \cap \mathrm{B}(0,M+1) = (\R^d \setminus \bar{\metasp}) \cap \mathrm{B}(0,M+1)$. Notice that in particular, $\Sigma_M := \partial \mathscr{U}_M$ and $\Sigma$ coincide on $\mathrm{B}(0,M+1)$, and on this ball, the normal vector $\mathsf{n}_M(q)$ to $\Sigma_M$ pointing toward the interior of $\mathscr{U}_M$ coincides with $\mathsf{n}(q)$.
  
  By~\cite[Lemma~14.16, p.~355]{GilTru01}, there exists $\alpha_0>0$, which depends on $M$, such that the Euclidean distance function to $\Sigma_M$, which we denote by $\mathsf{d}_{\Sigma_M}$, is $C^2$ on the set $\{q \in \bar{\mathscr{U}}_M : \mathsf{d}_{\Sigma_M}(q) < \alpha_0\}$, and it satisfies the eikonal equation $|\nabla \mathsf{d}_{\Sigma_M}(q)|=1$ there. For any $\alpha \in (0,\alpha_0)$, let 
  \begin{equation*}
      \mathscr{U}_{M,\alpha} := \{q \in \mathscr{U}_M : \mathsf{d}_{\Sigma_M}(q) > \alpha\}.
  \end{equation*}
  The set $\mathscr{U}_{M,\alpha}$ is open, bounded and by the implicit function theorem, its boundary is $C^2$. 
  
  Let $\iota_M : \R^d \times \R^d \to [0,1]$ be a $C^\infty$ function such that
  \begin{equation*}
      \iota_M(q,p) = \begin{cases}
        1 & \text{if $(q,p) \in \mathrm{B}(0,M) \times \mathrm{B}(0,M)$,}\\
        0 & \text{if $(q,p) \not\in \mathrm{B}(0,M+1) \times \mathrm{B}(0,M+1)$,}
      \end{cases}
  \end{equation*}
  and such that $\nabla_q \iota_M$, $\nabla_p \iota_M$ and $\Delta_q \iota_M$ are bounded on $\R^d \times \R^d$, uniformly in $M$.
  
  For any $(q,p) \in \mathscr{U}_{M,\alpha} \times \R^d$, we obviously have
  \begin{equation*}
      \iota_M(q,p)\rho(q,p) \mathcal{L}u^\bfen(q,p) = 0,
  \end{equation*}
  by Proposition~\ref{prop:Dirichlet-Langevin} and the construction of $\iota_M$. Since, by Proposition~\ref{prop:Dirichlet-Langevin} again and the construction of the set $\mathscr{U}_{M,\alpha}$, $u^\bfen$ is $C^\infty$ on the closure of the bounded set $\mathscr{U}_{M,\alpha} \times \mathrm{B}(0,M+1)$, all differential terms in $\mathcal{L}u^\bfen$ can be integrated by parts on $\mathscr{U}_{M,\alpha} \times \R^d$ to yield the identity
  \begin{equation}\label{eq:pf-exist:ipp-det:1}
      0 = \int_{\mathscr{U}_{M,\alpha} \times \R^d} \iota_M(q,p)\rho(q,p) \mathcal{L}u^\bfen(q,p) \dd q\dd p = \mathrm{I} + \mathrm{II} + \mathrm{III} + \mathrm{IV}, 
  \end{equation}
  with
  \begin{align*}
      \mathrm{I} &:= \int_{\Sigma_{M,\alpha} \times \R^d} \iota_M(q,p)\rho(q,p) u^\bfen(q,p) (-p \cdot \mathsf{n}_{M,\alpha}(q)) \dd \sigma_{\Sigma_{M,\alpha}}(q)\dd p,\\
      \mathrm{II} &:= \int_{\mathscr{U}_{M,\alpha} \times \R^d} \iota_M(q,p)u^\bfen(q,p)\mathcal{L}^*\rho(q,p)\dd q\dd p,\\
      \mathrm{III} &:= \int_{\mathscr{U}_{M,\alpha} \times \R^d} u^\bfen(q,p)\rho(q,p)\left\{-p \cdot \nabla_q \iota_M - (F(q)-\gamma p) \cdot \nabla_p \iota_M + \gamma\beta^{-1} \Delta_p \iota_M\right\}\dd q\dd p,\\
      \mathrm{IV}&:= 2\gamma\beta^{-1}\int_{\mathscr{U}_{M,\alpha} \times \R^d} u^\bfen(q,p)\nabla_p \rho(q,p) \cdot \nabla_p \iota_M(q,p)\dd q\dd p,
  \end{align*}
  where, in $\mathrm{I}$, $\mathsf{n}_{M,\alpha}(q)$ is the normal vector to $\Sigma_{M, \alpha} := \partial\mathscr{U}_{M,\alpha}$ pointing toward the interior of $\mathscr{U}_{M,\alpha}$ while $\dd \sigma_{\Sigma_{M,\alpha}}(q)$ is the surface measure thereon.
  
  Since $\rho$ is the invariant measure, $\mathcal{L}^*\rho=0$ and thus $\mathrm{II}=0$. Since the integrand in $\mathrm{III}$ and $\mathrm{IV}$ is bounded in $\mathscr{U}_M \times \R^d$ and vanishes outside a bounded set, it follows from the dominated convergence theorem that
  \begin{equation}\label{eq:pf-exist:ipp-det:2}
  \begin{aligned}
      \lim_{\alpha \dto 0} \mathrm{III} &= \int_{\mathscr{U}_M \times \R^d} u^\bfen(q,p)\rho(q,p)\left\{-p \cdot \nabla_q \iota_M - (F(q)-\gamma p) \cdot \nabla_p \iota_M + \gamma\beta^{-1} \Delta_p \iota_M\right\}\dd q\dd p,\\
      \lim_{\alpha \dto 0} \mathrm{IV} &= 2\gamma\beta^{-1}\int_{\mathscr{U}_M \times \R^d} u^\bfen(q,p)\nabla_p \rho(q,p) \cdot \nabla_p \iota_M(q,p)\dd q\dd p.
  \end{aligned}
  \end{equation}
  To complete the proof, we now show that 
  \begin{equation}\label{eq:pf-exist:sigma-alpha}
      \lim_{\alpha \dto 0} \mathrm{I} = \int_{\Sigma_M \times \R^d} \iota_M(q,p)\rho(q,p) u^\bfen(q,p) (-p \cdot \mathsf{n}_M(q)) \dd \sigma_{\Sigma_M}(q)\dd p,
  \end{equation}
  and then conclude by letting $M \to +\infty$.
  
\medskip \noindent
  \emph{Proof of~\eqref{eq:pf-exist:sigma-alpha}.} Since $u^\bfen$ is continuous and bounded on the closed set $(\R^d \setminus \metasp) \times \R^d$, it may be extended to a continuous and bounded function on $\R^d \times \R^d$, which we still denote by $u^\bfen$. The resulting function $g$ defined on $\R^d \times \R^d$ by
  \begin{equation*}
      g(q,p) = \iota_M(q,p)\rho(q,p)u^\bfen(q,p)
  \end{equation*}
  is then continuous and compactly supported. In particular, it is bounded and uniformly continuous on $\R^d \times \R^d$. As a consequence, it can be mollified so as to construct a family $g_\epsilon$ of $C^\infty$ functions with compact support, which converge to $g$ uniformly when $\epsilon \to 0$.
  
  For $\epsilon>0$, by the divergence theorem,
  \begin{align*}
      \int_{\Sigma_{M,\alpha} \times \R^d} g_\epsilon(q,p) p \cdot \mathsf{n}_{M,\alpha}(q)\dd \sigma_{\Sigma_{M,\alpha}}(q)\dd p &= \int_{\mathscr{U}_{M,\alpha} \times \R^d} p \cdot \nabla_q g_\epsilon(q,p) \dd q \dd p\\
      &= \int_{\mathscr{U}_M \times \R^d} \ind{\mathsf{d}_{\Sigma_M}(q)>\alpha} p \cdot \nabla_q g_\epsilon(q,p) \dd q \dd p,
  \end{align*}
  and since $\nabla_q g_\epsilon$ is globally bounded and compactly supported, by the dominated convergence theorem, the right-hand side converges to
  \begin{equation*}
      \int_{\mathscr{U}_M \times \R^d} p \cdot \nabla_q g_\epsilon(q,p) \dd q \dd p = \int_{\Sigma_M \times \R^d} g_\epsilon(q,p) p \cdot \mathsf{n}_M(q)\dd \sigma_{\Sigma_M}(q)\dd p
  \end{equation*}
  when $\alpha \dto 0$. As a consequence, for any $\epsilon>0$, we have
  \begin{align*}
      &\limsup_{\alpha \dto 0}\left|\int_{\Sigma_{M,\alpha} \times \R^d} g(q,p) p \cdot \mathsf{n}_{M,\alpha}(q)\dd \sigma_{\Sigma_{M,\alpha}}(q)\dd p - \int_{\Sigma_M \times \R^d} g_\epsilon(q,p) p \cdot \mathsf{n}_M(q)\dd \sigma_{\Sigma_M}(q)\dd p\right|\\
      &\leq \limsup_{\alpha \dto 0} \|g-g_\epsilon\|_\infty\left(\sigma_{\Sigma_{M,\alpha}}(\R^d)+\sigma_{\Sigma_M}(\R^d)\right)\int_{\mathrm{B}(0,M+1)}|p|\dd p.
  \end{align*}
  Therefore, to complete the proof of~\eqref{eq:pf-exist:sigma-alpha}, it remains to check that $\limsup_{\alpha \dto 0} \sigma_{\Sigma_{M,\alpha}}(\R^d) < +\infty$. In fact, we shall prove the more precise  result that
  \begin{equation*}
      \lim_{\alpha \dto 0} \sigma_{\Sigma_{M,\alpha}}(\R^d) = \sigma_{\Sigma_M}(\R^d).
  \end{equation*}
  To do so, we observe that the proof of~\cite[Lemma~14.16, p.~355]{GilTru01} entails the identity
  \begin{equation*}
      \mathsf{n}_{M,\alpha}(q) = \nabla\mathsf{d}_{\Sigma_M}(q)
  \end{equation*}
  for any $\alpha < \alpha_0$ and $q \in \Sigma_{M,\alpha}$. As a consequence, letting $C_{M,\alpha} := \mathscr{U}_M \setminus \mathscr{U}_{M,\alpha}$, we get
  \begin{align*}
      \sigma_{\Sigma_{M,\alpha}}(\R^d) - \sigma_{\Sigma_M}(\R^d) &= \int_{\Sigma_{M,\alpha}} \nabla\mathsf{d}_{\Sigma_M}(q) \cdot \mathsf{n}_{M,\alpha}(q) \dd \sigma_{\Sigma_{M,\alpha}}(q) - \int_{\Sigma_M} \nabla\mathsf{d}_{\Sigma_M}(q) \cdot \mathsf{n}_M(q) \dd \sigma_{\Sigma_M}(q)\\
      &= \int_{C_{M,\alpha}} \Delta \mathsf{d}_{\Sigma_M}(q) \dd q,
  \end{align*}
  and the right-hand side vanishes when $\alpha \dto 0$, because
  $\Delta \mathsf{d}_{\Sigma_M}(q)$ is bounded on
  $\{\mathsf{d}_{\Sigma_M}(q) < \alpha_0\}$. This completes the proof
  of~\eqref{eq:pf-exist:sigma-alpha}.

\medskip \noindent
  \emph{Conclusion of the proof of~\eqref{eq:pf-exist:ipp2}.} Putting together~\eqref{eq:pf-exist:ipp-det:1}, \eqref{eq:pf-exist:ipp-det:2} and~\eqref{eq:pf-exist:sigma-alpha}, we obtain the identity
  \begin{align*}
      &\int_{\Sigma_M \times \R^d} \iota_M(q,p)\rho(q,p) u^\bfen(q,p) p \cdot \mathsf{n}_M(q) \dd \sigma_{\Sigma_M}(q)\dd p\\
      &= \int_{\mathscr{U}_M \times \R^d} u^\bfen(q,p)\rho(q,p)\left\{-p \cdot \nabla_q \iota_M - (F(q)-\gamma p) \cdot \nabla_p \iota_M + \gamma\beta^{-1} \Delta_p \iota_M\right\}\dd q\dd p\\
      &\quad + 2\gamma\beta^{-1}\int_{\mathscr{U}_M \times \R^d} u^\bfen(q,p)\nabla_p \rho(q,p) \cdot \nabla_p \iota_M(q,p)\dd q\dd p,
  \end{align*}
  which, by the properties of the set $\mathscr{U}_M$ and of the function $\iota_M$, rewrites
  \begin{align*}
      &\int_{\Sigma \times \R^d} \iota_M(q,p)\rho(q,p) u^\bfen(q,p) p \cdot \mathsf{n}(q) \dd \sigma_{\Sigma}(q)\dd p\\
      &= \int_{(\R^d \setminus \bar{\metasp}) \times \R^d} u^\bfen(q,p)\rho(q,p)\left\{-p \cdot \nabla_q \iota_M - (F(q)-\gamma p) \cdot \nabla_p \iota_M + \gamma\beta^{-1} \Delta_p \iota_M\right\}\dd q\dd p\\
      &\quad + 2\gamma\beta^{-1}\int_{(\R^d \setminus \bar{\metasp}) \times \R^d} u^\bfen(q,p)\nabla_p \rho(q,p) \cdot \nabla_p \iota_M(q,p)\dd q\dd p.
  \end{align*}
  Since $u^\bfen$ is bounded, it follows from Assumption~(\ref{ass:C}) and the dominated convergence theorem that the left-hand side converges to
  \begin{equation*}
      \int_{\Sigma \times \R^d} \rho(q,p) u^\bfen(q,p) p \cdot \mathsf{n}(q) \dd \sigma_{\Sigma}(q)\dd p
  \end{equation*}
  when $M \to +\infty$. Besides, all terms $\nabla_q \iota_M$,
  $\nabla_p \iota_M$ and  $\Delta_p \iota_M$ converge pointwise to $0$
  while remaining bounded uniformly in $M$. Therefore, by
  Assumption~(\ref{ass:A3}), the boundedness of $u^\bfen$ and the
  dominated convergence theorem again, the right-hand side converges
  to $0$ when $M \to +\infty$. This completes the proof
  of~\eqref{eq:pf-exist:ipp2}, and thus of Proposition~\ref{prop:exist}. 
\end{proof}

We complete this subsection by stating an energy estimate on the functions $u^\bfen$ and $u^\bfex$, which will not be used before Subsection~\ref{ss:rev} but follows from the same arguments as the proof of Proposition~\ref{prop:exist}.

\begin{lem}[Energy estimates]\label{lem:energy}
  In the setting of Proposition~\ref{prop:Dirichlet-Langevin}, we have
  \begin{equation*}
      \int_{(\R^d \setminus \bar{\metasp}) \times \R^d} |\nabla_p u^\bfen(q,p)|^2 \rho(q,p)\dd q \dd p \leq \frac{Z^\bfen+Z^\bfex}{2\gamma\beta^{-1}} \|f^\bfen\|_\infty^2 < +\infty,
  \end{equation*}
  and
  \begin{equation*}
      \int_{\metasp \times \R^d} |\nabla_p u^\bfex(q,p)|^2 \rho(q,p)\dd q \dd p \leq \frac{Z^\bfen+Z^\bfex}{2\gamma\beta^{-1}} \|f^\bfex\|_\infty^2 < +\infty.
  \end{equation*}
\end{lem}
\begin{proof}
  We only prove the estimate on $u^\bfen$, the estimate on $u^\bfex$ follows from the same arguments. By Proposition~\ref{prop:Dirichlet-Langevin}, we have
  \begin{equation*}
      \mathcal{L}((u^\bfen)^2) = 2u^\bfen \mathcal{L}u^\bfen + 2 \gamma\beta^{-1}|\nabla_p u^\bfen|^2 = 2 \gamma\beta^{-1}|\nabla_p u^\bfen|^2
  \end{equation*}
  on $(\R^d \setminus \bar{\metasp}) \times \R^d$. With the notation of the proof of Proposition~\ref{prop:exist}, we deduce that
  \begin{equation*}
      \int_{\mathcal{U}_{M,\alpha} \times \R^d} \iota_M(q,p)\mathcal{L}((u^\bfen)^2)(q,p) \rho(q,p) \dd q \dd p = 2 \gamma\beta^{-1}\int_{\mathcal{U}_{M,\alpha} \times \R^d}\iota_M(q,p)|\nabla_p u^\bfen(q,p)|^2\rho(q,p) \dd q \dd p.
  \end{equation*}
  By the exact same arguments as in the proof of Proposition~\ref{prop:exist}, which only rely on the boundedness and the continuity of $u^\bfen$, the left-hand side of this identity can be decomposed as the sum of four terms $\mathrm{I}$, $\mathrm{II}$, $\mathrm{III}$ and $\mathrm{IV}$ resulting from integration by parts, and then shown to satisfy
  \begin{align*}
      &\lim_{M \to +\infty} \lim_{\alpha \dto 0} \int_{\mathcal{U}_{M,\alpha} \times \R^d} \iota_M(q,p)\mathcal{L}((u^\bfen)^2)(q,p) \rho(q,p) \dd q \dd p= \int_{\Sigma \times \R^d} (-p \cdot \mathsf{n}(q))u^\bfen(q,p)^2\rho(q,p)\dd\sigma_\Sigma(q)\dd p,
  \end{align*}
  which by Assumption~(\ref{ass:C}) and the boundedness of $u^\bfen$ is well-defined and satisfies
  \begin{align*}
      \left|\int_{\Sigma \times \R^d} (-p \cdot \mathsf{n}(q))u^\bfen(q,p)^2\rho(q,p)\dd\sigma_\Sigma(q)\dd p\right| &\leq \|f^\bfen\|_\infty^2 \int_{\Sigma \times \R^d} |p \cdot \mathsf{n}(q)|\rho(q,p)\dd\sigma_\Sigma(q)\dd p\\
      &= \|f^\bfen\|_\infty^2(Z^\bfen+Z^\bfex).
  \end{align*}
  On the other hand, it directly follows from the definition of $\mathcal{U}_{M,\alpha}$ and $\iota_M$ that 
  \begin{equation*}
      \lim_{M \to +\infty} \lim_{\alpha \dto 0} \int_{\mathcal{U}_{M,\alpha} \times \R^d}\iota_M(q,p)|\nabla_p u^\bfen(q,p)|^2\rho(q,p) \dd q \dd p = \int_{(\R^d \setminus \bar{\metasp}) \times \R^d}|\nabla_p u^\bfen(q,p)|^2\rho(q,p) \dd q \dd p,
  \end{equation*}
  which yields the conclusion.
\end{proof}

\subsection{Harris recurrence of the chain \texorpdfstring{$(Y^\Sigma_m)_{m \geq 0}$}{YSigma}}\label{ss:HRYSigma} In this subsection, we prove the following result.

\begin{prop}[Harris recurrence of the chain $(Y^\Sigma_m)_{m \geq 0}$]\label{prop:HRYSigma}
  Under the assumptions of Theorem~\ref{theo:main}, the Markov chain $(Y^\Sigma_m)_{m \geq 0}$ defined in Theorem~\ref{theo:Y} is Harris recurrent.
\end{prop}

By Lemma~\ref{lem:return}, the chain $(Y^\Sigma_m)_{m \geq 0}$ takes its values in the set
\begin{equation*}
  \mathcal{S} := (\Sigma \times \R^d) \setminus \Gamma^0 = \Gamma^\bfex \cup \Gamma^\bfen.
\end{equation*}
This set is endowed with the trace on $\mathcal S$ of the Borel $\sigma$-algebra of $\R^d \times \R^d$, which makes it a separable measurable space (see Appendix~\ref{app:Harris}). To prove Proposition~\ref{prop:HRYSigma}, we construct a subset of $\mathcal{S}$ which is a \emph{regeneration set}, in the sense of Definition~\ref{defi:reg-set}, for this chain. 

To proceed, we first fix $q^* \in \Sigma$. By Assumption~(\ref{ass:B1}), the exterior sphere property ensures that there exist $q_\mathrm{ext} \in \R^d$ and $r^*>0$ such that $\mathrm{B}(q_\mathrm{ext},r^*) \subset (\R^d \setminus \bar{\metasp})$ and $\bar{\mathrm{B}}(q_\mathrm{ext},r^*) \cap \bar{\metasp} = \{q^*\}$. We set $p^* = q_\mathrm{ext}-q^*$ and remark that $p^*=r^*\mathsf{n}(q^*)$, so that $(q^*,p^*) \in \Gamma^\bfex$. For any $\delta_q>0$ and $\delta_p>0$, we now set (see Figure~\ref{fig:lemma22} below for a schematic representation of the objects introduced in this section)
\begin{equation*}
  \mathcal{R}_{\delta_q,\delta_p} := \{(q,p) \in \Gamma^\bfex: |q^*-q| \leq \delta_q, |p^*-p| \leq \delta_p\}.
\end{equation*}


Proposition~\ref{prop:HRYSigma} simply follows from the combination of
Lemmas~\ref{lem:Harris:1} and~\ref{lem:Harris:2}.

\begin{lem}[Minorisation condition]\label{lem:Harris:1}
  Under the assumptions of Proposition~\ref{prop:HRYSigma}, there exist $\delta_q>0$, $\delta_p>0$, $\epsilon>0$ and a probability measure $\lambda$ on $\mathcal{S}$ such that, for any $(q,p) \in \mathcal{R}_{\delta_q,\delta_p}$, $\Pr_{(q,p)}(Y^\Sigma_1 \in \cdot) \geq \epsilon \lambda(\cdot)$.
\end{lem}
\begin{proof}
  By Lemma~\ref{lem:return}, for any $(q,p) \in \Gamma^\bfex$, we have
  $Y^\Sigma_1 \in \Gamma^\bfen $, $\Pr_{(q,p)}$-almost surely. Therefore, by a
  monotone class argument, it is enough to prove the existence of
  $\delta_q$, $\delta_p$, $\epsilon$ and $\lambda$ such that, for any
  continuous and bounded function $f^\bfen : \Gamma^\bfen \cup
  \Gamma^0 \to [0,+\infty)$, 
  \begin{equation*}
    \forall (q,p) \in \mathcal{R}_{\delta_q,\delta_p}, \qquad u^\bfen(q,p) := \Exp_{(q,p)}\left[f^\bfen(Y^\Sigma_1)\right] \geq \epsilon \lambda(f^\bfen).
  \end{equation*}
  Such a statement typically follows from Harnack inequalities associated with the operator $\mathcal{L}$. Indeed, by Proposition~\ref{prop:Dirichlet-Langevin}, the function $u^\bfen$ satisfies $\mathcal{L}u^\bfen=0$ on $(\R^d \setminus \bar{\metasp})\times \R^d$. 
  Besides, the Harnack inequality stated in~\cite[Theorem~2.15]{LelRamRey:kfp}, which is based on previous results for kinetic equations from~\cite{GolImbMouVas19}, asserts that if $\mathcal{L}u^\bfen=0$ on some subset $\metasp' \times \R^d$ of $\R^d \times \R^d$, with $\metasp'$ a bounded, connected, $C^2$ open subset of $\R^d$, then for any compact subset $K$ of $\metasp' \times \R^d$ there exists a constant $\epsilon>0$, which does not depend on $f^\bfen$, such that
  \begin{equation*}
    \inf_{(q,p) \in K} u^\bfen(q,p) \geq \epsilon \sup_{(q,p) \in K} u^\bfen(q,p),
  \end{equation*}
  and thus one may take $\lambda(\cdot) = \Pr_{(q,p)}(Y^\Sigma_1 \in
  \cdot)$ for any choice of a fixed $(q,p) \in K$. 
    
  However, this result cannot be applied directly here, because one
  would want to take for $K$ the set
  $\mathcal{R}_{\delta_q,\delta_p}$, which is a subset of $\Sigma
  \times \R^d$. But since the identity $\mathcal{L}u^\bfen=0$ is only
  satisfied on $(\R^d \setminus \bar{\metasp})\times \R^d$, it does
  not hold on cylindrical sets of the form $\metasp' \times \R^d$ with
  $\metasp' \cap \Sigma \not= \emptyset$. Therefore, before applying
  the Harnack inequality, a preliminary work is needed to show that,
  starting from $\mathcal{R}_{\delta_q,\delta_p}$, $(q_t,p_t)$ reaches
  a compact subset $K \subset \metasp' \times \R^d$, with $\metasp'
  \subset \R^d \setminus \bar{\metasp}$, before returning to $\metasp
  \times \R^d$ (and thus to $\Sigma \times \R^d$), with a probability which is uniformly bounded from below. This work is carried out in Steps~1 and~2. The conclusion of the proof, with the application of the Harnack inequality, is detailed in Step~3.

\medskip \noindent  \emph{Preparatory material.} Let us denote by $\mathsf{d}_\Sigma$ the signed distance function to $\Sigma$ in $\R^d$, with the convention that $\mathsf{d}_\Sigma(q)>0$ if $q \in \metasp$ and $\mathsf{d}_\Sigma(q)<0$ if $q \in \R^d\setminus\bar{\metasp}$. We recall that by Assumption~(\ref{ass:B1}), there is a neighbourhood of $\Sigma$ on which $\mathsf{d}_\Sigma$ is $C^2$, and for any $q \in \Sigma$ we have the identity $\mathsf{n}(q)=-\nabla \mathsf{d}_\Sigma(q)$. As a consequence, there exist $\delta'_q>0$ and $\delta'_p>0$ such that
  \begin{equation}\label{eq:pf-Harris1:0}
    \forall (q,p) \in \bar{\mathrm{B}}(q^*,\delta'_q) \times \bar{\mathrm{B}}(p^*,\delta'_p), \qquad p \cdot (-\nabla \mathsf{d}_\Sigma(q)) > 0.
  \end{equation} 
  We now fix the following positive quantities:
  \begin{equation*}
    t^* := 1 \wedge \frac{2\delta'_q}{3r^*}, \qquad \delta_q := \frac{t^*r^*}{6}, \qquad \delta_p := \frac{r^*}{6} \wedge \frac{\delta'_p}{2}, \qquad \delta'' := \frac{t^*r^*}{6} \wedge \frac{\delta'_p}{2},
  \end{equation*}
  and set
  \begin{equation*}
    K := \bar{\mathrm{B}}\left(q^*+t^*p^*, \frac{t^*r^*}{2}\right) \times \bar{\mathrm{B}}\left(p^*,\delta'_p\right).
  \end{equation*}
Notice that, thanks to the definitions of
$(t^*,\delta_q,\delta_p)$, one has 
$K \subset   \bar{\mathrm{B}}(q^*,\delta'_q) \times
\bar{\mathrm{B}}(p^*,\delta'_p)$ and $\mathcal{R}_{\delta_q,\delta_p}
\subset (\bar{\mathrm{B}}(q^*,\delta'_q) \times
\bar{\mathrm{B}}(p^*,\delta'_p)) \cap (\Sigma \times \R^d)$. We refer to
Figure~\ref{fig:lemma22} for a schematic representation of the
geometric setting in the position space.

\begin{psfrags}
  \psfrag{O}{$\mathcal O$}
  \psfrag{dq'}{$\delta_{q'}$}
  \psfrag{qext}{$q_\mathrm{ext}$}
  \psfrag{q*}{$q^*$}
  \psfrag{p*}{$p^*$}
  \psfrag{r*}{$r^*$}
  \psfrag{K}{\color{green}$\Pi_q(K)$}
  \psfrag{S}{$\Sigma$}
  \psfrag{En}{\color{blue}$\Pi_q(E_n)$}
  \psfrag{R}{\color{red}$\Pi_q(\mathcal{R}_{\delta_q,\delta_p})$}
\begin{figure}
\begin{center}
    \includegraphics[width=0.5\textwidth]{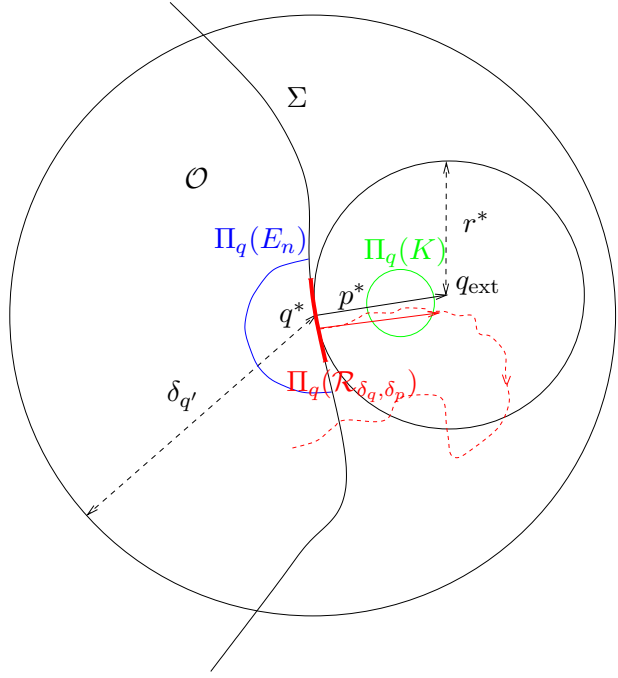}
\caption{Schematic representation of the objects introduced in the
  proofs of the Harris recurrence of the chain $(Y^\Sigma_m)_{m \ge
    0}$. Here, $\Pi_q(q,p)=q$ denotes the projection on the position space. Typically, a trajectory (dashed red line) starting in $\mathcal{R}_{\delta_q,\delta_p}$ goes through~$K$ before crossing $\Sigma$.}
\label{fig:lemma22}
\end{center}
\end{figure}
\end{psfrags}


\medskip \noindent    \emph{Step~1.} We  prove in Steps 1 and 2 that there exists $\alpha>0$ such that
  \begin{equation}\label{eq:pfHarris1:alpha:1}
    \forall (q,p) \in \mathcal{R}_{\delta_q,\delta_p}, \qquad \Pr_{(q,p)}((q_{t^*},p_{t^*}) \in K, \tau^\bfen_0 > t^*) \geq \alpha.
  \end{equation}
  To proceed, we first let $\hat{F} : \R^d \to \R^d$ be a $C^\infty$, bounded and globally Lipschitz continuous function such that $F(q)=\hat{F}(q)$ for any $q \in \bar{\mathrm{B}}(q^*,\delta'_q)$. We denote by $(\hat{q}_t,\hat{p}_t)_{t \geq 0}$ the unique strong solution to~\eqref{eq:Langevin} with $\hat{F}$ in place of $F$; driven by the same Brownian motion $(W_t)_{t \geq 0}$ and started with the same initial condition $(q,p) \in \mathcal{R}_{\delta_q,\delta_p}$ as $(q_t,p_t)_{t \geq 0}$, so that $(q_t,p_t)$ and $(\hat{q}_t,\hat{p}_t)$ coincide until they leave the set $\bar{\mathrm{B}}(q^*,\delta'_q) \times \R^d$.
  
  For any $t \in [0,t^*]$,
  \begin{align*}
    |\hat{q}_t - (q+tp)| + |\hat{p}_t-p| &= \left|\int_0^t \left(\hat{p}_s-p\right)\dd s \right| + \left|\int_0^t \left(\hat{F}(\hat{q}_s)-\gamma\hat{p}_s\right)\dd s + \sqrt{2\gamma\beta^{-1}}W_t\right|\\
    &\leq C \int_0^t \left(|\hat{q}_s - (q+sp)| + |\hat{p}_s-p|\right)\dd s + S_{q,p},
  \end{align*}
  where the constant $C \geq 0$ depends on $\gamma$ and the Lipschitz constant of $\hat{F}$, and 
  \begin{equation*}
    S_{q,p} := \sup_{t \in [0,t^*]} \left|\int_0^t \left(\hat{F}(q+sp)-\gamma p\right)\dd s + \sqrt{2\gamma\beta^{-1}}W_t\right|.
  \end{equation*}
  We deduce from the Gronwall lemma that
  \begin{equation*}
    \sup_{t \in [0,t^*]} \left\{|\hat{q}_t - (q+tp)| + |\hat{p}_t-p|\right\} \leq S_{q,p} \ee^{Ct^*}.
  \end{equation*}
  
  On the event $\{S_{q,p} \leq \delta'' \ee^{-Ct^*}\}$, we therefore
  have for any $t \in [0,t^*]$ (since $|p| \le r^* + \delta_p$),
  \begin{align*}
    |\hat{q}_t-q^*| &\leq |\hat{q}_t-(q+tp)| + |q+tp-q^*| \leq \delta'' + \delta_q + t^*(r^* + \delta_p) \leq \delta'_q,\\
    |\hat{p}_t-p^*| &\leq |\hat{p}_t-p| + |p-p^*| \leq \delta'' + \delta_p \leq \delta'_p,
  \end{align*}
  which by~\eqref{eq:pf-Harris1:0} ensures that $\hat{q}_t=q_t$, $\hat{p}_t=p_t$, and $p_t \cdot (-\nabla \mathsf{d}_\Sigma(q_t)) > 0$ for $t\in[0,t^*]$, so that $\tau^\bfen_0 > t^*$. Furthermore,
  \begin{equation*}
    |q_{t^*} - (q^* + t^* p^*)| \leq |\hat{q}_{t^*}-(q+t^*p)| + |(q+t^*p)-(q^*+t^*p^*)| \leq \delta'' + \delta_q + t^* \delta_p \leq \frac{t^*r^*}{2},
  \end{equation*}
  which implies that $(q_{t^*},p_{t^*}) \in K$. As a consequence, the estimate~\eqref{eq:pfHarris1:alpha:1} follows from the fact that
  \begin{equation}\label{eq:pfHarris1:alpha:2}
    \alpha := \inf_{(q,p) \in \mathcal{R}_{\delta_q,\delta_p}} \Pr\left(S_{q,p} \leq \delta'' \ee^{-Ct^*}\right) > 0,
  \end{equation}
  which is proved in Step~2. 
  
\medskip \noindent    \emph{Step~2.} We now prove the estimate~\eqref{eq:pfHarris1:alpha:2}. In this step we simply write 
  \begin{equation*}
    \delta := \delta'' \ee^{-Ct^*}, \qquad h_t := \frac{1}{\sqrt{2\gamma \beta^{-1}}}(\hat{F}(q+tp)-\gamma p) \in \R^d,
  \end{equation*}
  and remark that, since $\hat{F}$ is globally bounded, there exists $\bar{h} \geq 0$ such that, for any $t \in [0,t^*]$ and $(q,p) \in \mathcal{R}_{\delta_q,\delta_p}$, $|h_t| \leq \bar{h}$.
  
  By the Cameron--Martin formula, for any $(q,p) \in \mathcal{R}_{\delta_q,\delta_p}$ we have
  \begin{align*}
    &\Pr\left(S_{q,p} \leq \delta\right) = \Exp\left[\ind{\sup_{t \in [0,t^*]} |\sqrt{2\gamma\beta^{-1}} W_t| \leq \delta} M_{q,p}\right],\\
    &M_{q,p} := \exp\left(\int_0^{t^*} h_t \cdot \dd W_t - \frac{1}{2}\int_0^{t^*} \left|h_t\right|^2 \dd t\right).
  \end{align*}
  We therefore deduce from the Cauchy--Schwarz inequality that
  \begin{equation*}
    \Pr\left(S_{q,p} \leq \delta\right) \geq \frac{\Pr\left(\sup_{t \in [0,t^*]} |\sqrt{2\gamma\beta^{-1}} W_t| \leq \delta\right)^2}{\Exp\left[M_{q,p}^{-1}\right]}.
  \end{equation*}
  The numerator in the right-hand side is positive for any value of $\delta$, and does not depend on $(q,p)$. The denominator satisfies
  \begin{align*}
    \Exp\left[M_{q,p}^{-1}\right] &= \Exp\left[\exp\left(-\int_0^{t^*} h_t \cdot \dd W_t + \frac{1}{2}\int_0^{t^*} \left|h_t\right|^2 \dd t\right)\right]\\
    &= \Exp\left[\exp\left(-\int_0^{t^*} h_t \cdot \dd W_t - \frac{1}{2}\int_0^{t^*} \left|h_t\right|^2 \dd t\right)\right]\exp\left(\int_0^{t^*} \left|h_t\right|^2 \dd t\right)\\
    &\leq \exp\left(t^* \bar{h}^2\right),
  \end{align*}
  therefore~\eqref{eq:pfHarris1:alpha:2} holds with
  \begin{equation*}
    \alpha := \Pr\left(\sup_{t \in [0,t^*]} |\sqrt{2\gamma\beta^{-1}} W_t| \leq \delta\right)^2\exp\left(-t^* \bar{h}^2\right).
  \end{equation*}

\medskip \noindent    \emph{Step~3.} We may now complete the proof. Following the discussion in the introduction of this proof, we  fix a continuous and bounded function $f^\bfen : \Gamma^\bfen \cup \Gamma^0 \to [0,+\infty)$. For any $(q,p) \in \mathcal{R}_{\delta_q,\delta_p}$, we deduce from~\eqref{eq:pfHarris1:alpha:1} and the Markov property for $(q_t,p_t)_{t \geq 0}$ that
  \begin{align*}
    \Exp_{(q,p)}\left[f^\bfen(Y^\Sigma_1)\right] &= \Exp_{(q,p)}\left[f^\bfen(q_{\tau^\bfen_0},p_{\tau^\bfen_0})\right]\\
    &\geq \Exp_{(q,p)}\left[\ind{\tau^\bfen_0>t^*, (q_{t^*},p_{t^*}) \in K}f^\bfen(q_{\tau^\bfen_0},p_{\tau^\bfen_0})\right]\\
    &\geq \alpha \inf_{(q,p) \in K} u^\bfen(q,p),
  \end{align*}
  with $u^\bfen(q,p)$ defined as in Proposition~\ref{prop:Dirichlet-Langevin}. From the definition of $K$, $t^*$ and $p^*$, we deduce that for any $(q,p) \in K$,
  \begin{equation*}
    |q-q_\mathrm{ext}| \leq |q-(q^*+t^*p^*)| + |(q^*+t^*p^*)-q_\mathrm{ext}| \leq \frac{t^*r^*}{2} + (1-t^*)|p^*| = \left(1-\frac{t^*}{2}\right)r^* < r^*,
  \end{equation*}
  which by the exterior sphere property implies that $K \subset (\R^d \setminus \bar{\metasp}) \times \R^d$. Therefore there exists an open, bounded, $C^2$ and connected set $\metasp'$ such that 
  \begin{equation*}
    K \subset \metasp' \times \R^d \subset (\R^d \setminus \bar{\metasp}) \times \R^d.
  \end{equation*}
  Hence, by Proposition~\ref{prop:Dirichlet-Langevin}, $\mathcal{L}u^\bfen=0$ on $\metasp' \times \R^d$. Therefore, the Harnack inequality from~\cite[Theorem~2.15]{LelRamRey:kfp} applies and shows that there exists $\epsilon'>0$, which does not depend on the choice of $f^\bfen$, such that
  \begin{equation*}
    \inf_{(q,p) \in K} u^\bfen(q,p) \geq \epsilon' \sup_{(q,p) \in K} u^\bfen(q,p) \geq \epsilon' u^\bfen(q^*+t^*p^*,p^*).
  \end{equation*}
  We finally complete the proof by letting $\lambda(\cdot) := \Pr_{(q^*+t^*p^*,p^*)}((q_{\tau^\bfen_0},p_{\tau^\bfen_0}) \in \cdot)$ and $\epsilon := \alpha \epsilon'$.
\end{proof}

\begin{lem}[Recurrence of $\mathcal{R}_{\delta_q,\delta_p}$]\label{lem:Harris:2}
  Under the assumptions of Proposition~\ref{prop:HRYSigma}, for any $\delta_q>0$, $\delta_p>0$ and $(q,p) \in \mathcal{S}$, $\Pr_{(q,p)}(\exists m \geq 1: Y^\Sigma_m \in \mathcal{R}_{\delta_q,\delta_p})=1$.
\end{lem}
\begin{proof}
  For the sake of legibility, throughout the proof we fix $\delta_q$ and $\delta_p$ and simply denote $\mathcal{R} = \mathcal{R}_{\delta_q,\delta_p}$. Let $(q,p)\in\mathcal{S}$ and let $\tau_\mathcal{R}:=\inf\{t>0:(q_t,p_t) \in \mathcal{R}\}$. Since, by Lemma~\ref{lem:nonacc}, the sequences $(\tau^\bfex_n)_{n\geq0}$ and $(\tau^\bfen_n)_{n\geq0}$ do not accumulate, it is sufficient to prove that
  \begin{equation}\label{eq:hitting time R}
    \Pr_{(q,p)}(\tau_\mathcal{R}=+\infty)=0.  
  \end{equation}
 
  For $n\geq 1$, let us define the set $E_n$ by (see Figure~\ref{fig:lemma22})
  \begin{equation*}
    E_n := \left\{(q,p) \in \metasp \times \R^d: |q-q^*| \leq \frac{1}{n}, |p-p^*| \leq \frac{1}{n}\right\},
  \end{equation*}
  and set $\tau_{E_n}:=\inf\{t> 0: (q_t,p_t) \in E_n\}$. Since the process $(q_t,p_t)_{t\geq0}$ is ergodic and $E_n$ has positive Lebesgue measure, then for all $n\geq 1$, $\Pr_{(q,p)}(\tau_{E_n}<+\infty )=1$. Therefore,
  \begin{equation*}
    \Pr_{(q,p)}\left(\bigcap_{n\geq 1}\{\tau_{E_n}<+\infty\}\right)=1,
  \end{equation*}
  and the condition~\eqref{eq:hitting time R} is equivalent to the equality 
  \begin{equation*}
    0 = \Pr_{(q,p)}\left(\bigcap_{n\geq 1}\{\tau_{E_n}<+\infty,\tau_\mathcal{R}=+\infty\}\right) = \lim_{n \to +\infty} \Pr_{(q,p)}\left(\tau_{E_n}<+\infty,\tau_\mathcal{R}=+\infty\right).
  \end{equation*}
  For $n\geq 1$, the strong Markov property at the stopping time $\tau_{E_n}$ ensures that
  \begin{equation*}
    \Pr_{(q,p)}\left(\tau_{E_n}<+\infty,\tau_\mathcal{R}=+\infty\right)=\Exp_{(q,p)}\left[\ind{\tau_{E_n}<+\infty} \Pr_{(q_{\tau_{E_n}},p_{\tau_{E_n}})}\left(\tau_\mathcal{R}=+\infty\right)\right].
  \end{equation*}
  Therefore, if one proves that
  \begin{equation*}
    \lim_{n \to +\infty} \sup_{(q,p)\in \overline{E_n}}\Pr_{(q,p)}\left(\tau_\mathcal{R}=+\infty\right)=0,
  \end{equation*}
  the claimed result easily follows from the application of the dominated convergence theorem. Let $\alpha\in(0,1/2)$ and $t_n:=n^{-1/(1+\alpha)} \leq 1$. We prove here the stronger convergence 
  \begin{equation*}
    \lim_{n \to +\infty} \sup_{(q,p)\in \overline{E_n}}\Pr_{(q,p)}\left(\tau_\mathcal{R}>t_n\right) = 0.
  \end{equation*}
  
Let $n\geq1$, $(q,p)\in \overline{E_n}$, and
  \begin{equation*}
    \sigma := \inf\left\{t>0: (q_t,p_t) \not\in \bar{B}(q^*,\delta_q) \times \bar{B}(p^*,\delta_p)\right\}.
  \end{equation*}
We first prove in Step~1 that,  for $n$ sufficiently large, if the process remains in 
$\bar{B}(q^*,\delta_q) \times \bar{B}(p^*,\delta_p)$ until the time
$t_n$ (i.e. $\sigma > t_n$), then necessarily,
$\tau_\mathcal{R}\le t_n$. Then we conclude
in Step~2 by showing that the process necessarily leaves
$\bar{B}(q^*,\delta_q) \times \bar{B}(p^*,\delta_p)$ before the time
$t_n$, uniformly in the initial condition~$(q,p)$.

\medskip \noindent {\em Step 1.}    On the event $\{\sigma>t_n\}$, for any $t \in [0, t_n]$, we have $|p_t-p| \leq c_1 t + c_2 t^\alpha$, where
  \begin{equation*}
    c_1 := \sup\left\{|F(q')-\gamma p'|, (q',p') \in \bar{B}(q^*,\delta_q) \times \bar{B}(p^*,\delta_p)\right\}, \qquad c_2 := \sup_{0 < t \leq 1} \sqrt{2\gamma\beta^{-1}}\frac{|W_t|}{t^\alpha}.
  \end{equation*}
  Notice that since $\alpha<1/2$, the random variable $c_2$ is finite, almost surely. As a result,
  \begin{align*}
    |q_{t_n} - (q^*+t_n p^*)| &= \left|q-q^* + \int_0^{t_n} (p_t-p^*) \dd t\right|\\
    &\leq |q-q^*| + t_n |p-p^*| + \int_0^{t_n} |c_1 t + c_2 t^\alpha|\dd t\\
    &\leq \frac{1+t_n}{n} + \frac{c_1t_n^2}{2} + \frac{c_2t_n^{\alpha+1}}{\alpha+1}.
  \end{align*}
  We now recall the notation $r^* = p^* \cdot \mathsf{n}(q^*) = |p^*| > 0$. By the expression of $t_n$, there exists a deterministic index $n_0\geq1$ such that for all $n\geq n_0$, 
  \begin{equation*}
    \frac{1+t_n}{n} + \frac{c_1t_n^2}{2} \leq \frac{r^*}{6}t_n.
  \end{equation*}
  Let $n \geq n_0$. On the event 
  \begin{equation*}
    A_n := \left\{\sigma>t_n,\frac{c_2t_n^{\alpha+1}}{\alpha+1} \leq \frac{r^*}{6}t_n\right\},
  \end{equation*}
  one has
  \begin{equation}\label{sortie par R eq 1}
    \left|q_{t_n}-q^*\right| \leq \left(\frac{r^*}{3}+ |p^*|\right)t_n = \frac{4r^*}{3}t_n.
  \end{equation} 
  In addition, using the Cauchy--Schwarz inequality, one has also that
  \begin{equation*}
    \left(q_{t_n}-q^*-t_np^*\right) \cdot \mathsf{n}(q^*)\geq -\frac{r^*}{3}t_n.
  \end{equation*}
  As a result, either $q_{t_n}=q^*$ or
  \begin{equation*}
    \frac{\left(q_{t_n}-q^*\right) \cdot \mathsf{n}(q^*)}{|q_{t_n}-q^*|} = \frac{\left(q_{t_n}-q^*-t_np^*\right) \cdot \mathsf{n}(q^*) + t_n p^* \cdot \mathsf{n}(q^*)}{|q_{t_n}-q^*|} \geq \frac{1}{2}.
  \end{equation*}
  We deduce from this estimate that, on the event $A_n$, 
  \begin{align*}
    |q_{t_n}-q_\mathrm{ext}|^2 &= |q_{t_n}-q^*|^2 + 2 (q_{t_n}-q^*)\cdot(q^*-q_\mathrm{ext}) + |q^*-q_\mathrm{ext}|^2\\
    &= |q_{t_n}-q^*|^2 - 2 r^* (q_{t_n}-q^*)\cdot \mathsf{n}(q^*) + (r^*)^2\\
    &\leq |q_{t_n}-q^*|^2 - r^* |q_{t_n}-q^*| + (r^*)^2.
  \end{align*}
  Let $n_1$ be a deterministic index chosen large enough for the
  identity $t_n < 3/4$ to hold for $n \geq n_1$. Then by~\eqref{sortie
    par R eq 1}, if $n \geq n_0 \vee n_1$ then either $q_{t_n}=q^*$ or
  $|q_{t_n}-q^*|^2 - r^* |q_{t_n}-q^*|<0$ and, by the previous inequality, $|q_{t_n}-q_\mathrm{ext}|<r^*$. Thus, in both cases, by the exterior sphere property, $q_{t_n} \not\in \metasp$, and since we place ourselves on the event $\{\sigma>t_n\}$, then necessarily $\tau_\mathcal{R} \leq t_n$. 
  
  As a result, for $n \geq n_0 \vee n_1$ and $(q,p) \in \overline{E_n}$,
  \begin{equation*}
    \Pr_{(q,p)}\left(\tau_\mathcal{R}>t_n\right) = \Pr_{(q,p)}\left(\tau_\mathcal{R}>t_n, A_n^c\right) \leq \Pr_{(q,p)}\left(\sigma \leq t_n\right) + \Pr\left(\frac{c_2t_n^{\alpha+1}}{\alpha+1} > \frac{r^*}{6}t_n\right),
  \end{equation*}
  where $A_n^c$ denotes the complement of the event $A_n$. Since $c_2$ is almost surely finite and $t_n$ vanishes when $n \to +\infty$, we first note that
  \begin{equation*}
    \lim_{n \to +\infty} \Pr\left(\frac{c_2t_n^{\alpha+1}}{\alpha+1} > \frac{r^*}{6}t_n\right) = 0,
  \end{equation*}
  and the limit is uniform in the initial condition $(q,p)$. Therefore, to complete the proof, it remains to show that 
  \begin{equation*}
    \lim_{n \to +\infty} \sup_{(q,p) \in \overline{E_n}} \Pr_{(q,p)}\left(\sigma \leq t_n\right) = 0,
  \end{equation*}
which is the objective of the last step of the proof.

\medskip \noindent {\em Step 2.}   For $n$ large enough, $\overline{E}_n \subset \bar{B} := \bar{B}(q^*,\delta_q/2) \times \bar{B}(p^*,\delta_p/2)$, so that
  \begin{equation*}
    \sup_{(q,p) \in \overline{E_n}} \Pr_{(q,p)}\left(\sigma \leq t_n\right) \leq \sup_{(q,p) \in \bar{B}} \Pr_{(q,p)}\left(\sigma \leq t_n\right).
  \end{equation*}
  Let us define
  \begin{equation*}
    \bar{c} := \sup_{(q,p) \in  \bar{B}(q^*,\delta_q) \times \bar{B}(p^*,\delta_p)} \{|p| \vee |F(q)-\gamma p|\}.
  \end{equation*}
  Then for any $(q,p) \in \bar{B}$,
  \begin{align*}
    \Pr_{(q,p)}(\sigma \leq t_n) &= \Pr_{(q,p)}\left(\sup_{t \leq t_n} |q_{t \wedge \sigma}-q^*| > \delta_q \text{ or } \sup_{t \leq t_n} |p_{t \wedge \sigma}-p^*| > \delta_p\right)\\
    &\leq \Pr_{(q,p)}\left(\sup_{t \leq t_n} \left|\int_0^{t \wedge \sigma}p_s \dd s\right| > \frac{\delta_q}{2}\right)\\
    &\quad + \Pr_{(q,p)}\left(\sup_{t \leq t_n} \left|\int_0^{t \wedge \sigma}(F(q_s)-\gamma p_s) \dd s + \sqrt{2\gamma\beta^{-1}}W_{t \wedge \sigma}\right| > \frac{\delta_p}{2}\right).
  \end{align*}
  On the one hand,
  \begin{equation*}
    \sup_{t \leq t_n} \left|\int_0^{t \wedge \sigma}p_s \dd s\right| \leq \sup_{t \leq t_n} \bar{c} (t\wedge \sigma) \leq \bar{c} t_n;
  \end{equation*}
  on the other hand,
  \begin{align*}
    \sup_{t \leq t_n} \left|\int_0^{t \wedge \sigma}(F(q_s)-\gamma p_s) \dd s + \sqrt{2\gamma\beta^{-1}}W_{t \wedge \sigma}\right| &\leq \sup_{t \leq t_n} \left|\int_0^{t \wedge \sigma}(F(q_s)-\gamma p_s) \dd s\right| + \sup_{t \leq t_n}\sqrt{2\gamma\beta^{-1}}|W_{t \wedge \sigma}|\\
    &\leq \bar{c} t_n + \sup_{t \leq t_n}\sqrt{2\gamma\beta^{-1}}|W_t|.
  \end{align*}
  Therefore,
  \begin{equation*}
    \Pr_{(q,p)}(\sigma \leq t_n) \leq  \ind{\bar{c}t_n > \delta_q/2} + \Pr\left(\bar{c}t_n + \sup_{t \leq t_n} \sqrt{2\gamma\beta^{-1}}|W_t|  > \frac{\delta_p}{2}\right).
  \end{equation*}
  The right-hand side no longer depends on $(q,p)$ and vanishes when $n \to +\infty$, which completes the proof.
\end{proof}

\subsection{Completion of the proofs of Theorems~\ref{theo:main} and~\ref{theo:Y}}\label{ss:pfmain} We have proved in Subsection~\ref{ss:exist} that the probability measures $\pi^\bfex$ and $\pi^\bfen$ defined in Theorem~\ref{theo:main} are invariant for the Markov chains $(Y^\bfex_n)_{n \geq 0}$ and $(Y^\bfen_n)_{n \geq 0}$, and in Subsection~\ref{ss:HRYSigma} that the Markov chain $(Y^\Sigma_m)_{m \geq 0}$ defined in Theorem~\ref{theo:Y} is Harris recurrent. 

We first complete the proof of Theorem~\ref{theo:Y} by checking that the probability measure $\pi^\Sigma$ defined there is invariant for the Markov chain $(Y^\Sigma_m)_{m \geq 0}$. To do so, we let $f : \mathcal{S} \to \R$ be measurable and bounded. From the definition of $\pi^\Sigma$ and Proposition~\ref{prop:exist}, we get
\begin{align*}
  \Exp_{\pi^\Sigma}\left[f(Y^\Sigma_1)\right] &= \frac{1}{2}\int_{\Gamma^\bfex} \Exp_{(q,p)}\left[f(q_{\tau^\bfen_0},p_{\tau^\bfen_0})\right] \pi^\bfex(\dd q\dd p) + \frac{1}{2}\int_{\Gamma^\bfen} \Exp_{(q,p)}\left[f(q_{\tau^\bfex_0},p_{\tau^\bfex_0})\right] \pi^\bfen(\dd q\dd p)\\
  &= \frac{1}{2}\pi^\bfen(f) + \frac{1}{2}\pi^\bfex(f) = \pi^\Sigma(f),
\end{align*}
which is the expected result. 

It now remains to complete the proof of Theorem~\ref{theo:main} by checking that the Markov chains $(Y^\bfex_n)_{n \geq 0}$ and $(Y^\bfen_n)_{n \geq 0}$ are Harris recurrent. But since these chains are the traces of $(Y^\Sigma_m)_{m \geq 0}$ on $\Gamma^\bfex$ and $\Gamma^\bfen$, respectively, this statement immediately follows from the combination of Theorem~\ref{theo:Y} and Lemma~\ref{lem:trace-Harris}.

\subsection{Proof of Proposition~\ref{prop:rev}}\label{ss:rev} In this subsection, we assume that $F=-\nabla V$ and prove Proposition~\ref{prop:rev}. To proceed, we fix continuous and bounded functions $f_0^\bfen : \Gamma^\bfen \cup \Gamma^0 \to \R$, $f_1^\bfex : \Gamma^\bfex \cup \Gamma^0 \to \R$, and set
  \begin{equation*}
      \forall (q,p) \in \R^d \times \R^d, \qquad u_0^\bfen(q,p) := \Exp_{(q,p)}\left[f_0^\bfen(Y^\bfen_0)\right], \quad u_1^\bfen(q,p) := \Exp_{(q,p)}\left[f_1^\bfex(\mathsf{R}(Y^\bfen_0))\right].
  \end{equation*}
  We then set $v^\bfen_1 := u_1^\bfen \circ \mathsf{R}$, so that by Proposition~\ref{prop:Dirichlet-Langevin}, we have
  \begin{equation}\label{eq:pf-rev:1}
      \mathcal{L}u_0^\bfen = 0, \quad \mathcal{L}(v^\bfen_1 \circ \mathsf{R}) = \mathcal{L}u_1^\bfen = 0, \qquad \text{on $(\R^d \setminus \bar{\metasp}) \times \R^d$,}
  \end{equation}
 and on the boundaries $\Gamma^\bfen$ and $\Gamma^\bfex$:
  \begin{equation}\label{eq:pf-rev:2}
    u^\bfen_0(y) = \begin{cases}
      f^\bfen_0(y) & \text{on $\Gamma^\bfen$,}\\
      P^\Sigma f^\bfen_0(y) & \text{on $\Gamma^\bfex$,}
    \end{cases} \qquad v^\bfen_1(y) = \begin{cases}
      P^\Sigma (f^\bfex_1 \circ \mathsf{R})(\mathsf{R}(y)), &\text{on $\Gamma^\bfen$,}\\
      f^\bfex_1(y) &\text{on $\Gamma^\bfex$,}
    \end{cases}
  \end{equation}
  where $P^\Sigma$ denotes the transition kernel of the Markov chain $(Y^\Sigma_m)_{m \geq 0}$.
  
  Similarly to Proposition~\ref{prop:exist}, the proof relies on the formal integration by parts formula
  \begin{align*}
      \int_{(\R^d \setminus \bar{\metasp}) \times \R^d} v^\bfen_1(q,p) \mathcal{L} u^\bfen_0(q,p) \rho(q,p)\dd q\dd p &= \int_{\Sigma \times \R^d} p \cdot (-\mathsf{n}(q)) u^\bfen_0(q,p)v^\bfen_1(q,p)\rho(q,p)\dd\sigma_\Sigma(q)\dd p\\
      &\quad + \int_{(\R^d \setminus \bar{\metasp}) \times \R^d} \mathcal{L}(v^\bfen_1 \circ \mathsf{R})(q,p)  (u^\bfen_0 \circ \mathsf{R})(q,p) \rho(q,p)\dd q\dd p,
  \end{align*}
  which uses the reversibility up to momentum reversal of the continuous-time Langevin dynamics. By~\eqref{eq:pf-rev:1}, this identity reduces to
  \begin{equation}\label{eq:pf-rev:3}
      \int_{\Sigma \times \R^d} p \cdot \mathsf{n}(q) u^\bfen_0(q,p)v^\bfen_1(q,p)\rho(q,p)\dd\sigma_\Sigma(q)\dd p=0.
  \end{equation}
  In the sequel, we show how to complete the proof of Proposition~\ref{prop:rev} taking~\eqref{eq:pf-rev:3} for granted, and then we justify rigorously the identity~\eqref{eq:pf-rev:3}.
  
\begin{proof}[Proof of Proposition~\ref{prop:rev} under the assumption that~\eqref{eq:pf-rev:3} holds] Separating the integrals over $\Gamma^\bfen$ and $\Gamma^\bfex$ in the left-hand side of~\eqref{eq:pf-rev:3}, and using~\eqref{eq:pf-rev:2}, we end up with the identity
  \begin{equation*}
      \int_{\Gamma^\bfen} \pi^\bfen(\dd y^\bfen) f^\bfen_0(y^\bfen) P^\Sigma (f^\bfex_1 \circ \mathsf{R})(\mathsf{R}(y^\bfen)) = \int_{\Gamma^\bfex} \pi^\bfex(\dd y^\bfex) f^\bfex_1(y^\bfex) P^\Sigma f^\bfen_0(y^\bfex).
  \end{equation*}
  The right-hand side directly rewrites
  \begin{equation*}
      \int_{\Gamma^\bfex} \pi^\bfex(\dd y^\bfex) f^\bfex_1(y^\bfex) P^\Sigma f^\bfen_0(y^\bfex) = \Exp_{\pi^\bfex}\left[f^\bfex_1(Y^\Sigma_0)f^\bfen_0(Y^\Sigma_1)\right],
  \end{equation*}
  while setting $y^\bfex=\mathsf{R}(y^\bfen)$ in the left-hand side yields
  \begin{align*}
      \int_{\Gamma^\bfen} \pi^\bfen(\dd y^\bfen) f^\bfen_0(y^\bfen) P^\Sigma (f^\bfex_1 \circ \mathsf{R})(\mathsf{R}(y^\bfen)) &= \int_{\Gamma^\bfex} \pi^\bfex(\dd y^\bfex) f^\bfen_0(\mathsf{R}(y^\bfex)) P^\Sigma (f^\bfex_1 \circ \mathsf{R})(y^\bfex)\\
      &= \Exp_{\pi^\bfex}\left[f_0^\bfen(\mathsf{R}(Y^\Sigma_0))f_1^\bfex(\mathsf{R}(Y^\Sigma_1))\right],
  \end{align*}
  since it is obvious from their expressions in the conservative case that $\pi^\bfen$ and $\pi^\bfex$ are the pushforward of each other by $\mathsf{R}$. We thus conclude that under $\Pr_{\pi^\bfex}$, the pairs $(Y^\Sigma_0,Y^\Sigma_1)$ and $(\mathsf{R}(Y^\Sigma_1),\mathsf{R}(Y^\Sigma_0))$ have the same law. Using the same arguments but with an integration by parts on the domain $\metasp \times \R^d$ and test functions defined accordingly, we show that the pairs $(Y^\Sigma_0,Y^\Sigma_1)$ and $(\mathsf{R}(Y^\Sigma_1),\mathsf{R}(Y^\Sigma_0))$ also have the same law under $\Pr_{\pi^\bfen}$, and thereby under $\Pr_{\pi^\Sigma}$.
\end{proof}

\begin{proof}[Proof of~\eqref{eq:pf-rev:3}] We use the same notation as in the proof of Proposition~\ref{prop:exist}, and first write
  \begin{equation*}
      0 = \int_{\mathcal{U}_{M,\alpha}\times\R^d} \iota_M(q,p) v^\bfen_1(q,p) \mathcal{L}u^\bfen_0(q,p) \rho(q,p)\dd q\dd p,
  \end{equation*}
  thanks to~\eqref{eq:pf-rev:1}. Using the reversibility up to momentum reversal of the continuous-time Langevin dynamics with respect to $\mu(\dd q\dd p) = \rho(q,p)\dd q \dd p$, we get the integration by parts formula
    \begin{align*}
      &\int_{\mathcal{U}_{M,\alpha}\times\R^d} \iota_M(q,p) v^\bfen_1(q,p) \mathcal{L}u^\bfen_0(q,p) \rho(q,p)\dd q\dd p\\
      &= \int_{\Sigma_{M,\alpha} \times \R^d}(-p \cdot \mathsf{n}_{M,\alpha}(q))\iota_M(q,p) v^\bfen_1(q,p) u^\bfen_0(q,p) \rho(q,p)\dd\sigma_{\Sigma_{M,\alpha}}(q)\dd p\\
      &\quad + \int_{\mathcal{U}_{M,\alpha}\times\R^d} (u^\bfen_0 \circ \mathsf{R})(q,p) \mathcal{L}((\iota_M v^\bfen_1)\circ \mathsf{R})(q,p) \rho(q,p)\dd q\dd p,
  \end{align*}
  see~\cite[Section~2.2.3.1]{LelRouSto10} for details on the computation. On the one hand, by the same arguments as in the proof of Proposition~\ref{prop:exist}, we get
  \begin{align*}
      &\lim_{M \to +\infty} \lim_{\alpha \dto 0} \int_{\Sigma_{M,\alpha} \times \R^d}(-p \cdot \mathsf{n}_{M,\alpha}(q))\iota_M(q,p) v^\bfen_1(q,p) u^\bfen_0(q,p) \rho(q,p)\dd\sigma_{\Sigma_{M,\alpha}}(q)\dd p\\
      &= \int_{\Sigma \times \R^d}(-p \cdot \mathsf{n}(q)) v^\bfen_1(q,p) u^\bfen_0(q,p) \rho(q,p)\dd\sigma_\Sigma(q)\dd p.
  \end{align*}
  On the other hand, setting $\iota^\mathsf{R}_M := \iota_M \circ \mathsf{R}$, we get
  \begin{align*}
      \mathcal{L}((\iota_M v^\bfen_1)\circ \mathsf{R}) &=\mathcal{L}(\iota^\mathsf{R}_M u^\bfen_1)\\ 
      &= \iota^\mathsf{R}_M \mathcal{L} u_1^\bfen + u_1^\bfen \mathcal{L} \iota^\mathsf{R}_M + 2\gamma\beta^{-1} \nabla_p \iota^\mathsf{R}_M \cdot \nabla_p u_1^\bfen\\
      &= u_1^\bfen \mathcal{L} \iota^\mathsf{R}_M + 2\gamma\beta^{-1} \nabla_p \iota^\mathsf{R}_M \cdot \nabla_p u_1^\bfen,
  \end{align*}
  thanks to~\eqref{eq:pf-rev:1} again. Therefore,
  \begin{align*}
      &\int_{\mathcal{U}_{M,\alpha}\times\R^d} (u^\bfen_0 \circ \mathsf{R})(q,p) \mathcal{L}((\iota_M v^\bfen_1)\circ \mathsf{R})(q,p)\rho(q,p) \dd q\dd p\\
      &= \int_{\mathcal{U}_{M,\alpha}\times\R^d} u^\bfen_0(q,-p) u_1^\bfen(q,p) \mathcal{L} \iota^\mathsf{R}_M(q,p)\rho(q,p) \dd q\dd p\\
      &\quad + 2\gamma\beta^{-1} \int_{\mathcal{U}_{M,\alpha}\times\R^d} u^\bfen_0(q,-p) \nabla_p \iota^\mathsf{R}_M(q,p) \cdot \nabla_p u_1^\bfen(q,p)\rho(q,p) \dd q\dd p.
  \end{align*}
  The boundedness of $u^\bfen_0$, $u_1^\bfen$, the definition of $\iota_M$ and Assumption~(\ref{ass:A3}) allow to show that
  \begin{equation*}
      \lim_{M \to +\infty} \lim_{\alpha \dto 0} \int_{\mathcal{U}_{M,\alpha}\times\R^d} u^\bfen_0(q,-p) u_1^\bfen(q,p) \mathcal{L} \iota^\mathsf{R}_M(q,p)\rho(q,p) \dd q\dd p = 0,
  \end{equation*}
  while the Cauchy--Schwarz inequality yields
  \begin{align*}
      &\left|\int_{\mathcal{U}_{M,\alpha}\times\R^d} u^\bfen_0(q,-p) \nabla_p \iota^\mathsf{R}_M(q,p) \cdot \nabla_p u_1^\bfen(q,p)\rho(q,p) \dd q\dd p\right|\\
      &\leq \|u^\bfen_0\|_\infty \left(\int_{\mathcal{U}_{M,\alpha}\times\R^d} |\nabla_p \iota^\mathsf{R}_M(q,p)|^2 \rho(q,p)\dd q\dd p \int_{\mathcal{U}_{M,\alpha}\times\R^d} |\nabla_p u_1^\bfen(q,p)|^2 \rho(q,p)\dd q\dd p\right)^{1/2}.
  \end{align*}
  By the definition of $\iota_M$,
  \begin{equation*}
      \lim_{M \to +\infty} \lim_{\alpha \dto 0} \int_{\mathcal{U}_{M,\alpha}\times\R^d} |\nabla_p \iota^\mathsf{R}_M(q,p)|^2 \rho(q,p)\dd q\dd p = 0,
  \end{equation*}
  while by Lemma~\ref{lem:energy},
  \begin{equation*}
     \int_{\mathcal{U}_{M,\alpha}\times\R^d} |\nabla_p u_1^\bfen(q,p)|^2 \rho(q,p)\dd q\dd p \leq \int_{(\R^d \setminus \bar{\metasp}) \times\R^d} |\nabla_p u_1^\bfen(q,p)|^2 \rho(q,p)\dd q\dd p < +\infty.
  \end{equation*}
  Overall, this completes the proof of~\eqref{eq:pf-rev:3}.
\end{proof}

\section{Proofs of Proposition~\ref{prop:nuA}, Theorem~\ref{theo:Hill-Langevin} and Corollary~\ref{cor:MRT}}\label{s:Hill}

This section is divided into four subsections. In
Subsections~\ref{ss:react-distrib} and~\ref{ss:Hill-gen}, we
temporarily leave the Langevin dynamics~\eqref{eq:Langevin} aside and
consider an arbitrary Markov chain $(Y_n)_{n \geq 0}$, with transition
kernel denoted by $P$, which takes its values in a separable
measurable space $\mathcal{S}$. We work under the following general
setting for the Markov chain $(Y_n)_{n \geq 0}$ (we recall that generalities on Harris recurrent chains are gathered in
Appendix~\ref{app:Harris}):

  \begin{enumerate}[label=(E\arabic*),ref=E\arabic*]
    \item\label{ass:Harris1} The chain $(Y_n)_{n \geq 0}$ is positive Harris recurrent on $\mathcal{S}$. Its unique stationary probability distribution is denoted by $\pi$.
    \item\label{ass:Harris2} There exist two measurable sets ${\mathcal{A}}, {\mathcal{B}} \subset \mathcal{S}$ such that $\mathcal{S} = {\mathcal{A}} \cup {\mathcal{B}}$, ${\mathcal{A}} \cap {\mathcal{B}} = \emptyset$, and $\pi({\mathcal{A}})\pi({\mathcal{B}})>0$.
  \end{enumerate}

Under Assumptions~(\ref{ass:Harris1}--\ref{ass:Harris2}), we establish general versions of Proposition~\ref{prop:nuA} and Theorem~\ref{theo:Hill-Langevin}. Then we check in Subsection~\ref{ss:appl-Hill} that Assumptions~(\ref{ass:Harris1}--\ref{ass:Harris2}) are satisfied by the sequence $(Y^\bfen_n)_{n \geq 0}$ in the setting of Subsection~\ref{ss:Hill-Langevin}, which proves Proposition~\ref{prop:nuA} and Theorem~\ref{theo:Hill-Langevin}. We also prove Corollary~\ref{cor:MRT} there. Last, in Subsection~\ref{ss:rev-reactive}, we exhibit remarkable identities between reactive distributions in the conservative case.

\subsection{Reactive entrance and exit distributions}\label{ss:react-distrib}

\begin{equation*}
  \begin{array}{ll}
    \eta^\ret_{{\mathcal{A}},0} := \min\{n \geq 0: Y_n \in {\mathcal{A}}\}, \qquad &\eta^\ret_{{\mathcal{B}},0} := \min\{n \geq \eta^\ret_{{\mathcal{A}},0}: Y_n \in {\mathcal{B}}\},\\
    \eta^\ret_{{\mathcal{A}},k+1} := \min\{n \geq \eta^\ret_{{\mathcal{B}},k}: Y_n \in {\mathcal{A}}\}, \qquad & \eta^\ret_{{\mathcal{B}},k+1} := \min\{n \geq \eta^\ret_{{\mathcal{A}},k+1}: Y_n \in {\mathcal{B}}\}, \qquad k \geq 0.
  \end{array}
\end{equation*}
Both $(\eta^\ret_{{\mathcal{A}},k})_{k \geq 0}$ and $(\eta^\ret_{{\mathcal{B}},k})_{k \geq 0}$ are sequences of stopping times for the natural filtration of $(Y_n)_{n \geq 0}$. Obviously, they respectively refer to the successive return times of the chain in ${\mathcal{A}}$ after a visit in ${\mathcal{B}}$ (and conversely), with the convention that these times are counted from the first visit of the chain in ${\mathcal{A}}$. Therefore we always have
\begin{equation*}
  0 \leq \eta^\ret_{{\mathcal{A}},0} < \eta^\ret_{{\mathcal{B}},0} < \eta^\ret_{{\mathcal{A}},1} < \eta^\ret_{{\mathcal{B}},1} < \cdots.
\end{equation*}
In the sequel it shall be convenient to denote by
\begin{equation*}
  \eta^\exi_{{\mathcal{A}},k} := \eta^\ret_{{\mathcal{B}},k}-1, \qquad \eta^\exi_{{\mathcal{B}},k} := \eta^\ret_{{\mathcal{A}},k+1}-1, \qquad k \geq 0,
\end{equation*}
the respective reactive exit times from ${\mathcal{A}}$ and
${\mathcal{B}}$. In words, $\eta^\exi_{{\mathcal{A}},k}$
(resp. $\eta^\exi_{{\mathcal{B}},k}$) is the time of the last visit in
$\mathcal A$ (resp. $\mathcal B$) before the $k$-th entry in $\mathcal B$ (resp. $\mathcal A$).
Notice that in general, these times are \emph{not} stopping times.

For any $k \geq 0$, we set 
\begin{equation*}
  Y^\ret_{{\mathcal{A}},k} := Y_{\eta^\ret_{{\mathcal{A}},k}}, \qquad Y^\exi_{{\mathcal{A}},k} := Y_{\eta^\exi_{{\mathcal{A}},k}}.
\end{equation*}
These notations are illustrated on the Markov chain $(Y_n^-)_{n \ge 0}$, respectively $(Y_n^+)_{n \ge 0}$, on Figure~\ref{fig:proposition57} below (see in particular the points $Y^\ret_{{\mathcal{A}^-},k}$, $Y^\exi_{{\mathcal{A}^-},k+1}$, respectively the points $Y^\ret_{{\mathcal{B}^-},k+1}$, $Y^\exi_{{\mathcal{B}^-},k}$). Following~\cite{LuNol15}, for all $\ell \geq 1$ we define the \emph{empirical reactive entrance distribution} in ${\mathcal{A}}$ by
\begin{equation*}
  \nu^\ret_{{\mathcal{A}},\ell} := \frac{1}{\ell}\sum_{k=0}^{\ell-1} \delta_{Y^\ret_{{\mathcal{A}},k}},
\end{equation*}
and the \emph{empirical reactive exit distribution} from ${\mathcal{A}}$ by
\begin{equation*}
  \nu^\exi_{{\mathcal{A}},\ell} := \frac{1}{\ell}\sum_{k=0}^{\ell-1} \delta_{Y^\exi_{{\mathcal{A}},k}}.
\end{equation*}

\begin{prop}[Reactive entrance and exit distributions]\label{prop:reactive}
  Under Assumptions~(\ref{ass:Harris1}--\ref{ass:Harris2}), let us define the probability measures
  \begin{equation*}
    \nu^\ret_{\mathcal{A}}(\dd y) := \frac{\Pr_\pi(Y_0 \in {\mathcal{B}}, Y_1 \in \dd y)}{\Pr_\pi(Y_0 \in {\mathcal{B}}, Y_1 \in {\mathcal{A}})}, \qquad \nu^\exi_{\mathcal{A}}(\dd y) := \frac{\Pr_\pi(Y_0 \in \dd y, Y_1 \in {\mathcal{B}})}{\Pr_\pi(Y_0 \in {\mathcal{A}}, Y_1 \in {\mathcal{B}})}, \qquad \text{on ${\mathcal{A}}$.}
  \end{equation*}
  For any $f \in L^1({\mathcal{A}},\nu^\ret_{\mathcal{A}})$,
  \begin{equation*}
    \lim_{\ell \to +\infty} \nu^\ret_{{\mathcal{A}},\ell}(f) = \nu^\ret_{\mathcal{A}}(f), \qquad \text{almost surely,}
  \end{equation*}
  and for any $f \in L^1({\mathcal{A}},\nu^\exi_{\mathcal{A}})$,
  \begin{equation*}
    \lim_{\ell \to +\infty} \nu^\exi_{{\mathcal{A}},\ell}(f) = \nu^\exi_{\mathcal{A}}(f), \qquad \text{almost surely.}
  \end{equation*}
\end{prop}
\begin{proof}
  The starting point of the proof is the elementary observation that the sequences of pairs $(Y_{\eta^\exi_{{\mathcal{A}},k}}, Y_{\eta^\ret_{{\mathcal{B}},k}})_{k \geq 0}$ and $(Y_{\eta^\exi_{{\mathcal{B}},k}}, Y_{\eta^\ret_{{\mathcal{A}},k+1}})_{k \geq 0}$, which respectively represent the transitions of the chain $(Y_n)_{n \geq \eta^\ret_{{\mathcal{A}},0}}$ from ${\mathcal{A}}$ to ${\mathcal{B}}$ and from ${\mathcal{B}}$ to ${\mathcal{A}}$, are the respective traces of the chain $(Y_n,Y_{n+1})_{n \geq \eta^\ret_{{\mathcal{A}},0}}$ on ${\mathcal{A}} \times {\mathcal{B}}$ and ${\mathcal{B}} \times {\mathcal{A}}$. We thus deduce from Lemmas~\ref{lem:pair-Harris} and~\ref{lem:trace-Harris} that these sequences are positive Harris recurrent chains on $\mathcal{S} \times \mathcal{S}$, with respective stationary distributions 
  \begin{equation*}
    \pi \otimes P(\dd y_0\dd y_1|{\mathcal{A}} \times {\mathcal{B}}) = \frac{\Pr_\pi(Y_0 \in \dd y_0, Y_1 \in \dd y_1)}{\Pr_\pi(Y_0 \in {\mathcal{A}}, Y_1 \in {\mathcal{B}})} \qquad \text{on ${\mathcal{A}} \times {\mathcal{B}}$},
  \end{equation*}
   and 
   \begin{equation*}
     \pi \otimes P(\dd y_0\dd y_1|{\mathcal{B}} \times {\mathcal{A}}) = \frac{\Pr_\pi(Y_0 \in \dd y_0, Y_1 \in \dd y_1)}{\Pr_\pi(Y_0 \in {\mathcal{B}}, Y_1 \in {\mathcal{A}})} \qquad \text{on ${\mathcal{B}} \times {\mathcal{A}}$}.
   \end{equation*} 
   Therefore, applying Proposition~\ref{prop:LLN-Harris} to these chains, we deduce that $\nu^\ret_{{\mathcal{A}},\ell}(f)$ converges to $\nu^\ret_{\mathcal{A}}(f)$, where $\nu^\ret_{\mathcal{A}}$ is the second marginal of the stationary distribution of $(Y_{\eta^\exi_{{\mathcal{B}},k}}, Y_{\eta^\ret_{{\mathcal{A}},k+1}})_{k \geq 0}$, and similarly $\nu^\exi_{{\mathcal{A}},\ell}(f)$ converges to $\nu^\exi_{\mathcal{A}}(f)$, where $\nu^\exi_{\mathcal{A}}$ is the first marginal of the stationary distribution of $(Y_{\eta^\exi_{{\mathcal{A}},k}}, Y_{\eta^\ret_{{\mathcal{B}},k}})_{k \geq 0}$. 
\end{proof}

We respectively call $\nu^\ret_{\mathcal{A}}$ and $\nu^\exi_{\mathcal{A}}$
the \emph{reactive entrance} and \emph{reactive exit
  distributions}. For a trajectory $(Y_n)_{n \ge 0}$ starting under
the stationary state ($Y_0 \sim \pi$), $\nu^\ret_{\mathcal{A}}$ is the
law of $Y_1$ conditionally to $(Y_0,Y_1) \in \mathcal{B} \times
\mathcal{A}$ and $\nu^\exi_{\mathcal{A}}$ is the law of $Y_0$ conditionally to $(Y_0,Y_1) \in \mathcal{A} \times
\mathcal{B}$. In a potential theoretic setting they might also be called \emph{first-entrance} and \emph{last-exit biased distributions}~\cite[Theorem~7.10]{BovDen15}. Clearly, similar distributions can be defined on ${\mathcal{B}}$ using the sequences $(Y_{\eta^\ret_{{\mathcal{B}},k}})_{k \geq 0}$ and $(Y_{\eta^\exi_{{\mathcal{B}},k}})_{k \geq 0}$. 

\begin{rk}[Reversibility and capacity]\label{rk:cap}
  If the chain $(Y_n)_{n \geq 0}$ is reversible with respect to $\pi$,
  then it is clear that $\nu^\ret_{\mathcal{A}} =
  \nu^\exi_{\mathcal{A}}$ (see also
  Proposition~\ref{prop:momentum_reversal} for a related discussion
  for the Langevin dynamics). In this case and under Assumptions~(\ref{ass:Harris1}--\ref{ass:Harris2}), the normalisation factor
  \begin{equation}\label{eq:capacity}
    \Pr_\pi(Y_0 \in {\mathcal{A}}, Y_1 \in {\mathcal{B}})=\Pr_\pi(Y_0 \in {\mathcal{B}}, Y_1 \in {\mathcal{A}})
  \end{equation}
  is usually called the \emph{capacity} between ${\mathcal{A}}$ and ${\mathcal{B}}$~\cite[Theorem~7.10]{BovDen15} and denoted by $\mathrm{cap}({\mathcal{A}},{\mathcal{B}})$. In the non-reversible case, the measures $\nu^\ret_{\mathcal{A}}$ and $\nu^\exi_{\mathcal{A}}$ need not coincide, however under Assumptions~(\ref{ass:Harris1}--\ref{ass:Harris2}), the identity~\eqref{eq:capacity} remains true. Indeed, it is obvious that
  \begin{equation*}
    \pi({\mathcal{A}}) = \Pr_\pi(Y_0 \in {\mathcal{A}}) = \Pr_\pi(Y_0 \in {\mathcal{A}}, Y_1 \in {\mathcal{A}}) + \Pr_\pi(Y_0 \in {\mathcal{A}}, Y_1 \in {\mathcal{B}}),
  \end{equation*}
  but since $\pi$ is stationary we also have
  \begin{equation*}
    \pi({\mathcal{A}}) = \Pr_\pi(Y_1 \in {\mathcal{A}}) = \Pr_\pi(Y_0 \in {\mathcal{A}}, Y_1 \in {\mathcal{A}}) + \Pr_\pi(Y_0 \in {\mathcal{B}}, Y_1 \in {\mathcal{A}}),
  \end{equation*}
  which yields the claimed identity.
\end{rk}

As far as the reactive entrance distribution is concerned, the statement of Proposition~\ref{prop:reactive} can be strengthened as follows.

\begin{prop}[Markov property for entrance points]\label{prop:Markov-reactive}
  Under Assumptions~(\ref{ass:Harris1}--\ref{ass:Harris2}), the sequence $(Y^\ret_{{\mathcal{A}},k})_{k \geq 0}$ is a positive Harris recurrent chain in ${\mathcal{A}}$, with stationary probability distribution $\nu^\ret_{\mathcal{A}}$.
\end{prop}
\begin{proof}
  It follows from the strong Markov property that $(Y^\ret_{{\mathcal{A}},k})_{k \geq 0}$ is indeed a Markov chain, because $(\eta^\ret_{{\mathcal{A}},k})_{k \geq 0}$ is a sequence of stopping times, and the recursive construction of these times makes time homogeneity clear. Now Proposition~\ref{prop:reactive} shows that $\nu^\ret_{\mathcal{A}}$ is invariant for this chain, and thus by Proposition~\ref{prop:LLN-Harris} the latter is positive Harris recurrent.
\end{proof}

\subsection{The Hill relation}\label{ss:Hill-gen}

We may now write the Hill relation in the setting of Assumptions~(\ref{ass:Harris1}--\ref{ass:Harris2}). In the next statement, we use the notation $\pi_{\mathcal{A}}(\cdot) := \pi(\cdot|{\mathcal{A}})$.

\begin{theo}[Hill relation]\label{theo:Hill}
  Under Assumptions~(\ref{ass:Harris1}--\ref{ass:Harris2}) and with the notation of Proposition~\ref{prop:reactive}, we have for any $g \in L^1({\mathcal{A}},\pi_{\mathcal{A}})$,
  \begin{equation}\label{eq:Hill0}
    \Exp_{\nu^\ret_{\mathcal{A}}}\left[\sum_{n=0}^{\eta^\ret_{{\mathcal{B}},0}-1} |g(Y_n)|\right] < +\infty,
  \end{equation}
  and
  \begin{equation}\label{eq:Hill}
    \Exp_{\nu^\ret_{\mathcal{A}}}\left[\sum_{n=0}^{\eta^\ret_{{\mathcal{B}},0}-1} g(Y_n)\right] = \frac{\pi_{\mathcal{A}}(g)}{\Pr_{\pi_{\mathcal{A}}}(Y_1 \in {\mathcal{B}})}.
  \end{equation}
\end{theo}

Extending $g$ to $\mathcal{S}$ by letting $g(y)=0$ on ${\mathcal{B}}$ and dropping the normalisation by $\pi({\mathcal{A}})$, the right-hand side of~\eqref{eq:Hill} rewrites $\pi(g)/\mathrm{cap}({\mathcal{A}},{\mathcal{B}})$, where we recall from Remark~\ref{rk:cap} that 
\begin{equation*}
  \mathrm{cap}({\mathcal{A}},{\mathcal{B}}) = \Pr_\pi\left(Y_0 \in {\mathcal{A}}, Y_1 \in {\mathcal{B}}\right) = \Pr_\pi\left(Y_0 \in {\mathcal{B}}, Y_1 \in {\mathcal{A}}\right).
\end{equation*}
Besides, the left-hand side rewrites $\nu^\ret_{\mathcal{A}}(f)$, where $f$ is the solution to the Dirichlet problem
\begin{equation}\label{eq:Dirichlet}
  \left\{\begin{aligned}
    -(P-I)f(y) &= g(y), &y \in {\mathcal{A}},\\
    f(y) &= 0, &y \in {\mathcal{B}}.
  \end{aligned}\right.
\end{equation}
Therefore the Hill relation~\eqref{eq:Hill} establishes a link between this Dirichlet problem (and its associated Green kernel, which maps $g$ to $f$), the reactive entrance distribution $\nu^\ret_{\mathcal{A}}$, the invariant measure $\pi$ and the capacity $\mathrm{cap}({\mathcal{A}},{\mathcal{B}})$. This relation is proved in~\cite[Theorem~7.10, Eq.~(7.1.37)]{BovDen15} for a reversible Markov chain in a discrete state space\footnote{In this reference, the relation involves the reactive exit distribution $\nu^\exi_{\mathcal{A}}$ rather that the reactive entrance distribution $\nu^\ret_{\mathcal{A}}$. Since the chain is assumed there to be reversible with respect to $\pi$, following Remark~\ref{rk:cap} both measures coincide.} and in~\cite[Corollary~4.3]{BauGuyLel} for a Feller chain in a compact state space. Both references are based on a potential theoretic approach, in the sense that they crucially exploit the representation of the left-hand side of~\eqref{eq:Hill} in terms of the Dirichlet problem~\eqref{eq:Dirichlet}. In contrast, our argument starts with the remark, already made in the proof of Proposition~\ref{prop:reactive}, that $\nu^\ret_{\mathcal{A}}$ is the second marginal of the invariant measure of the trace of the pair chain $(\bar{Y}_n)_{n \geq 0}=(Y_n,Y_{n+1})_{n \geq 0}$ on ${\mathcal{B}} \times {\mathcal{A}}$. This allows us to deduce~\eqref{eq:Hill} from the application of a standard representation formula, which is recalled in Proposition~\ref{prop:rep-pi}, for the invariant measure of this chain.

\begin{proof}[Proof of Theorem~\ref{theo:Hill}]
  Let us take $g \in L^1({\mathcal{A}},\pi_{\mathcal{A}})$ and set $g(y)=0$ for any $y \in {\mathcal{B}}$. Then the function $\bar{g}$ defined on $\mathcal{S} \times \mathcal{S}$ by $\bar{g}(y_0,y_1) := g(y_1)$ is in $L^1(\mathcal{S} \times \mathcal{S},\bar{\pi})$, where we denote by $\bar{\pi} = \pi \otimes P$ the invariant measure of the pair chain $(\bar{Y}_n)_{n \geq 0}$. We may therefore apply Proposition~\ref{prop:rep-pi} to this chain, with the set ${\mathcal{B}} \times {\mathcal{A}}$. This yields
  \begin{equation}\label{eq:Hill-barY}
    \Exp_{\bar{\pi}_{{\mathcal{B}} \times {\mathcal{A}}}}\left[\sum_{n=0}^{\bar{\eta}-1} |\bar{g}(\bar{Y}_n)|\right] < +\infty, \qquad \Exp_{\bar{\pi}_{{\mathcal{B}} \times {\mathcal{A}}}}\left[\sum_{n=0}^{\bar{\eta}-1} \bar{g}(\bar{Y}_n)\right] = \frac{\bar{\pi}(\bar{g})}{\bar{\pi}({\mathcal{B}} \times {\mathcal{A}})},
  \end{equation}
  with $\bar{\eta} := \inf\{n \geq 1: \bar{Y}_n \in {\mathcal{B}} \times {\mathcal{A}}\}$. 
  
  If $\bar{Y}_0 = (Y_0,Y_1) \sim \bar{\pi}_{{\mathcal{B}} \times {\mathcal{A}}}$, then $\eta^\ret_{{\mathcal{A}},0}=1$ and $\eta^\ret_{{\mathcal{A}},1}=\bar{\eta}+1$. Therefore
  \begin{equation*}
    \sum_{n=0}^{\bar{\eta}-1} |\bar{g}(\bar{Y}_n)| = \sum_{n=0}^{\eta^\ret_{{\mathcal{A}},1}-2} |g(Y_{n+1})| = \sum_{n=1}^{\eta^\ret_{{\mathcal{A}},1}-1} |g(Y_n)| = \sum_{n=\eta^\ret_{{\mathcal{A}},0}}^{\eta^\ret_{{\mathcal{B}},0}-1} |g(Y_n)|,
  \end{equation*}
  where we have used the fact that $g$ vanishes on ${\mathcal{B}}$ for the last equality. Furthermore, if $\bar{Y}_0 \sim \bar{\pi}_{{\mathcal{B}} \times {\mathcal{A}}}$, then $Y_1 \sim \nu^\ret_{\mathcal{A}}$ and by the Markov property, the sequence $(Y_n)_{n \geq 1}$ has the same law as $(Y_n)_{n \geq 0}$ under $\Pr_{\nu^\ret_{\mathcal{A}}}$, for which $\eta^\ret_{{\mathcal{A}},0}=0$. Therefore,
  \begin{equation*}
    \Exp_{\bar{\pi}_{{\mathcal{B}} \times {\mathcal{A}}}}\left[\sum_{n=0}^{\bar{\eta}-1} |\bar{g}(\bar{Y}_n)|\right] = \Exp_{\nu^\ret_{\mathcal{A}}}\left[\sum_{n=\eta^\ret_{{\mathcal{A}},0}}^{\eta^\ret_{{\mathcal{B}},0}-1} |g(Y_n)|\right] = \Exp_{\nu^\ret_{\mathcal{A}}}\left[\sum_{n=0}^{\eta^\ret_{{\mathcal{B}},0}-1} |g(Y_n)|\right],
  \end{equation*}
  so that the first part of~\eqref{eq:Hill-barY} yields~\eqref{eq:Hill0}. Replacing $|\bar{g}|$ with $\bar{g}$ in the argument, we then deduce from the second part of~\eqref{eq:Hill-barY} that
  \begin{equation*}
    \Exp_{\nu^\ret_{\mathcal{A}}}\left[\sum_{n=0}^{\eta^\ret_{{\mathcal{B}},0}-1} g(Y_n)\right] = \frac{\bar{\pi}(\bar{g})}{\bar{\pi}({\mathcal{B}} \times {\mathcal{A}})}.
  \end{equation*}
  On the one hand,
  \begin{equation*}
    \bar{\pi}(\bar{g}) = \int_{y_0,y_1 \in \mathcal{S}} g(y_1) \pi(\dd y_0)P(y_0,\dd y_1) = \pi(g)
  \end{equation*}
  since $\pi P = \pi$; on the other hand,
  \begin{equation*}
    \bar{\pi}({\mathcal{B}} \times {\mathcal{A}}) = \pi \otimes P({\mathcal{B}} \times {\mathcal{A}}) = \Pr_\pi(Y_0 \in {\mathcal{B}}, Y_1 \in {\mathcal{A}}).
  \end{equation*}
  Following the discussion preceding this proof, this shows~\eqref{eq:Hill}.
\end{proof}

\subsection{Proofs of Proposition~\ref{prop:nuA}, Theorem~\ref{theo:Hill-Langevin} and Corollary~\ref{cor:MRT}}\label{ss:appl-Hill}

Under the assumptions of Theorem~\ref{theo:main} with $\metasp = A \cup B$ as in Subsection~\ref{ss:Hill-Langevin}, the sequence $(Y^\bfen_n)_{n \geq 0}$ satisfies Assumptions~(\ref{ass:Harris1}--\ref{ass:Harris2}) and therefore Proposition~\ref{prop:nuA} and Theorem~\ref{theo:Hill-Langevin} immediately follow from Proposition~\ref{prop:Markov-reactive} and Theorem~\ref{theo:Hill}, respectively. 

The proof of Corollary~\ref{cor:MRT} requires more work. We start by stating the following result.

\begin{lem}[Integrability of $\tau^\bfen_1$ under $\pi_{\mathcal{A}^-}$]\label{lem:integ-g}
  With the assumptions of Proposition~\ref{prop:nuA} and Assumption~(\ref{ass:D}), $\Exp_{\pi_{\mathcal{A}^-}}[\tau^\bfen_1] < +\infty$. 
\end{lem}
\begin{proof}
  Since, starting from any $(q_0,p_0) \in \Gamma^\bfen$, we have $0 = \tau^\bfen_0 < \tau^\bfex_0 < \tau^\bfen_1$, by Proposition~\ref{prop:exist} and the strong Markov property we have the identity
  \begin{equation*}
    \Exp_{\pi^\bfen}[\tau^\bfen_1] = \Exp_{\pi^\bfen}[\tau^\bfex_0] + \Exp_{\pi^\bfex}[\tau^\bfen_0].
  \end{equation*}
  We thus deduce from Assumption~(\ref{ass:D}) that $\Exp_{\pi^\bfen}[\tau^\bfen_1] < +\infty$, which implies that $\Exp_{\pi_{\mathcal{A}^-}}[\tau^\bfen_1] < +\infty$.
\end{proof}

We prove Corollary~\ref{cor:MRT} assuming that $G \geq 0$. The general statement then follows by applying the result to $[G]_-$ and $[G]_+$, separately. By the Fubini--Tonelli theorem,
\begin{align*}
  \Exp_{\nu^\ret_{\mathcal{A}^-}}\left[\int_0^{\tau^\ret_{{\mathcal{B}^-},0}} G(q_s,p_s)\dd s\right] &= \Exp_{\nu^\ret_{\mathcal{A}^-}}\left[\sum_{n=0}^{\eta^\ret_{{\mathcal{B}^-},0}-1} \int_{\tau^\bfen_n}^{\tau^\bfen_{n+1}} G(q_s,p_s)\dd s\right]\\
  &= \sum_{n=0}^{+\infty} \Exp_{\nu^\ret_{\mathcal{A}^-}}\left[\Exp_{\nu^\ret_{\mathcal{A}^-}}\left[\left.\ind{n < \eta^\ret_{{\mathcal{B}^-},0}}\int_{\tau^\bfen_n}^{\tau^\bfen_{n+1}} G(q_s,p_s)\dd s \right| \mathcal{F}_{\tau^\bfen_n}\right]\right].
\end{align*}
For any $n \geq 0$, the event $\{n < \eta^\ret_{{\mathcal{B}^-},0}\}$ is in $\mathcal{F}_{\tau^\bfen_n}$, while by the strong Markov property for $(q_t,p_t)_{t \geq 0}$, we have
\begin{equation*}
  \Exp_{\nu^\ret_{\mathcal{A}^-}}\left[\left.\int_{\tau^\bfen_n}^{\tau^\bfen_{n+1}} G(q_s,p_s)\dd s\right| \mathcal{F}_{\tau^\bfen_n}\right] = g^\bfen\left(Y^\bfen_n\right),
\end{equation*}
where
\begin{equation*}
  \forall y \in {\mathcal{A}^-}, \qquad g^\bfen(y) := \Exp_y\left[\int_0^{\tau^\bfen_1} G(q_s,p_s)\dd s\right].
\end{equation*}
Therefore we deduce that
\begin{equation}\label{eq:pfcor-MRT:1}
  \Exp_{\nu^\ret_{\mathcal{A}^-}}\left[\int_0^{\tau^\ret_{{\mathcal{B}^-},0}} G(q_s,p_s)\dd s\right] = \Exp_{\nu^\ret_{\mathcal{A}^-}}\left[\sum_{n=0}^{\eta^\ret_{{\mathcal{B}^-},0}-1} g^\bfen(Y^\bfen_n)\right].
\end{equation}

Lemma~\ref{lem:integ-g} combined with the boundedness of $G$ ensures that the function $g^\bfen$ defined above is in $L^1({\mathcal{A}^-},\pi_{\mathcal{A}^-})$. Thus,~\eqref{eq:MRT:exp} and~\eqref{eq:MRT:Hill} in Corollary~\ref{cor:MRT} follow from~\eqref{eq:pfcor-MRT:1} and Theorem~\ref{theo:Hill-Langevin}. To complete the proof of Corollary~\ref{cor:MRT}, it therefore remains to show~\eqref{eq:MRT:limps}, which we rewrite
\begin{equation}\label{eq:pfcor-MRT:2}
  \lim_{\ell \to +\infty} \frac{1}{\ell}\sum_{k=0}^{\ell-1} Z^\bfren_{AB,k} = \Exp_{\nu^\ret_{\mathcal{A}^-}}\left[Z^\bfren_{AB,0}\right], \qquad \text{almost surely,}
\end{equation}
where we have set
\begin{equation*}
  Z^\bfren_{AB,k} := \int_{\tau^\ret_{{\mathcal{A}^-},k}}^{\tau^\ret_{{\mathcal{B}^-},k}} G(q_s,p_s)\dd s.
\end{equation*}
In this purpose, we first remark that the sequence $(Y^\ret_{{\mathcal{A}^-},k}, Z^\bfren_{AB,k})_{k \geq 0}$ is a time homogeneous Markov chain in ${\mathcal{A}^-} \times [0,+\infty)$, with transition kernel
\begin{equation*}
  \mathbf{P}^\bfren_{AB}((y_0,z_0), \dd y \dd z) = P^\ret_{\mathcal{A}^-}(y_0,\dd y) Q^\bfren_{AB}(y, \dd z),
\end{equation*}
where $P^\ret_{\mathcal{A}^-}(y_0,\dd y)$ denotes the transition kernel of the chain $(Y^\ret_{{\mathcal{A}^-},k})_{k \geq 0}$, while for any $y \in {\mathcal{A}^-}$, $Q^\bfren_{AB}(y, \dd z)$ is the law of $Z^\bfren_{AB,0}$ under $\Pr_y$. The fact that $\mathbf{P}^\bfren_{AB}((y_0,z_0), \dd y \dd z)$ does not depend on $z_0$ expresses the fact that conditionally on $(Y^\ret_{{\mathcal{A}^-},k})_{k \geq 0}$, the variables $Z^\bfren_{AB,k}$ are independent and respectively depend only on $Y^\ret_{{\mathcal{A}^-},k}$, a fact which is reminiscent of the notion of \emph{Markov renewal process}~\cite[Section~VII.4]{Asm03}.

\begin{prop}[Positive Harris recurrence of $(Y^\ret_{{\mathcal{A}^-},k}, Z^\bfren_{AB,k})_{k \geq 0}$]\label{prop:phrMRP}
  Under the assumptions of Proposition~\ref{prop:nuA}, the Markov chain $(Y^\ret_{{\mathcal{A}^-},k}, Z^\bfren_{AB,k})_{k \geq 0}$ is positive Harris recurrent, and its invariant probability measure writes
  \begin{equation*}
    \boldsymbol{\nu}^\bfren_{AB}(\dd y\dd z) := \nu^\ret_{\mathcal{A}^-}(\dd y) Q^\bfren_{AB}(y, \dd z),
  \end{equation*}
  where we recall from Proposition~\ref{prop:nuA} that $\nu^\ret_{\mathcal{A}^-}$ is the invariant measure of $(Y^\ret_{{\mathcal{A}^-},k})_{k \geq 0}$.
\end{prop}
\begin{proof}
  We first check that the probability measure $\boldsymbol{\nu}^\bfren_{AB}$ is invariant for the Markov kernel $\mathbf{P}^\bfren_{AB}$. We have
  \begin{align*}
    \boldsymbol{\nu}^\bfren_{AB}\mathbf{P}^\bfren_{AB}(\dd y\dd z) &= \int_{(y_0,z_0) \in {\mathcal{A}^-} \times [0,+\infty)} \boldsymbol{\nu}^\bfren_{AB}(\dd y_0\dd z_0)\mathbf{P}^\bfren_{AB}((y_0,z_0), \dd y \dd z)\\
    &= \int_{(y_0,z_0) \in {\mathcal{A}^-} \times [0,+\infty)} \nu^\ret_{\mathcal{A}^-}(\dd y_0) Q^\bfren_{AB}(y_0, \dd z_0)P^\ret_{\mathcal{A}^-}(y_0,\dd y) Q^\bfren_{AB}(y, \dd z)\\
    &= \int_{y_0 \in {\mathcal{A}^-}} \left(\int_{z_0 \in [0,+\infty)} Q^\bfren_{AB}(y_0, \dd z_0)\right) \nu^\ret_{\mathcal{A}^-}(\dd y_0)P^\ret_{\mathcal{A}^-}(y_0,\dd y) Q^\bfren_{AB}(y, \dd z)\\
    &= \nu^\ret_{\mathcal{A}^-}(\dd y)Q^\bfren_{AB}(y, \dd z) = \boldsymbol{\nu}^\bfren_{AB}(\dd y\dd z),
  \end{align*}
  where, at the fourth line, we have used the facts that $Q^\bfren_{AB}(y_0, \cdot)$ is a probability measure on $[0,+\infty)$ and that $\nu^\ret_{\mathcal{A}^-}P^\ret_{\mathcal{A}^-}=\nu^\ret_{\mathcal{A}^-}$.
  
  It now remains to show that the chain $(Y^\ret_{{\mathcal{A}^-},k}, Z^\bfren_{AB,k})_{k \geq 0}$ is Harris recurrent. Since $(Y^\ret_{{\mathcal{A}^-},k})_{k \geq 0}$ is positive Harris recurrent, it admits a regeneration set $\mathcal{R}$. Let $\epsilon$, $r$ and $\lambda$ be associated with $\mathcal{R}$ in Definition~\ref{defi:reg-set}. We check that $\mathcal{R} \times [0,+\infty)$ defines a regeneration set for $(Y^\ret_{{\mathcal{A}^-},k}, Z^\bfren_{AB,k})_{k \geq 0}$. It is obvious that the second point of Definition~\ref{defi:reg-set} is satisfied. In order to address the first point, we note that the conditional structure of $(Y^\ret_{{\mathcal{A}^-},k}, Z^\bfren_{AB,k})_{k \geq 0}$ implies that the $r$-th iterate of the Markov kernel $\mathbf{P}^\bfren_{AB}$ writes
  \begin{equation*}
    \left(\mathbf{P}^\bfren_{AB}\right)^r\left((y_0,z_0),\dd y\dd z\right) = \left(P^\ret_{\mathcal{A}^-}\right)^r(y_0,\dd y) Q^\bfren_{AB}(y,\dd z).
  \end{equation*}
  This is checked by a direct computation. Therefore, for any $(y_0,z_0) \in \mathcal{R} \times [0,+\infty)$, 
  \begin{equation*}
    \left(\mathbf{P}^\bfren_{AB}\right)^r\left((y_0,z_0),\dd y \dd z\right) = \left(P^\ret_{\mathcal{A}^-}\right)^r(y_0,\dd y) Q^\bfren_{AB}(y,\dd z) \geq \epsilon \lambda(\dd y) Q^\bfren_{AB}(y,\dd z),
  \end{equation*}
  so that the first point of Definition~\ref{defi:reg-set} is satisfied with $\boldsymbol{\lambda}(\dd y\dd z) := \lambda(\dd y) Q^\bfren_{AB}(y,\dd z)$.
\end{proof}

As a consequence of Proposition~\ref{prop:phrMRP}, one may apply Proposition~\ref{prop:LLN-Harris} to the positive Harris recurrent chain $(Y^\ret_{{\mathcal{A}^-},k}, Z^\bfren_{AB,k})_{k \geq 0}$ with the function $\mathbf{f}(y,z) = z$ so as to obtain~\eqref{eq:pfcor-MRT:2}. This completes the proof of Corollary~\ref{cor:MRT}.

\subsection{Remarkable identities in the conservative case}\label{ss:rev-reactive} 
Let us conclude this section by complementary results which may be of interest to better understand the links between the probability measures we have introduced, in the conservative case $F=-\nabla V$. 

Until now, we considered in this section the process
$(Y_n^-)_{n \ge0}$ with values in $\mathcal{A}^- \cup \mathcal{B}^-$,
and the associated probability measures $\pi^-$,
 $\nu^{\ret}_{\mathcal{A}^-/\mathcal{B}^-}$ and $\nu^{\exi}_{\mathcal{A}^-/\mathcal{B}^-}$,
 which respectively correspond to the stationary measures of the
 successive entry points in $A \cup B$, of the first entrance in $A/B$
 coming from $B/A$, and of the last entrance in $A/B$ before going to
 $B/A$. Likewise, one can consider, with obvious notation,  the process
$(Y_n^+)_{n \ge0}$ with values in $\mathcal{A}^+ \cup \mathcal{B}^+$,
and the associated probability measures $\pi^+$,
 $\nu^{\ret}_{\mathcal{A}^+/\mathcal{B}^+}$ and $\nu^{\exi}_{\mathcal{A}^+/\mathcal{B}^+}$,
 which respectively correspond to the stationary measures of the
 successive exit points from $A \cup B$, of the first exit from $A/B$
 coming from $B/A$, and of the last exit from $A/B$ before going to
 $B/A$. Then the following statement easily stems from the combination of the reversibility identities from Corollary~\ref{cor:rev} with the explicit expressions of the reactive distributions from Proposition~\ref{prop:reactive}. We also refer to Figure~\ref{fig:proposition57} for
an illustration.

\begin{prop}[Links between reactive exit and entrance distributions in the conservative case]\label{prop:momentum_reversal}
  In the setting of Proposition~\ref{prop:nuA} with $F=-\nabla V$, one has the following
identities:
$$\nu^{\exi}_{\mathcal{A}^-} =   \nu^{\ret}_{\mathcal{A}^+} \circ \mathsf{R}^{-1}
\text{ and } \nu^{\exi}_{\mathcal{A}^+}=  \nu^{\ret}_{\mathcal{A}^-}\circ \mathsf{R}^{-1},$$
$$\nu^{\ret}_{\mathcal{B}^-}= \nu^{\exi}_{\mathcal{B}^+}\circ \mathsf{R}^{-1} \text{
  and }\nu^{\ret}_{\mathcal{B}^+} = \nu^{\exi}_{\mathcal{B}^-}\circ \mathsf{R}^{-1}.$$
\end{prop}

\begin{psfrags}
  \psfrag{A}{$A$}
  \psfrag{B}{$B$}
  \psfrag{YexB-}{\small$Y^\exi_{{\mathcal{B}^-},k}$}
  \psfrag{YexB+}{\small$Y^\exi_{{\mathcal{B}^+},k}$}
  \psfrag{YreA-}{\small$Y^\ret_{{\mathcal{A}^-},k}$}
  \psfrag{YreA+}{\small $Y^\ret_{{\mathcal{A}^+},k}$}
  \psfrag{YexA-}{\small$Y^\exi_{{\mathcal{A}^-},k+1}$}
  \psfrag{YexA+}{\small$Y^\exi_{{\mathcal{A}^+},k+1}$}
  \psfrag{YreB+}{\small$Y^\ret_{{\mathcal{B}^+},k+1}$}
  \psfrag{YreB-}{\small$Y^\ret_{{\mathcal{B}^-},k+1}$}
\begin{figure}
\begin{center}
    \includegraphics[width=0.7\textwidth]{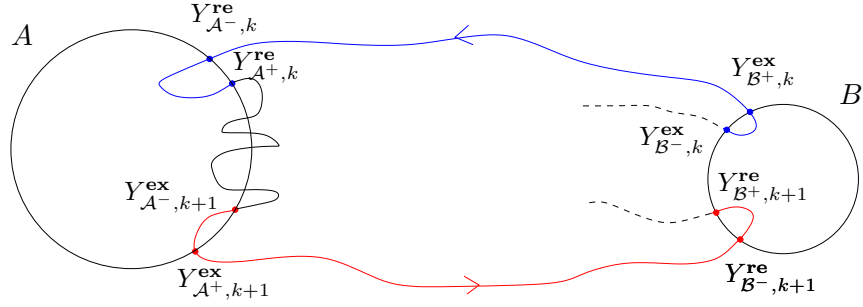}
\caption{Schematic representation of a reactive trajectory from $B$ to
$A$ (in blue), and of a reactive trajectory from $A$ to $B$ (in
red). At stationarity and in the conservative case, the law of a reactive trajectory from $B$ to $A$,
reversed in time and
after momentum reversal, is the same as the law of a reactive trajectory from $A$ to $B$.}
\label{fig:proposition57}
\end{center}
\end{figure}
\end{psfrags}

\appendix
\section{Generalities on Harris recurrent chains}\label{app:Harris}

In this appendix we gather various definitions and facts related with the notion of \emph{Harris recurrent chain}. Most statements are taken from~\cite[Section~VII.3]{Asm03} and~\cite[Chapter~4]{HerLas03}. We also refer to~\cite{MeyTwe09} and~\cite{DouMouPriSou18} for extensive monographs.

\subsection{Harris recurrence} Let $(Y_n)_{n \geq 0}$ be a time homogeneous Markov chain taking its values in a measurable space $\mathcal{S}$, which is assumed to be \emph{separable}, that is to say that the $\sigma$-field on $\mathcal{S}$ is generated by a countable collection of sets~\cite[Section~4.1]{HerLas03}. The transition kernel of the chain is denoted by~$P$. 

\begin{defi}[$\varphi$-recurrence]\label{defi:phi-rec}
  Let $\varphi$ be a $\sigma$-finite nonnegative measure on $\mathcal{S}$. The Markov chain $(Y_n)_{n \geq 0}$ is called \emph{$\varphi$-recurrent} if, for any $y \in \mathcal{S}$ and any measurable subset $\mathcal{C} \subset \mathcal{S}$ such that $\varphi(\mathcal{C})>0$, then
  \begin{equation*}
    \sum_{n \geq 0} \ind{Y_n \in \mathcal{C}} = +\infty, \qquad \text{$\Pr_y$-almost surely.}
  \end{equation*}
\end{defi}

\begin{defi}[Regeneration set]\label{defi:reg-set}
  A set $\mathcal{R} \subset \mathcal{S}$ is called a \emph{regeneration set} for the Markov chain $(Y_n)_{n \geq 0}$ if:
  \begin{enumerate}[label=(\roman*),ref=\roman*]
    \item there exist $\epsilon>0$, $r\geq 1$ and a probability measure $\lambda$ on $\mathcal{S}$ such that
    \begin{equation*}
      \forall y \in \mathcal{R}, \qquad \Pr_y(Y_r \in \cdot) \geq \epsilon \lambda(\cdot);
    \end{equation*}
    \item for any $y \in \mathcal{S}$, $\Pr_y(\exists n \geq 1: Y_n \in \mathcal{R})=1$.
  \end{enumerate}
\end{defi}

The next result allows to define the notion of Harris recurrence. It may be found in~\cite[Corollary~3.12, p.~205]{Asm03}.

\begin{prop}[Harris recurrence]\label{prop:Harris}
  The Markov chain $(Y_n)_{n \geq 0}$ is $\varphi$-recurrent for some nontrivial $\varphi$ if and only if it has a regeneration set. In this case, it is called \emph{Harris recurrent}, and:
  \begin{enumerate}[label=(\roman*),ref=\roman*]
    \item it has a unique, up to multiplicative constant, $\sigma$-finite nonnegative invariant measure $\Pi$;
    \item any measurable set $\mathcal{C}$ such that $\Pi(\mathcal{C})>0$ contains a regeneration set.
  \end{enumerate}
\end{prop}

Under the assumptions of Proposition~\ref{prop:Harris}, if $\Pi(\mathcal{S})<+\infty$ then the chain has a unique invariant probability measure $\pi(\cdot)=\Pi(\cdot)/\Pi(\mathcal{S})$, and it is called \emph{positive Harris recurrent}. 

The next result, based on~\cite[Theorem~4.2.13, p.~51]{HerLas03}, shows that positive Harris recurrence is equivalent to the fact that the ergodic theorem for Markov chains holds for \emph{all} initial conditions.

\begin{prop}[Positive Harris recurrence and ergodic theorem]\label{prop:LLN-Harris}
  A Markov chain $(Y_n)_{n \geq 0}$ is positive Harris recurrent, with invariant probability measure $\pi$, if and only if, for any $y \in \mathcal{S}$ and $f \in L^1(\mathcal{S},\pi)$, 
  \begin{equation*}
    \lim_{n \to +\infty} \frac{1}{n}\sum_{k=0}^{n-1} f(Y_k) = \pi(f), \qquad \text{$\Pr_y$-almost surely.}
  \end{equation*}
\end{prop}

\begin{rk}\label{rk:pi-rec}
  It follows from Proposition~\ref{prop:LLN-Harris} that a positive Harris recurrent chain with invariant probability measure $\pi$ is $\pi$-recurrent in the sense of Definition~\ref{defi:phi-rec}.
\end{rk}

\subsection{Trace chain} Let $(Y_n)_{n \geq 0}$ be a positive Harris
recurrent chain with invariant probability measure $\pi$ on
$\mathcal{S}$. Let $\mathcal{C}$ be a measurable subset of
$\mathcal{S}$ such that $\pi(\mathcal{C})>0$. Then by
Remark~\ref{rk:pi-rec}, the set $\{n \geq 0: Y_n \in \mathcal{C}\}$ is
infinite almost surely. We denote by $(\eta_k)_{k \geq 0}$ the
increasing ordering of its elements, namely the successive return
times of the chain in $\mathcal{C}$:
$$\eta_0=\inf\{n \ge 0, \, Y_n \in \mathcal{C}\} \qquad \text{ and } \quad \forall k \ge 0,
\eta_{k+1}=\inf\{n > \eta_k,\,  Y_n \in \mathcal{C}\}.$$ 
 It is clear that they are stopping times for the natural filtration of $(Y_n)_{n \geq 0}$, and by the strong Markov property, the sequence $(Y_{\eta_k})_{k \geq 0}$ is a time homogeneous Markov chain with values in $\mathcal{C}$. It is called the \emph{trace} of the chain on the set $\mathcal{C}$. The next result is a straightforward consequence of Proposition~\ref{prop:LLN-Harris}.

\begin{lem}[Positive Harris recurrence for trace chains]\label{lem:trace-Harris}
  In the setting described above, the trace chain $(Y_{\eta_k})_{k \geq 0}$ is positive Harris recurrent, and its invariant probability measure is the conditional measure $\pi_\mathcal{C}(\cdot) := \pi(\cdot|\mathcal{C})$.
\end{lem}

Our derivation of the Hill formula in Section~\ref{s:Hill} relies on the use of the representation formula for $\pi$ given in the next proposition. This formula can be seen as a generalisation to the continuous state space setting of the standard identity
\begin{equation*}
  \Exp_y\left[\sum_{n=0}^{\sigma_y-1} g(Y_n)\right] = \frac{\pi(g)}{\pi(y)}, \qquad \sigma_y := \inf\{n \geq 1: Y_n=y\},
\end{equation*}
for Markov chains $(Y_n)_{n \geq 0}$ which take their values in discrete spaces~\cite[Lemma~4.21, p.~76]{BovDen15}.

\begin{prop}[Representation formula for $\pi$]\label{prop:rep-pi}
  In the setting and with the notation of Lemma~\ref{lem:trace-Harris}, for any $g \in L^1(\mathcal{S},\pi)$, we have
  \begin{equation*}
    \Exp_{\pi_\mathcal{C}}\left[\sum_{n=0}^{\eta_1-1} |g(Y_n)|\right] < +\infty,
  \end{equation*}
  and
  \begin{equation*}
    \Exp_{\pi_\mathcal{C}}\left[\sum_{n=0}^{\eta_1-1} g(Y_n)\right] = \frac{\pi(g)}{\pi(\mathcal{C})}.
  \end{equation*}
\end{prop}
\begin{proof}
  Theorem~3.6.5, p.~71 in~\cite{DouMouPriSou18} asserts that the second identity holds in $[0,+\infty)$ for any nonnegative measurable function $g$ on $\mathcal{S}$. Therefore, letting $g \in L^1(\mathcal{S},\pi)$ and applying this identity to $[g]_-$ and $[g]_+$ separately, we get the claimed statements.
\end{proof}

\subsection{Pair chain} Let $(Y_n)_{n \geq 0}$ be a time homogeneous Markov chain on $\mathcal{S}$, with transition kernel $P$. It is known that the chain $(\bar{Y}_n)_{n \geq 0}$ defined by $\bar{Y}_n = (Y_n,Y_{n+1})$ is a time homogeneous Markov chain on $\mathcal{S} \times \mathcal{S}$, with transition kernel $\bar{P}$ given by
\begin{equation*}
  \bar{P}\,\bar{f}(y_0,y_1) := \int_{\mathcal{S}} \bar{f}(y_1,y_2)P(y_1,\dd y_2),
\end{equation*}
for any measurable and bounded function $\bar{f}$ on $\mathcal{S} \times \mathcal{S}$. It is called the \emph{pair chain}. The following result is taken from the proof of~\cite[Proposition~3.9]{BauGuyLel}.

\begin{lem}[Harris recurrence for the pair chain]\label{lem:pair-Harris}
  If the chain $(Y_n)_{n \geq 0}$ is positive Harris recurrent, with invariant probability measure $\pi$, then the pair chain $(\bar{Y}_n)_{n \geq 0}$ is positive Harris recurrent with invariant probability measure $\bar{\pi} := \pi \otimes P$.
\end{lem}

\section{Proof of Proposition~\ref{prop:Dirichlet-Langevin}}\label{app:DL}

The proofs of both statements~\eqref{it:propDL:1} and~\eqref{it:propDL:2} in Proposition~\ref{prop:Dirichlet-Langevin} follow from symmetric arguments. Therefore, we only detail the proof of the first statement. We fix $f^\bfen$ a continuous and bounded function on $\Gamma^\bfen \cup \Gamma^0$ and define $u^\bfen$ on $\R^d \times \R^d$ accordingly. It is then obvious that $u^\bfen = f^\bfen$ on $\Gamma^\bfen \cup \Gamma^0$. We show that, on $(\R^d \setminus \bar{\metasp}) \times \R^d$, $u^\bfen$ is $C^\infty$ and satisfies $\mathcal{L}u^\bfen = 0$ in Subsection~\ref{ss:u-harm}, and that $u^\bfen$ is continuous on $(\R^d \setminus \metasp) \times \R^d$ in Subsection~\ref{ss:u-cont}. The proofs use arguments from~\cite[Theorem~2.10]{LelRamRey:kfp}, where the time dependent equation $\partial_t u = \mathcal{L}u$ on a bounded domain in $q$ was considered.

\subsection{Harmonicity of \texorpdfstring{$u^\bfen$}{u-}}\label{ss:u-harm}

The main step in our argument is the following result.

\begin{prop}[Weak harmonicity of $u^\bfen$]\label{prop:pfDL}
  Under the assumptions of the first statement in Proposition~\ref{prop:Dirichlet-Langevin}, let $\Phi$ be a $C^\infty$ function with compact support in the open set $(\R^d \setminus \bar{\metasp}) \times \R^d$. Then
  \begin{equation*}
    \int_{(\R^d \setminus \bar{\metasp}) \times \R^d} u^\bfen(q,p) \mathcal{L}^*\Phi(q,p) \dd q \dd p = 0,
  \end{equation*}
  where we recall that the differential operator $\mathcal{L}^*$ is defined in~\eqref{eq:Lstar}.
\end{prop}

Since, under the assumptions of Proposition~\ref{prop:Dirichlet-Langevin}, the differential operator $\mathcal{L}$ is known to be hypoelliptic, Proposition~\ref{prop:pfDL} implies that $u^\bfen$ is $C^\infty$ on the open set $(\R^d \setminus \bar{\metasp}) \times \R^d$ and satisfies $\mathcal{L}u^\bfen=0$ there. Thus we now focus on the proof of Proposition~\ref{prop:pfDL}.  Throughout the remainder of this appendix, we take the convention to denote by $x=(q,p)$ generic elements of $\R^d \times \R^d$, and by $X_t = (q_t,p_t)$ the solution to~\eqref{eq:Langevin}.

Let $(\tilde{W}_t)_{t\geq0}$ be a $d$-dimensional Brownian motion,
independent from $(W_t)_{t\geq0}$. For any $m \geq 0$, let $F_m$ be a
smooth and compactly supported vector field on $\R^d$, which coincides
with $F$ on the open ball $\mathrm{B}(0,m)$ and which is such that
$|F_m(q)| \leq |F(q)|$ for any $q \in \R^d$. For $\epsilon \geq 0$, let $(X^{\epsilon,m}_t=(q^{\epsilon,m}_t,p^{\epsilon,m}_t))_{t\geq0}$ be the strong solution, defined on the same probability space as $(X_t)_{t \geq 0}$, to the stochastic differential equation
\begin{equation}\label{eq:Langevin-eps-m}
  \left\{\begin{aligned}
    \dd q^{\epsilon,m}_t &= p^{\epsilon,m}_t \dd t + \sqrt{2\epsilon} \dd \tilde{W}_t,\\
    \dd p^{\epsilon,m}_t &= F_m(q^{\epsilon,m}_t)\dd t - \gamma p^{\epsilon,m}_t \dd t + \sqrt{2\gamma\beta^{-1}}\dd W_t.
  \end{aligned}\right.
\end{equation}
We denote by $\mathcal{L}_{\epsilon,m}$ the associated infinitesimal generator. We write $X^m_t := X^{0,m}_t$ and $\mathcal{L}_m:=\mathcal{L}_{0,m}$.

Let $(V_k)_{k\geq1}$ be a nondecreasing sequence of $C^2$ open bounded subsets of $(\R^d \setminus \bar{\metasp}) \times \R^d$ such that $\cup_{k\geq1}V_k=(\R^d \setminus \bar{\metasp}) \times \R^d$. For $k\geq1$, $m \geq 0$ and $\epsilon>0$, let us define the stopping times
\begin{align*}
\tau^{\epsilon,m,k}&:=\inf\{t> 0: X^{\epsilon,m}_t \notin V_k\},\\
\tau^{\epsilon,m}&:=\inf\{t> 0: X^{\epsilon,m}_t \not\in (\R^d \setminus \bar{\metasp}) \times \R^d\},\\
\tau^{m}&:=\inf\{t> 0: X^m_t \not\in (\R^d \setminus \bar{\metasp}) \times \R^d\}.
\end{align*} 

\begin{lem}[Approximation results in $k$ and $\epsilon$]\label{lem: tau convergence}
 For all $x\in(\R^d \setminus \bar{\metasp}) \times \R^d$, for all $m
 \ge 0$, one has $\Pr_{x}$-almost surely,
\begin{enumerate}[label=(\roman*),ref=\roman*]
    \item\label{cv 1 recrossing} $\lim_{k\to\infty}\tau^{\epsilon,m,k}=\tau^{\epsilon,m},$
    \item\label{cv 2 recrossing} for any $t>0$, $\lim_{\epsilon\to 0} \ind{\tau^{\epsilon,m}\leq t}=\ind{\tau^{m}\leq t},$
    \item\label{cv 3 recrossing} on the event $\{\tau^{m}<\infty\}$, $\lim_{\epsilon\to 0}X^{\epsilon,m}_{\tau^{\epsilon,m}}=X^{m}_{\tau^{m}}.$ 
\end{enumerate}  
\end{lem} 
\begin{proof} 
  In this proof we fix $m\geq 0$ and $x\in(\R^d \setminus \bar{\metasp}) \times \R^d$. 
  
  Let $\epsilon>0$. Since $(V_k)_{k\geq1}$ is a nondecreasing sequence of open subsets of $(\R^d \setminus \bar{\metasp}) \times \R^d$, then $(\tau^{\epsilon,m,k})_{k\geq1}$ is a nondecreasing sequence of stopping times, bounded from above by $\tau^{\epsilon,m}$, therefore it converges $\Pr_{x}$-almost surely toward $\sup_{k\geq1}\tau^{\epsilon,m,k} \leq \tau^{\epsilon,m}$. Besides, it follows from the equality $\cup_{k\geq1}V_k=(\R^d \setminus \bar{\metasp}) \times \R^d$ that $\Pr_{x}$-almost surely $\sup_{k\geq1}\tau^{\epsilon,m,k}=\tau^{\epsilon,m}$, hence~\eqref{cv 1 recrossing}.
  
  Now let us fix $t>0$. In order to prove~\eqref{cv 2 recrossing} it is sufficient to show the convergence on the partition of events $\{\tau^m<t\}$, $\{\tau^m>t\}$ and $\{\tau^m=t\}$. Since $\Pr_x(\tau^m=t)\leq\Pr_x(X^m_t\in \Sigma \times \R^d)=0$, one only needs to prove the convergence almost surely on the events $\{\tau^m<t\}$ and $\{\tau^m>t\}$. 
  
  Since $F_m$ is globally Lipschitz continuous on $\R^d$, using Gronwall's lemma one obtains the existence of a constant $C_m>0$ such that for all $x\in \R^d\times \R^d$, $\Pr_{x}$-almost surely, for all $t\geq0$, 
  \begin{equation}\label{ineq gronwall recrossing}
    \sup_{s\in[0,t]}\vert X^{\epsilon,m}_s-X^{m}_s\vert\leq\sqrt{2\epsilon}\sup_{s\in[0,t]}\vert\tilde{W}_s\vert\mathrm{e}^{C_mt}.
  \end{equation} 
  For a fixed $t>0$, we consider the event $\{\tau^m<t\}$. By Lemma~\ref{lem:return}, $X^m_{\tau^{m}}\in\Gamma^\bfen$, $\Pr_{x}$-almost surely. Besides, the process $(X^{m}_t)_{t\geq0}$ visits $\metasp \times \R^d$, $\Pr_{x}$-almost surely on $[\tau^{m},\tau^{m}+\alpha]$ for any $\alpha>0$. Therefore, for $\alpha$ small enough so that $\tau^{m}+\alpha<t$ and for $\epsilon$ small enough, one has, using~\eqref{ineq gronwall recrossing}, that $\tau^{\epsilon,m}<t$. This ensures the convergence~\eqref{cv 2 recrossing} on the event $\{\tau^m<t\}$.
  
  Assume now that $\{\tau^m>t\}$. It follows from~\eqref{ineq gronwall recrossing} that for $\epsilon$ small enough, the process $(X^{\epsilon,m}_s)_{s\in[0,t]}$ is at a positive distance from $\metasp \times \R^d$. Therefore, one has that $\tau^{\epsilon,m}>t$ which ensures the convergence~\eqref{cv 2 recrossing} on the event $\{\tau^m>t\}$. 
  
  On the event $\{\tau^m < \infty\}$, by~\eqref{cv 2 recrossing} we have $\tau^{\epsilon,m} < \infty$ for $\epsilon$ small enough, thus we deduce from~\eqref{ineq gronwall recrossing} and the continuity of the trajectory of $(X^m_t)_{t \geq 0}$ that $X^{\epsilon,m}_{\tau^{\epsilon,m}}$ converges to $X^{m}_{\tau^{m}}$, which is the assertion~\eqref{cv 3 recrossing}.
\end{proof}

We are now ready to prove Proposition~\ref{prop:pfDL}.

\begin{proof}[Proof of Proposition~\ref{prop:pfDL}]
  We extend $f^\bfen$ to a function which is continuous and bounded on the whole space $\R^d \times \R^d$, and for $\epsilon>0$, $k\geq1$ and $m\geq 0$, we set
  \begin{equation*}
    u^\bfen_{\epsilon,m,k}(x) := \Exp_{x}\left[f^\bfen(X^{\epsilon,m}_{\tau^{\epsilon,m,k}})\right],
  \end{equation*}
  for any $x \in V_k$. By \cite[Theorem~5.1 in Chapter~6]{Fri75}, $\tau^{\epsilon,m,k}<\infty$ almost surely and $u^\bfen_{\epsilon,m,k}$ is $C^\infty$ on $V_k$ and satisfies $\mathcal{L}_{\epsilon,m}u^\bfen_{\epsilon,m,k}=0$ there. Now there exists $k_0$ such that for all $k\geq k_0$, the support of $\Phi$ is contained in $V_k$. As a result, for all $k\geq k_0$, a direct integration by parts yields
  \begin{equation}\label{solution distrib recrossing}
    \int_{(\R^d \setminus \bar{\metasp}) \times \R^d} u^\bfen_{\epsilon,m,k}(x) \mathcal{L}^*_{\epsilon,m}\Phi(x)\dd x=0,
  \end{equation}  
  where
  \begin{equation*}
    \mathcal{L}^*_{\epsilon,m}\Phi(x) = -p \cdot \nabla_q \Phi - \nabla_p \cdot (F_m(q)\Phi) + \gamma \nabla_p \cdot (p \Phi) + \gamma\beta^{-1}\Delta_p \Phi + \epsilon \Delta_q \Phi, \qquad x=(q,p).
  \end{equation*}
  When $\epsilon \to 0$ and $m \to +\infty$, this quantity converges
  to $\mathcal{L}^*\Phi(x)$ for any $x \in (\R^d \setminus
  \bar{\metasp}) \times \R^d$, and since for $m$ sufficiently large, $F_m$ coincides with
    $F$ on the support of $\Phi$, there is a constant $N(\Phi)$ which
  depends neither on $\epsilon$ nor on $m$ such that
  $|\mathcal{L}^*_{\epsilon,m}\Phi(x)| \leq N(\Phi)\ind{x \in
    V_{k_0}}$. Besides, by construction the function
  $u^\bfen_{\epsilon,m,k}$ is bounded uniformly in $k$, $\epsilon$ and
  $m$. Therefore, using the dominated convergence theorem
  in~\eqref{solution distrib recrossing}, the proof of the proposition
  is a consequence of the following statement: for any $x \in (\R^d \setminus \bar{\metasp}) \times \R^d$,
  \begin{equation*}
    \lim_{m\to\infty}\lim_{\epsilon\to0}\lim_{k\to\infty}u^\bfen_{\epsilon,m,k}(x)=u^\bfen(x).
  \end{equation*}
  Let us now prove this result.  For $x\in (\R^d \setminus \bar{\metasp}) \times \R^d$ and $t>0$, let
  \begin{equation*}
    u^{\bfen,(1)}_{\epsilon,m,k}(t,x):=\Exp_{x}\left[f^\bfen(X^{\epsilon,m}_{\tau^{\epsilon,m,k}})\ind{\tau^{\epsilon,m,k}\leq t}\right],\qquad u^{\bfen,(2)}_{\epsilon,m,k}(t,x):=\Exp_{x}\left[f^\bfen(X^{\epsilon,m}_{\tau^{\epsilon,m,k}})\ind{\tau^{\epsilon,m,k}> t}\right],
  \end{equation*}
  so that $u^\bfen_{\epsilon,m,k}(x)=u^{\bfen,(1)}_{\epsilon,m,k}(t,x)+u^{\bfen,(2)}_{\epsilon,m,k}(t,x)$. As a result, it is enough to prove that
  \begin{equation}\label{cv u1 u2 recrossing}
    \lim_{t\to\infty}\lim_{m\to\infty}\lim_{\epsilon\to0}\lim_{k\to\infty}u^{\bfen,(1)}_{\epsilon,m,k}(t,x)=u^\bfen(x),\quad\limsup_{t\to\infty}\limsup_{m\to\infty}\limsup_{\epsilon\to0}\limsup_{k\to\infty}|u^{\bfen,(2)}_{\epsilon,m,k}(t,x)|=0.
  \end{equation}
    
  Since $k\mapsto \tau^{\epsilon,m,k}$ is nondecreasing and $s\mapsto\ind{s\leq t}$ is left-continuous, one has using~\eqref{cv 1 recrossing} in Lemma~\ref{lem: tau convergence} that almost surely, $\ind{\tau^{\epsilon,m,k}\leq t}$ converges to $\ind{\tau^{\epsilon,m}\leq t}$, hence
  \begin{equation*}
    u^{\bfen,(1)}_{\epsilon,m,k}(t,x) \underset{k\to\infty}{\longrightarrow} \Exp_{x}\left[f^\bfen(X^{\epsilon,m}_{\tau^{\epsilon,m}})\ind{\tau^{\epsilon,m}\leq t}\right].
  \end{equation*} 
  Furthermore, by~\eqref{cv 2 recrossing} and~\eqref{cv 3 recrossing} in Lemma~\ref{lem: tau convergence},
\begin{align*}
    &\left\vert\Exp_{x}\left[f^\bfen(X^{\epsilon,m}_{\tau^{\epsilon,m}})\ind{\tau^{\epsilon,m}\leq t}\right]-\Exp_{x}\left[f^\bfen(X^{m}_{\tau^{m}})\ind{\tau^{m}\leq t}\right]\right\vert\\
    &\leq\Vert f^\bfen\Vert_\infty\Exp_{x}\left[\left\vert\ind{\tau^{\epsilon,m}\leq t}-\ind{\tau^{m}\leq t}\right\vert\right]+\Exp_{x}\left[\ind{\tau^{m}\leq t}\left\vert f^\bfen(X^{\epsilon,m}_{\tau^{\epsilon,m}})-f^\bfen(X^{m}_{\tau^{m}})\right\vert\right]\underset{\epsilon\to0}{\longrightarrow}0.
\end{align*} 

  For $m\geq 0$, let $B_m:=\mathrm{B}(0,m)$, $\tau^m_{B_m^c}:=\inf\{t> 0:X^{m}_t\notin B_m \times \R^d\}$ and $\tau_{B_m^c}:=\inf\{t> 0:X_t\notin B_m \times \R^d\}$. One has for $t>0$,
  \begin{equation*}
    \Exp_{x}\left[f^\bfen(X^{m}_{\tau^{m}})\ind{\tau^{m}\leq t}\right]=\Exp_{x}\left[f^\bfen(X^{m}_{\tau^{m}})\ind{\tau^{m}\leq\tau^m_{B_m^c}\land t}\right]+\Exp_{x}\left[f^\bfen(X^{m}_{\tau^{m}})\ind{\tau^{m}>\tau^m_{B_m^c}}\ind{\tau^{m}\leq t}\right].
  \end{equation*}  
  Besides, since $F_m$ and $F$ coincide on $B_m$, the trajectories of $(X^m_t)_{t\geq0}$ and $(X_t)_{t\geq 0}$ coincide until $\tau^m_{B_m^c}$ and thus $\tau^m_{B_m^c}=\tau_{B_m^c}$, $\Pr_x$-almost surely. Therefore, for all $m\geq0$, 
  \begin{equation*}
    \Exp_{x}\left[f^\bfen(X^m_{\tau^{m}})\ind{\tau^{m}\leq\tau^m_{B^c_m}\land t}\right]=\Exp_{x}\left[f^\bfen(X_{\tau^\bfen_0})\ind{\tau^\bfen_0\leq\tau_{B_m^c}\land t}\right]\underset{m\to\infty}{\longrightarrow}\Exp_{x}\left[f^\bfen(X_{\tau^\bfen_0})\ind{\tau^\bfen_0\leq t}\right],
  \end{equation*}
  since $\tau^\bfen_0<\infty$ by Lemma~\ref{lem:return} and $\tau_{B_m^c}\underset{m\to\infty}{\longrightarrow}\infty$ by Assumption~\eqref{ass:A1}, $\Pr_x$-almost surely. Moreover,  
  \begin{align*}
    \Exp_{x}\left[f^\bfen(X^{m}_{\tau^{m}})\ind{\tau^{m}>\tau^m_{B_m^c}}\ind{\tau^{m}\leq t}\right]&\leq\Vert f^\bfen\Vert_\infty\Pr(\tau^{m}>\tau^m_{B_m^c}) \leq\Vert f^\bfen\Vert_\infty\Pr(\tau^\bfen_0>\tau_{B_m^c})\underset{m\to\infty}{\longrightarrow}0 \label{cv 2eme terme u_m},
  \end{align*}
  by Assumption~\eqref{ass:A1} and the fact that
    $\tau^\bfen_0<\infty$ almost surely. Finally, using again that $\tau^\bfen_0<\infty$, one deduces that 
  \begin{equation*}
    \Exp_{x}\left[f^\bfen(X_{\tau^\bfen_0})\ind{\tau^\bfen_0\leq t}\right]\underset{t\to\infty}{\longrightarrow}u^\bfen(x),
  \end{equation*}
  which ensures the first part of~\eqref{cv u1 u2 recrossing}. 
  
  Consider now the convergence of $u^{\bfen,(2)}_{\epsilon,m,k}(t,x)$. We have
  \begin{equation*}
    \vert u^{\bfen,(2)}_{\epsilon,m,k}(t,x)\vert\leq\Vert
    f^\bfen\Vert_\infty\Pr_x(\tau^{\epsilon,m,k}>
    t)\underset{k\to\infty}{\longrightarrow}\Vert f^\bfen\Vert_\infty \Pr_x(\tau^{\epsilon,m}> t),
  \end{equation*}
  by~\eqref{cv 1 recrossing} in Lemma~\ref{lem: tau convergence}. Besides, $\Pr_x(\tau^{\epsilon,m}> t)\underset{\epsilon\to0}{\longrightarrow}\Pr_x(\tau^{m}> t)$ by~\eqref{cv 2 recrossing} in Lemma~\ref{lem: tau convergence}. In addition,
  \begin{align*}
    \Pr_x(\tau^{m}> t)&=\Pr_x(\tau^{m}> t,\tau^{m}>\tau^m_{B_m^c})+\Pr_x(\tau^{m}> t,\tau^{m}\leq\tau^m_{B_m^c})\\
    &\leq\Pr_x(\tau^\bfen_0>\tau_{B_m^c})+\Pr_x(\tau^\bfen_0>t)\underset{m\to\infty}{\longrightarrow}\Pr_x(\tau^\bfen_0> t),
  \end{align*}
  since $\tau^\bfen_0<\infty$, $\Pr_x$-almost surely. Finally, since $\Pr_x(\tau^\bfen_0> t)$ vanishes when $t \to \infty$, one obtains the second part of~\eqref{cv u1 u2 recrossing}.
\end{proof}

\subsection{Continuity of \texorpdfstring{$u^\bfen$}{u-}}\label{ss:u-cont}

In this section, we let $(x^n)_{n \geq 1}$ and $x^\infty$ be configurations in $(\R^d \setminus \metasp) \times \R^d$ such that $x^n$ converges to $x^\infty$, and prove that $u^\bfen(x^n)$ converges to $u^\bfen(x^\infty)$. To proceed, for any $n \in \{1, \ldots, \infty\}$ we denote by $(X^n_t)_{t \geq 0}$ the strong solution to~\eqref{eq:Langevin}, with initial condition $x^n$, and assume that all these processes are defined on the same probability space and with respect to the same Brownian motion. We denote by 
\begin{equation*}
  \tau^n = \inf\{t \geq 0: X^n_t \in \Gamma^\bfen\}    
\end{equation*}
the corresponding realisation of $\tau^\bfen_0$, so that $u^\bfen(x^n) = \Exp[f^\bfen(X^n_{\tau^n})]$. Since $f^\bfen$ is continuous and bounded, the desired convergence directly stems from the following trajectorial result.

\begin{lem}[Continuity of the exit configuration]
  In the setting described above, 
  \begin{equation*}
      \lim_{n \to +\infty} X^n_{\tau^n} = X^\infty_{\tau^\infty}, \qquad \text{almost surely.}
  \end{equation*}
\end{lem}
\begin{proof}
  Since the coefficients of the stochastic differential equation~\eqref{eq:Langevin} are assumed to be locally Lipschitz continuous, the Gronwall lemma implies that, almost surely, for any $T>0$, there exists a finite random constant $C(T)$ such that
  \begin{equation}\label{eq:lipX}
      \forall n \geq 1, \qquad \sup_{t \in [0,T]} |X^n_t - X^\infty_t| \leq C(T)|x^n-x^\infty|.
  \end{equation}
  It is easy to deduce from this estimate that 
  \begin{equation*}
      \liminf_{n \to +\infty} \tau^n \geq \tau^\infty, \qquad \text{almost surely.}
  \end{equation*}
Besides, the proof of Lemma~\ref{lem:return} shows that, almost surely, for any $\epsilon>0$ there exists $t \in (\tau^\infty,\tau^\infty+\epsilon)$ such that $q^\infty_t \in \metasp$, and therefore by~\eqref{eq:lipX},
  \begin{equation*}
      \limsup_{n \to +\infty} \tau^n \leq \tau^\infty + \epsilon, \qquad \text{almost surely.}
  \end{equation*}
  As a conclusion, $\tau^n$ converges to $\tau^\infty$, almost surely, and the final claim follows from~\eqref{eq:lipX} again, using the fact that $\tau^\infty<\infty$ almost surely by Lemma~\ref{lem:return}.
\end{proof}

\section{Discussion of Assumption~(\ref{ass:D})}\label{app:Etau}

In this appendix, we provide sufficient conditions for
Assumption~(\ref{ass:D}) to hold. Moreover, in the conservative case ($F=-\nabla V$), we use results by Kopec~\cite{Kop15} to show that these conditions are satisfied under mild assumptions on the potential $V$.

Let us denote by $C^\infty_\mathrm{pol}(\R^d \times \R^d)$ the set of $C^\infty$ functions on $\R^d \times \R^d$ with polynomially growing derivatives of all orders. Under the assumptions of Theorem~\ref{theo:main}, we introduce the conditions:
\begin{enumerate}[label=(D'\arabic*),ref=D'\arabic*]
    \item\label{ass:D1} for any $k \geq 0$, $\sup_{t \geq 0} \Exp_{\pi^\bfen}[|q_t|^k + |p_t|^k] < +\infty$ and $\sup_{t \geq 0} \Exp_{\pi^\bfex}[|q_t|^k + |p_t|^k] < +\infty$;
    \item\label{ass:D2} for any $\tilde{\varphi} \in C^\infty_\mathrm{pol}(\R^d \times \R^d)$ such that $\mu(\tilde{\varphi})=0$, there exists a solution $\Phi \in C^\infty_\mathrm{pol}(\R^d \times \R^d)$ to the Poisson equation $\mathcal{L}\Phi = \tilde{\varphi}$, where we recall that $\mathcal{L}$ is the infinitesimal generator associated with~\eqref{eq:Langevin} and $\mu$ is the stationary distribution of $(q_t,p_t)_{t \geq 0}$.
\end{enumerate}

\begin{prop}[Sufficient condition for Assumption~(\ref{ass:D})]\label{prop:Etau}
   In the setting of Theorem~\ref{theo:main}, if Assumptions~(\ref{ass:D1}--\ref{ass:D2}) hold then Assumption~(\ref{ass:D}) holds.
\end{prop}
\begin{proof}
   Since $\Exp_{\pi^\bfen}[\tau_0^\bfex]$ and $\Exp_{\pi^\bfex}[\tau_0^\bfen]$ play symmetric roles in Assumption~(\ref{ass:D}), we only prove that $\Exp_{\pi^\bfen}[\tau_0^\bfex] < +\infty$. To proceed, let us fix $q_0$ and $r>0$ such that $\mathrm{B}(q_0,r) \subset \R^d \setminus \bar{\metasp}$ and let $\varphi \in C^\infty_\mathrm{pol}(\R^d \times \R^d)$ be a nonnegative function such that 
  \begin{equation*}
      \varphi(q,p) = \begin{cases}
        1 & \text{if $|q-q_0| \leq r/2$,}\\
        0 & \text{if $|q-q_0| \geq r$.}
      \end{cases}
  \end{equation*}
  Then, by Assumption~(\ref{ass:A2}), $\mu(\varphi)>0$, and
  \begin{align*}
      \Exp_{\pi^\bfen}[\tau_0^\bfex] &= \int_0^{+\infty} \Pr_{\pi^\bfen}(\tau_0^\bfex>t)\dd t\\
      &\leq\int_0^{+\infty} \Pr_{\pi^\bfen}\left(\frac{1}{t}\int_0^t \varphi(q_s,p_s)\dd s = 0\right)\dd t\\
      &\leq \int_0^{+\infty} \Pr_{\pi^\bfen}\left(\frac{1}{t}\int_0^t \tilde{\varphi}(q_s,p_s)\dd s \geq \mu(\varphi)\right)\dd t\\
      &\leq 1 + \int_1^{+\infty} \Pr_{\pi^\bfen}\left(\frac{1}{t}\int_0^t \tilde{\varphi}(q_s,p_s)\dd s \geq \mu(\varphi)\right)\dd t,
  \end{align*}
  where $\tilde{\varphi} := \mu(\varphi) - \varphi \in C^\infty_\mathrm{pol}(\R^d \times \R^d)$ is such that $\mu(\tilde{\varphi})=0$.
  
  Let us fix $\alpha>2$. By Markov's inequality, for any $t \geq 1$,
  \begin{equation*}
      \Pr_{\pi^\bfen}\left(\frac{1}{t}\int_0^t \tilde{\varphi}(q_s,p_s)\dd s \geq \mu(\varphi)\right) \leq \frac{1}{(t\mu(\varphi))^\alpha}\Exp_{\pi^\bfen}\left[\left|\int_0^t \tilde{\varphi}(q_s,p_s)\dd s\right|^\alpha\right].
  \end{equation*}
  Let $\Phi \in C^\infty_\mathrm{pol}(\R^d \times \R^d)$ be given by Assumption~(\ref{ass:D2}) as the solution to the Poisson equation $\mathcal{L}\Phi = \tilde{\varphi}$. By Itô's formula, for any $t \geq 0$,
  \begin{equation*}
      \int_0^t \tilde{\varphi}(q_s,p_s)\dd s = \Phi(q_t,p_t) - \Phi(q_0,p_0) - \sqrt{2\gamma\beta^{-1}} \int_0^t \nabla_p \Phi(q_s,p_s) \cdot \dd W_s,
  \end{equation*}
  and then by Jensen's inequality,
  \begin{equation*}
      \left|\int_0^t \tilde{\varphi}(q_s,p_s)\dd s\right|^\alpha \leq 3^{\alpha-1} \left(|\Phi(q_t,p_t)|^\alpha + |\Phi(q_0,p_0)|^\alpha + \left|\sqrt{2\gamma\beta^{-1}} \int_0^t \nabla_p \Phi(q_s,p_s) \cdot \dd W_s\right|^\alpha\right).
  \end{equation*}
  Since $\Phi$ has polynomial growth, by Assumption~(\ref{ass:D1}) the expectation under $\Pr_{\pi^\bfen}$ of the first two terms in the right-hand side above is bounded uniformly in $t$. Besides, since $\nabla_p \Phi$ has polynomial growth, by the Burkholder--Davis--Gundy inequality and Assumption~(\ref{ass:D1}) again, the expectation under $\Pr_{\pi^\bfen}$ of the third term is of order $t^{\alpha/2}$. We therefore conclude that
  \begin{equation*}
      \Exp_{\pi^\bfen}[\tau_0^\bfex] \leq 1 + C\int_1^{+\infty} \frac{1+t^{\alpha/2}}{t^\alpha}\dd t
  \end{equation*}
  for some constant $C \geq 0$, which completes the proof since $\alpha>2$.
\end{proof}

In the conservative case, we now provide explicit conditions on $F=-\nabla V$ ensuring that Assumptions~(\ref{ass:D1}--\ref{ass:D2}) hold. As a preliminary, we recall that in the conservative case, Assumption~(\ref{ass:C}) is equivalent to the statement that
\begin{equation*}
  \int_\Sigma \ee^{-\beta V(q)}\dd \sigma_\Sigma(q) < +\infty.
\end{equation*}
The condition~\eqref{eq:integVeps} below is of a similar nature.

\begin{lem}[Sufficient conditions for Assumptions~(\ref{ass:D1}--\ref{ass:D2}) in the conservative case]\label{lem:Kopec}
  Under the assumptions of Theorem~\ref{theo:main}, assume that $F=-\nabla V$, with a potential $V$ satisfying the conditions~B-1, B-2, B-3 and~B-4 from~\cite{Kop15}, and that $\Sigma$ is such that
  \begin{equation}\label{eq:integVeps}
    \forall \ell \geq 0, \qquad \int_\Sigma |q|^\ell \ee^{-\beta V(q)}\dd \sigma_\Sigma(q) < +\infty.
  \end{equation}
  Then Assumptions~(\ref{ass:D1}--\ref{ass:D2}) hold.
\end{lem}
\begin{proof}
  First, by~\cite[Lemma~2.5]{Kop15}, to prove Assumption~(\ref{ass:D1}), it suffices to show that $q_0$ and $p_0$ have finite moments of all orders under $\pi^\bfen$ and $\pi^\bfex$. Using the explicit formula for the densities of $\pi^\bfen$ and $\pi^\bfex$ in the conservative case, it is readily seen that this condition holds unconditionally for $p_0$, while it follows from~\eqref{eq:integVeps} for $q_0$.
  
  Second, Assumption~(\ref{ass:D2}) is a straightforward consequence of~\cite[Lemma~2.12]{Kop15}. Note that in the latter reference, it is proved for the Poisson equation $\mathcal{L}^\dagger\Phi=\tilde{\varphi}$, where $\mathcal{L}^\dagger$ is the formal adjoint of $\mathcal{L}$ in $L^2(\R^d \times \R^d, \mu)$, so the result follows for $\mathcal{L}$ by reverting the sign of $p$.
\end{proof}

\begin{rk}
  The conditions~B-1, B-2, B-3 and~B-4 from~\cite{Kop15} are in particular satisfied if $V$ is the sum of a polynomial $V^\sharp$ with at least quadratic growth at infinity, and a function $V^\flat \in C^\infty_\mathrm{pol}(\R^d)$ which is bounded and has a bounded gradient.
\end{rk}

\begin{rk}
  The condition~\eqref{eq:integVeps} holds in particular if, in addition to Assumption~(\ref{ass:B1}), the set $\metasp$ is bounded, because then $\Sigma$ has finite surface measure. Otherwise, this condition may depend on the geometry of $\Sigma$; for instance, it may fail if for some $r>0$, $\sigma_\Sigma(\mathrm{B}(q,r))$ is not bounded as a function of $q \in \R^d$.
\end{rk}

To complete the discussion, we note that in~\cite{Kop15}, both the
uniform moment estimate (Lemma~2.5) and the existence of a solution to
the Poisson equation (Lemma~2.12) follow from the use of a suitable
Lyapunov functional, which relies on the fact that the force is conservative ($F=-\nabla V$). In the general non-conservative situation, similar Lyapunov functionals, merely based on growth and drift conditions on $F$, were constructed for example in~\cite[Assumption~4]{SacLeiDan17}. Therefore, under these conditions, one may expect the arguments of~\cite{Kop15} to be adapted in order to check Assumptions~(\ref{ass:D1}--\ref{ass:D2}), and thereby extend the statement of Lemma~\ref{lem:Kopec} to non-conservative cases.


\subsection*{Acknowledgements}
The authors would like to thank Gabriel Stoltz for fruitful discussions in particular about the proof of
Proposition~\ref{prop:Etau} and the connection with the work of Marie Kopec, and Arnaud Guyader for his careful reading and useful comments.

This work benefitted from the
support of the European Research Council under the
European Union's Horizon 2020 research and innovation programme (grant agreement No 810367), project
EMC2. This work also benefitted from the
support of the projects ANR EFI (ANR-17-CE40-0030) and ANR QuAMProcs (ANR-19-CE40-0010)
from the French National Research Agency. Part of this project
was carried out as TL was a visiting professor at Imperial College of London (ICL), with a
visiting professorship grant from the Leverhulme Trust. The Department of Mathematics
at ICL and the Leverhulme Trust are warmly thanked for their support. MR is supported by Samsung Science and Technology Foundation and was supported by the Région Ile-de-France through a PhD fellowship of the Domaine d'Intérêt Majeur (DIM) Math Innov.




\end{document}